\documentclass[12pt,reqno]{amsart}
\usepackage{amsmath,amsthm,amssymb,amsrefs,enumerate,mathtools,verbatim,stmaryrd,xcolor,microtype,graphicx,aliascnt,mathrsfs}
\usepackage[bookmarksdepth=2,linktoc=page,colorlinks,linkcolor={red!70!black},citecolor={red!80!black},urlcolor={blue!80!black},pdfpagemode=UseOutlines,pdfstartview={fitH}]{hyperref}   
\usepackage[T1]{fontenc}
\usepackage[utf8]{inputenc}
\usepackage[english]{babel} 
\usepackage[top=3.5cm,bottom=3.55cm,left=3.5cm,right=3.5cm]{geometry}
\usepackage{tikz}
\usepackage{tikz-cd}
\usepackage[mathcal]{euscript}                               
\usepackage{cleveref}
\usepackage{xr}
\usepackage{mathptmx}                                        
\usepackage[colorinlistoftodos,bordercolor=orange,backgroundcolor=orange!20,linecolor=orange,textsize=footnotesize,textwidth=20mm]{todonotes} \makeatletter \providecommand \@dotsep{5} \def\listtodoname{List of Todos} \def\listoftodos{\@starttoc{tdo}\listtodoname} \makeatother 

\allowdisplaybreaks 

\usepackage{etoolbox}
\makeatletter
\patchcmd{\@startsection}{\@afterindenttrue}{\@afterindentfalse}{}{}             
\patchcmd{\part}{\bfseries}{\bfseries\LARGE}{}{}
\patchcmd{\section}{\scshape}{\bfseries}{}{}\renewcommand{\@secnumfont}{\bfseries} 
\patchcmd{\@settitle}{\uppercasenonmath\@title}{\large}{}{}
\patchcmd{\@setauthors}{\MakeUppercase}{}{}{}
\makeatother

\usepackage{fancyhdr}

\addto\extrasenglish{}  
\theoremstyle{plain}
\newtheorem{thm}{Theorem}[section] 
\newaliascnt{lemma}{thm}\newtheorem{lemma}[lemma]{Lemma}\aliascntresetthe{lemma}
\newaliascnt{cor}{thm}\newtheorem{cor}[cor]{Corollary}\aliascntresetthe{cor}
\newaliascnt{prop}{thm}\newtheorem{prop}[prop]{Proposition}\aliascntresetthe{prop}

\newtheorem*{prop*}{Proposition}
\newtheorem*{con*}{Conjecture}
\newtheorem*{thm*}{Theorem}
\newtheorem*{lem*}{Lemma}
\newtheorem*{cor*}{Corollary}
\def\equationautorefname~#1\null{Equation~(#1)\null}

\theoremstyle{definition}
\newaliascnt{df}{thm}\newtheorem{df}[df]{Definition}\aliascntresetthe{df}
\newaliascnt{rem}{thm}\newtheorem{rem}[rem]{Remark}\aliascntresetthe{rem}
\newaliascnt{ex}{thm}\newtheorem{ex}[ex]{Example}\aliascntresetthe{ex}
\newaliascnt{conj}{thm}\newtheorem{conj}[conj]{Conjecture}\aliascntresetthe{conj}
\newaliascnt{problem}{thm}\aliascntresetthe{problem}
\newtheorem{question}[thm]{Question}

\newtheorem*{df*}{Definition}
\newtheorem*{ex*}{Example}
\newtheorem*{rem*}{Remark}

\theoremstyle{remark}

\DeclareRobustCommand{\gobblefour}[5]{}    

\DeclareSymbolFont{sfoperators}{OT1}{bch}{m}{n} \DeclareSymbolFontAlphabet{\mathsf}{sfoperators} \makeatletter\def\operator@font{\mathgroup\symsfoperators}\makeatother 
\DeclareSymbolFont{cmletters}{OML}{cmm}{m}{it}              
\DeclareSymbolFont{cmsymbols}{OMS}{cmsy}{m}{n}
\DeclareSymbolFont{cmlargesymbols}{OMX}{cmex}{m}{n}
\DeclareMathSymbol{\myjmath}{\mathord}{cmletters}{"7C}     \let\jmath\myjmath 
\DeclareMathSymbol{\myamalg}{\mathbin}{cmsymbols}{"71}     
\DeclareMathSymbol{\mycoprod}{\mathop}{cmlargesymbols}{"60}\let\coprod\mycoprod
\DeclareMathSymbol{\myalpha}{\mathord}{cmletters}{"0B}     \let\alpha\myalpha 
\DeclareMathSymbol{\mybeta}{\mathord}{cmletters}{"0C}      \let\beta\mybeta
\DeclareMathSymbol{\mygamma}{\mathord}{cmletters}{"0D}     \let\gamma\mygamma
\DeclareMathSymbol{\mydelta}{\mathord}{cmletters}{"0E}     \let\delta\mydelta
\DeclareMathSymbol{\myepsilon}{\mathord}{cmletters}{"0F}   \let\epsilon\myepsilon
\DeclareMathSymbol{\myzeta}{\mathord}{cmletters}{"10}      \let\zeta\myzeta
\DeclareMathSymbol{\myeta}{\mathord}{cmletters}{"11}       \let\eta\myeta
\DeclareMathSymbol{\mytheta}{\mathord}{cmletters}{"12}     \let\theta\mytheta
\DeclareMathSymbol{\myiota}{\mathord}{cmletters}{"13}      \let\iota\myiota
\DeclareMathSymbol{\mykappa}{\mathord}{cmletters}{"14}     \let\kappa\mykappa
\DeclareMathSymbol{\mylambda}{\mathord}{cmletters}{"15}    \let\lambda\mylambda
\DeclareMathSymbol{\mymu}{\mathord}{cmletters}{"16}        \let\mu\mymu
\DeclareMathSymbol{\mynu}{\mathord}{cmletters}{"17}        \let\nu\mynu
\DeclareMathSymbol{\myxi}{\mathord}{cmletters}{"18}        \let\xi\myxi
\DeclareMathSymbol{\mypi}{\mathord}{cmletters}{"19}        \let\pi\mypi
\DeclareMathSymbol{\myrho}{\mathord}{cmletters}{"1A}       \let\rho\myrho
\DeclareMathSymbol{\mysigma}{\mathord}{cmletters}{"1B}     \let\sigma\mysigma
\DeclareMathSymbol{\mytau}{\mathord}{cmletters}{"1C}       \let\tau\mytau
\DeclareMathSymbol{\myupsilon}{\mathord}{cmletters}{"1D}   \let\upsilon\myupsilon
\DeclareMathSymbol{\myphi}{\mathord}{cmletters}{"1E}       \let\phi\myphi
\DeclareMathSymbol{\mychi}{\mathord}{cmletters}{"1F}       \let\chi\mychi
\DeclareMathSymbol{\mypsi}{\mathord}{cmletters}{"20}       \let\psi\mypsi
\DeclareMathSymbol{\myomega}{\mathord}{cmletters}{"21}     \let\omega\myomega
\DeclareMathSymbol{\myvarepsilon}{\mathord}{cmletters}{"22}\let\varepsilon\myvarepsilon
\DeclareMathSymbol{\myvartheta}{\mathord}{cmletters}{"23}  \let\vartheta\myvartheta
\DeclareMathSymbol{\myvarpi}{\mathord}{cmletters}{"24}     \let\varpi\myvarpi
\DeclareMathSymbol{\myvarrho}{\mathord}{cmletters}{"25}    \let\varrho\myvarrho
\DeclareMathSymbol{\myvarsigma}{\mathord}{cmletters}{"26}  \let\varsigma\myvarsigma
\DeclareMathSymbol{\myvarphi}{\mathord}{cmletters}{"27}    \let\varphi\myvarphi

\DeclareMathOperator{\upN}{N}   
\DeclareMathOperator{\upL}{L}   
\DeclareMathOperator{\upH}{H}
\DeclareMathOperator{\upB}{B} 
\DeclareMathOperator{\upO}{O}   
\DeclareMathOperator{\upS}{S}   
\DeclareMathOperator{\upR}{R}   
\DeclareMathOperator{\ulineL}{\underline{L}}   
\DeclareMathOperator{\ulineS}{\underline{S}}   
\DeclareMathOperator{\ulineDr}{\underline{Dr}}   
\DeclareMathOperator{\Gr}{Gr}
\DeclareMathOperator{\PolyGr}{PolyGr}

\DeclareMathOperator{\ulineGr}{\underline{Gr}}
\DeclareMathOperator{\BP}{BP}
\DeclareMathOperator{\Dr}{Dr}
\DeclareMathOperator{\Hom}{Hom}

\DeclareMathOperator{\supp}{supp}
\DeclareMathOperator{\res}{res}

\newcommand\C{{\mathbb C}}
\newcommand\D{{\mathbb D}}

\newcommand\F{{\mathbb F}}
\newcommand\G{{\mathbb G}}
\renewcommand\H{{\mathbb H}}

\newcommand\N{{\mathbb N}}

\renewcommand\P{{\mathbb P}}
\newcommand\Q{{\mathbb Q}}
\newcommand\R{{\mathbb R}}

\newcommand\T{{\mathbb T}}
\newcommand\U{{\mathbb U}}

\newcommand\Z{{\mathbb Z}}

\newcommand\cA{{\mathcal A}}
\newcommand\cB{{\mathcal B}}

\newcommand\cF{{\mathcal F}}

\newcommand\cM{{\mathcal M}}

\newcommand\cT{{\mathcal T}}

\newcommand\fm{{\mathfrak m}}

\newcommand\fo{{\mathfrak o}}

\newcommand\Funpm{{\F_1^\pm}}

\newcommand\HC{\textnormal{HC}}
\newcommand\im{\textup{im}}

\newcommand\trop{\textup{trop}}
\newcommand\an{\textup{an}}
\newcommand\simp{\textup{simp}}
\renewcommand\max{\textup{max}}
\renewcommand{\min}{\textup{min}}
\renewcommand\geq{\geqslant}
\renewcommand\leq{\leqslant}

\newcommand{\sqfree}{{\scalebox{0.6}{$\boxtimes$}}}

\newcommand{\norm}[1]{|#1|}
\newcommand{\past}[2]{#1\!\sslash\!#2}

\renewcommand{\smallsetminus}{\backslash}

\renewcommand\emptyset\varnothing

\title{Lorentzian polynomials and matroids over triangular hyperfields\\[10pt] \normalsize Part 1: Topological aspects}

\author{Matthew Baker}
\address{\rm Matthew Baker, Georgia Institute of Technology}
\email{mbaker@math.gatech.edu}

\author{June Huh}
\address{\rm June Huh, Princeton University and Korea Institute for Advanced Study}
\email{huh@princeton.edu}

\author{Mario Kummer}
\address{\rm Mario Kummer, Technische Universit\"at Dresden}
\email{mario.kummer@tu-dresden.de}

\author{Oliver Lorscheid}
\address{\rm Oliver Lorscheid, University of Groningen}
\email{o.lorscheid@rug.nl}

\hypersetup{pdftitle={Lorentzian polynomials and matroids over triangular hyperfields. Part 1},pdfauthor={Matthew Baker, June Huh, Mario Kummer and Oliver Lorscheid}}

\begin{document}

\begin{abstract}
Lorentzian polynomials serve as a bridge between continuous and discrete convexity, connecting analysis and combinatorics.
In this article, we study the topology of the space $\P\upL_J$ of Lorentzian polynomials on $J$ modulo $\mathbb{R}_{>0}$, which is nonempty if and only if $J$ is the set of bases of a polymatroid. We prove that $\P\upL_J$ is a manifold with boundary of dimension equal to the Tutte rank of $J$, and more precisely, that it is homeomorphic to a closed Euclidean ball with the Dressian of $J$ removed from its boundary. Furthermore, we show that $\P\upL_J$ is homeomorphic to the thin Schubert cell $\Gr_J(\T_q)$ of $J$ over the triangular hyperfield $\T_q$, introduced by Viro in the context of tropical geometry and Maslov dequantization, for any positive real number $q$. This identification enables us to apply the representation theory of polymatroids developed in the companion paper \cite{BHKL0}, as well as earlier work by the first and fourth authors on foundations of matroids, to give a simple explicit description of $\P\upL_J$ up to homeomorphism in several key cases. 
Our results show that $\P\upL_J$ always admits a compactification homeomorphic to a closed Euclidean ball. They can also be used to answer a question of Brändén in the negative by showing that the closure of $\P\upL_J$ within the space of all polynomials modulo $\mathbb{R}_{>0}$ is not homeomorphic to a closed Euclidean ball in general. In addition, we introduce the Hausdorff compactification of the space of rescaling classes of Lorentzian polynomials and show that the Chow quotient of a complex Grassmannian maps naturally to this compactification. This provides a geometric framework that connects the asymptotic structure of the space of Lorentzian polynomials with classical constructions in algebraic geometry.
\end{abstract}

\maketitle

\begin{small} \tableofcontents \end{small}


\section{Introduction} 
In \cite{Branden-Huh20}, Br\"and\'en and the second author introduced the notion of {\em Lorentzian polynomials} and used it to establish the log-concavity of various sequences of combinatorial origin. 
Lorentzian polynomials serve as a bridge between continuous and discrete convexity, connecting analysis and combinatorics.
Among their many applications to combinatorics, Lorentzian polynomials were used by Br\"and\'en--Huh \cite{Branden-Huh20}, and independently Anari et al. \cite{Anari-et-al-3}, to prove the following conjecture of Mason from \cite{Mason72}:
\begin{quote}
\emph{If $I_k(M)$ is the number of independent sets of size $k$ of a matroid $M$ on $[n]$, then the sequence $I_k(M)/\binom{n}{k}$ is log-concave in $k$.}
\end{quote}

If we assume the support of the polynomials in question to be {\em square-free}\footnote{The \emph{support} of a polynomial $f$ is the set of monomials appearing in $f$ with nonzero coefficients. The support is \emph{square-free} if every monomial in it is square-free. Polynomials with square-free support are also known in the literature as {\em multi-affine} polynomials. Throughout this introduction, we focus on the case of square-free support for expositional simplicity.}, 
then Lorentzian polynomials admit a simple characterization: 
a nonzero homogeneous polynomial $f$ in $n$ variables with nonnegative coefficients and square-free support is {\em Lorentzian} if and only if $\log(f)$ is a concave function on the positive orthant $\R_{>0}^n$. In general, a homogeneous polynomial $f$ with nonnegative coefficients is Lorentzian if and only if the partial derivative $\partial^\alpha f$ is identically zero or log-concave on $\mathbb{R}^n_{>0}$ for all $\alpha \in \mathbb{Z}^n_{\ge 0}$. See \Cref{sec:lorentziandefinitions} for an extended discussion in the general setting.

As with stable polynomials (which Lorentzian polynomials generalize), the support of a square-free Lorentzian polynomial can be identified with the set of bases of a matroid.
We denote by $\P\upL_M$ the projectivization of the space of Lorentzian polynomials whose support is a given matroid $M$.
Our primary goal in this paper is to study the topology of the space $\P\upL_M$, and to relate it to representations of matroids over \emph{tracts} (in the sense of \cite{Baker-Bowler19}).

There are some hints in the literature of a connection between Lorentzian polynomials and matroid representations. For example, part of the initial motivation for Br\"and\'en and Huh to develop the theory of Lorentzian polynomials was the fact that stable polynomials are not as closely connected to combinatorics as one might hope: the Fano matroid $F_7$, for example, is not the support of any stable polynomial, whereas every matroid turns out to be the support of a Lorentzian polynomial. The proof that $F_7$ doesn’t support any stable polynomials uses the fact that it is binary (i.e., representable over the field $\F_2$), but it was not clear how to generalize such observations beyond the binary case. The results of this paper provide a natural framework for thinking about such questions by establishing a homeomorphism between $\P\upL_M$ and the space $\Gr_M(\T_1)$ of representations of $M$ over the \emph{triangular hyperfield} $\T_1$ introduced by Viro \cite{Viro10}. 

The homeomorphism between $\P\upL_M$ and $\Gr_M(\T_1)$ is rather unexpected (it came as a surprise to us, in any case), and it allows us to draw highly non-trivial conclusions about the topology of $\P\upL_M$, such as assertion (2) below, which it's hard to imagine proving directly from the definition of Lorentzian polynomials.

The main results of this paper include the following:

(1) For every matroid $M$ of rank $r$ on $n$ elements, the space $\P\upL_M$ is a manifold with boundary which is homeomorphic to $\Gr_M(\T_1)$. In particular, the dimension of $\P\upL_M$ is equal to the rank of the Tutte group of $M$ (as defined by Dress and Wenzel in \cite{Dress-Wenzel89}).

(2) For certain classes of matroids, we can describe the space $\P\upL_M$ explicitly using the theory of matroid representations over tracts.
For example, if $M$ is ternary (i.e., representable over the field $\F_3$), we use the homeomorphism in (1) and the results of \cite{Baker-Lorscheid20} to completely characterize the possible homeomorphism types for $\P\upL_M$, generalizing the much simpler classification (by Br\"and\'en and Gonz{\'a}lez d'Le{\'o}n \cite{Branden-deLeon10}) for binary matroids.

(3) Br\"and\'en asked if a certain natural compactification of $\P\upL_M$ is homeomorphic to a ball. We show that the answer to this question is \emph{no}. Our explicit counterexample is constructed by computing the Euler characteristic of $\P\upL_M$ for a specific matroid $M$ and showing that it is not equal to 1. The techniques which we develop to compute such Euler characteristics should be of independent interest.

(4) On the other hand, we show that a \emph{different} natural compactification of $\P\upL_M$ \emph{is} homeomorphic to a ball. More precisely, we show that the space $\P\upL_M$ can be compactified to a closed Euclidean ball by adding a point for each ray contained in the local Dressian $\Dr_M$ of $M$.

(5)      Our results also imply the formula
\[\sum_M \chi(\Dr_M)=1, \]
where the sum is over all matroids of rank $d$ on $n$ elements and $\chi(\Dr_M)$ is the Euler characteristic of $\Dr_M$. We do not know how to prove this purely combinatorial statement without the detour into the theory of Lorentzian polynomials.

\subsection{Background: topology of the space of Lorentzian polynomials}

We now give a more detailed overview of the contents of this paper.

Let $\upH(d,n)$ be the vector space of homogeneous polynomials of degree $d$ with real coefficients in variables $x_1,\ldots,x_n$, and let  $\upL(d,n)$ be the set of Lorentzian polynomials in $\upH(d,n)$. We use $\upH(d,n)_\sqfree$ and $\upL(d,n)_\sqfree$ to denote, respectively, the subsets of $\upH(d,n)$ and $\upL(d,n)$ consisting of polynomials with square-free support. For a subset $S$ of $[n]$, we write $x^S$ for the square-free monomial $\prod_{i \in S} x_i$. Thus, $\upH(d,n)_\sqfree$ is the span of $x^S$, where $S$ ranges over all $d$-element subsets of $[n]$.
We denote by $\P\upH(d,n)$ the projectivization of the vector space $\upH(d,n)$, and by $\P X$ the image of $X\smallsetminus 0$ in $\P\upH(d,n)$ for a subset $X$ of $\upH(d,n)$.
A fundamental result on the space of Lorentzian polynomials is the following statement due to Br\"and\'en \cite{Branden21}:

\begin{thm}\label{thm:petterballness}
    The spaces $\P\upL(d,n)$ and $\P\upL(d,n)_\sqfree$ are  homeomorphic to closed Euclidean balls of dimensions $\dim\P\upH(d,n)$ and $\dim\P\upH(d,n)_\sqfree$, respectively.
\end{thm}

Our main goal is to relate spaces of Lorentzian polynomials to various Grassmannians. A first instance of such a connection goes back to \cite{Choe-Oxley-Sokal-Wagner04}:

\begin{ex}
    Let $(p_S)$ be the Pl\"ucker coordinates of a linear subspace of $\C^n$ of dimension $d$, where $S$ ranges over all $d$-element subsets of $[n]$. Then the polynomial $f=\sum_{S} |p_S|^2 \, x^S$ is stable \cite{Choe-Oxley-Sokal-Wagner04}*{Theorem 8.1}, and hence Lorentzian \cite{Branden-Huh20}*{Proposition 2.2}. 
    Therefore, we obtain a continuous map
    \begin{equation}\label{eq:mapfromgrassmannian}
        \Gr(d,n)(\C)\longrightarrow\P\upL(d,n)_\sqfree, \qquad (p_S) \longmapsto \sum_{S} |p_S|^2 \, x^S.
    \end{equation}
    We study this map in detail in \Cref{sec:smalluniform}. For example, when $d=2$ and $n=4$, we show that the map sends $\Gr(2,4)(\R)$ to the boundary of $\P\upL(2,4)_\sqfree$, and the induced map
    \begin{equation*}
        \Gr(2,4)(\R)\longrightarrow \partial\P\upL(2,4)_\sqfree
    \end{equation*}
    is the quotient of $\Gr(2,4)(\R)$ by the action of $\{\pm1\}^4$. By \Cref{thm:petterballness}, the right-hand side is homeomorphic to the four dimensional sphere. 
\end{ex}

\subsection{Thin Schubert cells}

Let $\cB(M)$ be a collection of $d$-element subsets of $[n]$, viewed as a collection of degree $d$ square-free monomials in $x_1,\ldots,x_n$. 
We set
\[
\upH_M\coloneq\big\{f \in \upH(d,n)\mid \text{the support of $f$ is $\cB(M)$}\big\} \ \ \text{and} \ \ \upL_M\coloneq \upL(d,n)\cap\upH_M.
\]
One of the basic results proved in \cite{Branden-Huh20} hints at a surprising link between the continuous and the discrete world: The set $\upL_M$ is nonempty if and only if $\cB(M)$ is the set of bases of a {\em matroid} of rank $d$ on $[n]$.\footnote{A matroid $M$ on $[n]$ is given by a nonempty collection $\cB(M)$ of subsets of $[n]$ satisfying the {\em symmetric exchange property}: For any $B_1,B_2 \in \cB(M)$ and any $b_1 \in B_1 \smallsetminus B_2$, there is $b_2 \in B_2 \smallsetminus B_1$ such that both $B_1 \smallsetminus \{b_1\} \cup \{ b_2 \}$ and $B_2 \smallsetminus \{b_2\} \cup \{ b_1 \}$ belong to $\cB(M)$. Members of the collection $\cB(M)$ are called {\em bases} of the matroid $M$. The symmetric exchange property implies that any two bases have the same cardinality, called the {\em rank} of the matroid. For general introduction to matroids, we refer to \cite{OxleyMatroidTheory}.}

We regard $\P\upL_M$ as the analogue of a \emph{thin Schubert cell} (or \emph{matroid stratum}) in $\P\upL(d,n)_\sqfree$. While a thin Schubert cell in a Grassmannian over a field can be empty, depending on whether or not the matroid is representable over that field, every matroid occurs as the support of some Lorentzian polynomial. In fact, the  generating polynomial $\sum_{B \in \cB(M)} x^B$ is a Lorentzian polynomial if and only if  $\cB(M)$ is the set of bases of a matroid \cite{Branden-Huh20}*{Theorem 3.10}. Our first result is the following:

\begin{thm}[\Cref{cor:lorismanifoldwithboundary}]
    For every matroid $M$, the space $\P\upL_M$ is a manifold with boundary.
\end{thm}

This should be contrasted with the various \emph{universality theorems} for thin Schubert cells in Grassmannians over a field, which roughly state that such thin Schubert cells can exhibit arbitrary singularities \cites{Lafforgue03,Vakil06,Lee-Vakil13}.
In fact, we can explicitly describe $\P\upL_M$ up to homeomorphism. In order to formulate the result, we recall from \cite{Branden-Huh20} the connection between Lorentzian polynomials and the \emph{Dressian}, the piecewise-linear space of all rank $d$ valuated matroids on $[n]$ which plays the role of the Grassmannian in tropical geometry.

\begin{ex}\label{ex:dressianinlorentzian}
    Let $M$ be a matroid of rank $d$ on $[n]$ with the set of bases $\cB(M)$.  A function $\nu\colon \cB(M)\to\R$ is a \emph{valuated matroid} if  for any $B_1,B_2 \in \cB(M)$ and any $b_1 \in B_1 \smallsetminus B_2$, there is $b_2 \in B_2 \smallsetminus B_1$ such that  \[
    \nu(B_1)+\nu(B_2)\leq\nu(B_1 \smallsetminus \{b_1\} \cup \{b_2\} )+\nu(B_2 \smallsetminus \{b_2\} \cup \{b_1\}).
    \]
    The \emph{Dressian} $\Dr_M$ of $M$ is the set of all valuated matroids $\nu\colon \cB(M)\to\R$ modulo addition of constant functions on $\cB(M)$.
    If $\nu\in\Dr_M$, then $f_\nu \coloneq \sum_{B\in\cB(M)} e^{-\nu(B)}\, x^B$ is a Lorentzian polynomial \cite{Branden-Huh20}*{Theorem 3.14}. Therefore, we obtain a continuous injective map
    \begin{equation}\label{eq:expdressian}
        \exp\colon\Dr_M\longrightarrow\P\upL_M, \qquad \nu \longmapsto f_\nu.
    \end{equation}
\end{ex}

Clearly, we have $t\nu\in\Dr_M$ for any $t>0$ and any $\nu\in\Dr_M$. Moreover, if $\nu$ is not constant, then the limit of $f_{t\nu}$ as $t$ goes to infinity is a Lorentzian polynomial with support strictly smaller than $\cB(M)$, and hence it is not an element of $\P\upL_M$. In a sense, this the only obstruction to $\P\upL_M$ being compact:

\begin{thm}[\Cref{thm:Alor}]
    The space $\P\upL_M$ can be compactified to a closed Euclidean ball by adding a point for each ray contained in the Dressian $\Dr_M$.
\end{thm}

We also compute the dimension of $\P\upL_M$ in terms of $M$. An obvious upper bound is the ambient dimension $|\cB(M)|-1$, which is not tight in general. The following constraint is analogous to the corresponding statement of Br\"and\'en on stable polynomials \cite{Branden07}*{Lemma 6.1}:

\begin{lemma}[\Cref{lemma: petterlem}]\label{lem:tinfinityconditions}
Let $f=\sum_B p_B \, x^B$ be a degree $d$ Lorentzian polynomial with square-free support.
Let $S$ be a set of cardinality $d-2$ in $[n]$, and let $i,j,k,l$ be pairwise distinct elements of $[n]$ not in $S$. If one of the three terms
    \begin{equation}\label{eq:tinfinityconditions}
        p_{S\cup\{i,j\}}p_{S\cup\{k,l\}},\,\, p_{S\cup\{i,k\}}p_{S\cup\{j,l\}},\,\, p_{S\cup\{i,l\}}p_{S\cup\{j,k\}}
    \end{equation}
    is zero, then the other two are equal.
\end{lemma}

Note that the three terms in \Cref{eq:tinfinityconditions} correspond to the monomial terms in the \emph{three-term Pl\"ucker relations} for Grassmannians, which shows immediately that \Cref{lem:tinfinityconditions} holds for polynomials in the image of the map from \Cref{eq:mapfromgrassmannian}.
We also note that the conditions from \Cref{lem:tinfinityconditions} only depend on the support of $f$, and they give binomial equations on $\upH_M$ for each matroid $M$. Hence the image of the solution set to these equations in $\upH_M$ under the coefficient-wise ``logarithm of the absolute value'' map
\begin{equation*}
    \log|\cdot|\colon\upH_M\longrightarrow\R^{\cB(M)}  
\end{equation*}
is a linear subspace $V_M$, whose dimension is an upper bound on the dimension of $\upL_M$. Composing the map $\exp$ from \Cref{eq:expdressian} with $\log|\cdot|$, we may consider $\Dr_M$ as a subset of $V_M/\mathbf{1}$, where $\mathbf{1}$ is the all-ones vector. In this way we obtain the following more precise description of $\P\upL_M$:

\begin{thm}[\Cref{thm:Alor}]\label{thm:topologystratum}
    Let $\upB\subseteq V_M/\mathbf{1}$ be the closed unit ball around the origin with respect to a norm on $V_M/\mathbf{1}$. Then $\P\upL_M$ is homeomorphic to the space $\upB\smallsetminus(\partial \upB \cap \Dr_M)$. In particular, the dimension of $\P\upL_M$ is equal to $\dim(V_M)-1$.
\end{thm}
All results outlined here can be straight-forwardly extended to the general case of Lorentzian polynomials with a support that is not necessarily square-free, 
see \Cref{sec:lortopo}.

In the next subsection of this introduction, we explain a conceptual interpretation and common framework for the space $V_M$, its dimension, and the map from \Cref{eq:expdressian}.

\subsection{Grassmannians over triangular hyperfields}
In order to make precise the intuition that the space of Lorentzian polynomials is like a Grassmannian, we recall from \cite{Baker-Bowler19} that there is a natural way to generalize the notion of a field by relaxing the requirement that addition is a binary operation; such generalized fields are called {\em tracts}.

The definition of a tract will be given in \Cref{sec:tractdef}. For now it is enough to know that a tract $F$ consists of a multiplicative abelian group $F^\times$, whose neutral element is denoted by $1$, together with the {\em null set} of ``additive relations'' of the form $a_1 + \cdots + a_k=0$ with $a_i\in F=F^\times\cup\{0\}$ and a unique element $-1\in F$ which satisfies $1+(-1)=0$. 

Given nonnegative integers $d \le n$ and a tract $F$, we define the $F$-{\em Grassmannian} $\Gr(d,n)(F)$ as the set of solutions in $F^{\binom{n}{d}}/F^\times$ to the \emph{Pl{\"u}cker equations}: For any set $S_1$ of cardinality $d-1$ in $[n]$ and any set $S_2$ of cardinality $d+1$ in $[n]$, we have
\[
\sum_{x \in S_2 \smallsetminus S_1} \textrm{sign}(x) p_{S_1 \cup \{x\}}p_{S_2 \smallsetminus \{x\}}=0, 
\]
where $\textrm{sign}(x)=\pm 1$ is determined by the parity of $\#\{x<s \in S_1\}+\#\{x<s \in S_2\}$, see \Cref{subsec:polymatroid}.
When $F=K$ is a field, this recovers the usual Grassmannian $\Gr(d,n)(K)$ over $K$, parametrizing $d$-dimensional subspaces of $K^n$. Similarly, for a matroid $M$ of degree $d$ on $[n]$, one defines the \emph{thin Schubert cell} $\Gr_M(F)$ to be the subset of $\Gr(d,n)(F)$ consisting of points with support $\cB(M)$, which is naturally embedded into $(F^\times)^{\cB(M)}/F^\times$. 
If the tract $F$ carries a topology, this induces a natural topology on $\Gr(d,n)(F)$ and $\Gr_M(F)$, see \cite{BHKL0}*{Section~12} and \cite{Baker-Jin-Lorscheid24}.

As an algebraic framework for Maslov dequantization, Viro defines in \cites{Viro10,Viro11} a family $\T_q$ of tracts indexed by a positive real number $q$, called {\em triangular hyperfields}. See \Cref{subsubsection: Maslov dequantization} for their relation to amoebas and tropicalizations. The multiplicative group $\T_q^\times$ of $\T_q$ is $(\R_{> 0}, \cdot)$ for all $q > 0$, and for $a_1,\ldots,a_k\geq0$, we have  $a_1 + \cdots + a_k = 0$ in $\T_q$ if and only if $a_1^{1/q},\ldots,a_k^{1/q}$ form the side lengths of a (possibly degenerate) convex $k$-gon. So, for example, $a+b+c = 0$ in $\T_1$ if and only if $a,b,c$ form the side lengths of a possibly degenerate Euclidean triangle (hence the name ``triangular hyperfield''), and, up to a positive multiple, $\Gr(2,4)(\T_1)$ is
 the set of nonnegative real numbers $(p_{12},p_{13},p_{14},p_{23},p_{24},p_{34})$  such that
\[
p_{ij}p_{kl}+ p_{ik}p_{jl} \ge p_{il}p_{jk} \ \  \text{for any $i,j,k,l$.}
\]

This family of tracts has two limit objects $\T_0$ (the \emph{tropical hyperfield}) and $\T_\infty$ (the \emph{degenerate triangular hyperfield}), whose ground sets are also $\R_{\geq0}$ with the usual multiplicative structure. The additive relations are given as follows:
\begin{enumerate}[(1)]\itemsep 5pt
\item $a_1 + \cdots + a_k = 0$ holds in $\T_0$ if and only if it holds in all  $\T_q$ with $q>0$.
\item $a_1 + \cdots + a_k = 0$ holds in $\T_\infty$ if and only if it holds in some $\T_q$ with $q>0$.
\end{enumerate}
See \Cref{lemma:zeroinfdesc} for alternative descriptions of the null sets of $\T_0$ and $\T_{\infty}$. 
The topology on $\T_q$ is defined as the Euclidean topology on $\R_{\geq0}$.

In the previous subsection, we have already seen the Grassmannians over $\T_0$ and $\T_\infty$ in disguise. Namely, the coordinate-wise logarithm map
$\R^{\cB(M)}_{>0}/\R_{>0}\to\R^{\cB(M)}/\R\mathbf{1}$ takes $\Gr_M(\T_0)$ to $-\Dr_M$ and $\Gr_M(\T_\infty)$ to the vector space $V_M/\R\mathbf{1}$. The statements of \Cref{ex:dressianinlorentzian} and \Cref{lem:tinfinityconditions} can be rephrased as follows:

\begin{thm}
    For every matroid $M$, there are natural inclusions
    \begin{equation} \label{eq:easysandwichthm}
        \Gr_M(\T_0)\hookrightarrow\P\upL_M\hookrightarrow\Gr_M(\T_\infty).
    \end{equation}
\end{thm}
The following theorem allows us to apply the theory of matroids over tracts to study the topology of $\P\upL_M$. We will explain some consequences in the next subsection.
\begin{thm}[\Cref{cor:allarehomeo}]\label{thm:lortqhomeo}
    For every positive real number $q$ and every matroid $M$, the spaces $\P\upL_M$ and $\Gr_M(\T_q)$ are homeomorphic.
\end{thm}

We have $\P\upL(d,n)_\sqfree  =  \coprod_M \P\upL_M $ and $\Gr(d,n)(\T_q)  =  \coprod_M \Gr_M(\T_q)$,
where the unions are over all matroids of rank $d$ on $[n]$. 
In light of \Cref{thm:lortqhomeo}, we make the following conjecture.

\begin{conj} \label{conj:PLGrTqhomeo}
    The spaces $\P\upL(d,n)_\sqfree$ and $\Gr(d,n)(\T_q)$ are homeomorphic to each other for all $0<q<\infty$.
\end{conj}

By \Cref{thm:petterballness}, this would also imply that $\Gr(d,n)(\T_q)$ is homeomorphic to a ball. \Cref{thm:lortqhomeo} shows that it has the correct Euler characteristic.

The dimension of $\P\upL_M$ has a natural interpretation resulting from the theory of Grassmannians over tracts as well. Namely, for every matroid $M$, there is a tract $T_M$, called the \emph{universal tract of $M$}, that represents the thin Schubert cell $\Gr_M$, considered as a functor from the category of tracts to the category of sets \cite{BHKL0}*{Proposition~5.2}. Its multiplicative group $T_M^\times$, called the \emph{Tutte group of $M$}, is finitely generated. The \emph{Tutte rank of $M$} is the free rank of $T_M^\times$.
\begin{thm}[\Cref{thm:ATq}]
    The dimension of $\P\upL_M$ is equal to the Tutte rank of $M$.
\end{thm}

The results from this subsection, along with \Cref{conj:PLGrTqhomeo}, can be extended to the general case of Lorentzian polynomials with not necessarily square-free support. For this, one has to extend the theory of Grassmannians over tracts from matroids to polymatroids as in \cite{BHKL0}.

\subsection{Orbit spaces}
For every matroid $M$, there is an action of $\R_{>0}^n$ on $\upL_M$ defined by 
\begin{equation} \label{eq:action2}
(c_1,\ldots,c_n) \cdot f(x_1,\ldots,x_n) \coloneq f(c_1 x_1, \ldots, c_n x_n).
\end{equation}
We denote the quotient space by $\ulineL_M$. Similarly, the group $(F^\times)^n$ acts on $\Gr_M(F)$ for every tract $F$, and we denote the quotient space by $\ulineGr_M(F)$. After taking coefficient-wise logarithms, this action of $\R_{>0}^n$ corresponds to addition of elements from a certain linear subspace $W_M$ of $V_M$. We set $\ulineDr_M \coloneq \Dr_M/W_M$ and call this space the \emph{reduced Dressian} over $M$. 
We determine the homeomorphism types of the orbit spaces $\ulineL_M$ and $\ulineGr_M(\T_q)$:

\begin{thm}[Theorems \ref{thm:ATqreduced} and \ref{thm:Alorreduced}]\label{thm:topologyorbitstratum}
    Let $\upB\subseteq V_M/W_M$ be the closed unit ball around the origin with respect to a norm on $V_M/W_M$, and let $q$ be a positive real number. Then $\ulineL_M$ and $\ulineGr_M(\T_q)$ are both homeomorphic to the space $\upB\smallsetminus(\partial \upB \cap \ulineDr_M)$.
\end{thm}

For the relationship between \Cref{thm:Alor} and \Cref{thm:topologyorbitstratum}, see \Cref{prop:product}.

\begin{ex}
    For $M=U_{2,4}$, the reduced Dressian consists of three rays. Hence the space $\ulineL_M$ is homeomorphic to a closed disc with three points from the boundary removed, see \Cref{fig: Dressian and log-Lorentzian for U24}.
    For an arbitrary matroid $M$, the map from \Cref{eq:mapfromgrassmannian} induces a continuous map $\ulineGr_M(\C)\to\ulineL_M$. For $M=U_{2,4}$, this map is a 2-to-1 cover 
    that restricts to a homeomorphism $\ulineGr_M(\R)\to\partial\ulineL_M$, see \Cref{thm:maptoboundary}.
    \begin{figure}[ht]
 \includegraphics[scale=0.5]{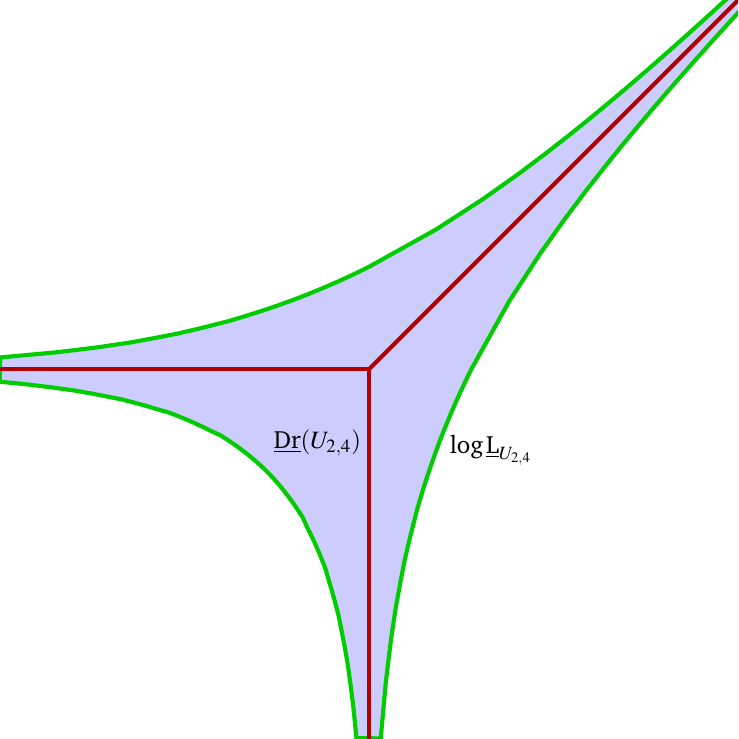}
 \caption{The picture shows the reduced Dressian $\ulineDr_M$ (red) inside the space $\ulineL_M$ (purple) of orbits of Lorentzian polynomials in logarithmic coordinates when $M$ is the uniform matroid $U_{2,4}$. The boundary (green) is the image of $\ulineGr_M(\R)$ under the map from \Cref{eq:mapfromgrassmannian}.}
 \label{fig: Dressian and log-Lorentzian for U24}
\end{figure}
\end{ex}

Using the theory of matroid representations over tracts, we derive several results on the topology of $\ulineGr_M(\T_q)$, and thus of $\ulineL_M$. 

\begin{ex}[\Cref{ex:finitefoundation}]
    If $M$ is a binary matroid, then $\ulineGr_M(\T_q)$ is a singleton for all $q\in[0,\infty]$. This implies that $\ulineL_M$ is also a singleton.
\end{ex}

\begin{ex}[\Cref{cor:ternary}]
    If $M$ is a ternary matroid, then $\ulineGr_M(\T_q)$ is homeomorphic to the product of finitely many half-open intervals and discs with three points removed from the boundary. By \Cref{thm:topologyorbitstratum}, the same is true for $\ulineL_M$.
\end{ex}

\begin{ex}[\Cref{cor:directlimit}]
    If $M$ is any matroid, then $\ulineGr_{M}(\T_q)$ is homeomorphic to the inverse
limit of a finite directed system of topological spaces, each of which is either a disc with three points removed from the boundary or a five-dimensional ball with a copy of the Petersen graph removed from the boundary. By \Cref{thm:topologyorbitstratum}, the same is true for $\ulineL_M$.
\end{ex}

\begin{ex}[\Cref{thm:betsyross}]
    The Betsy Ross matroid $M=B_{11}$ is the rank $3$ matroid on $11$ elements whose point-line arrangement is illustrated in \Cref{fig: Betsy-Ross}.
    There is an explicit homeomorphism $\ulineGr_{M}(\T_\infty)\to\R$ which maps $\ulineGr_{M}(\T_q)$ to the closed interval $[-q,q]$ for all $0\leq q<\infty$. It identifies the space $\ulineL_{M}$ with the closed interval $[-2,2]$, and under this identification the points of $[-2,2]$ which correspond to stable polynomials are precisely the two boundary points $\{-2,2\}$. We note that the space $\ulineGr_{M}(\T_0)$ is a singleton, while $\ulineGr_{M}(\T_\infty)$ is of positive dimension; this answers a question of Brandt--Speyer \cite{Brandt-Speyer22}*{Remark 3.2}.
\end{ex}
\begin{figure}[ht]
 \includegraphics[scale=0.5]{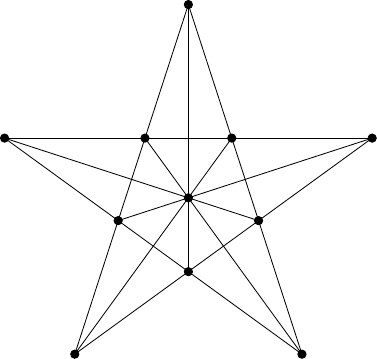}
 \caption{Point-line arrangement of the Betsy Ross matroid}
 \label{fig: Betsy-Ross}
\end{figure}

\subsection{Compactifications}\label{subsec:compactificationsintro} \Cref{thm:topologystratum} implies that the space $\P\upL_M$ admits a compactification by a closed Euclidean ball for every matroid $M$ of rank $d$ on $[n]$. It is interesting to compare this with another natural compactification, namely the closure $\overline{\P\upL}_M$ of $\P\upL_M$ in $\P\upH(d,n)$. Since $\P\upL(d,n)_\sqfree=\overline{\P\upL}_M$ when $M$ is the uniform matroid $U_{d,n}$, \Cref{thm:petterballness} implies that the space $\overline{\P\upL}_M$ is homeomorphic to a closed Euclidean ball, and in particular $\chi(\overline{\P\upL}_M)=1$, in this case. Our results imply that the same holds for another family of matroids. 

\begin{thm}[\Cref{prop:singletonthenball} and \ref{prop:rigideuler}]
    Let $M$ be a matroid $M$ of rank $d$ on $[n]$.
    \begin{enumerate}[(1)]\itemsep 5pt
        \item If $\ulineGr_M(\T_\infty)$ is a singleton (this happens, e.g., for binary matroids), then $\overline{\P\upL}_M$ is homeomorphic to a closed Euclidean ball.
        \item If $\ulineGr_M(\T_0)$ is a singleton, then the Euler characteristic of $\overline{\P\upL}_M$ is equal to one.
    \end{enumerate}
\end{thm}

On the other hand, we find a matroid for which both parts of the theorem fail. This answers \cite{Branden21}*{Questions 5.1, 5.3} in the negative:

\begin{ex}[\Cref{ex:T11counterexample}]\label{ex:ellipticballness}
    The \emph{elliptic matroid} $M=\cT_{11}$ is the matroid of rank 3 on $[11]$ whose non-bases are all 3-element subsets $\{i,j,k\}\subseteq[11]$ such that $i+j+k$ is divisible by $11$. Using \Cref{thm:topologystratum}, we can show that
    the Euler characteristic of $\overline{\P\upL}_{M}$ is equal to $11$. In particular, the space $\overline{\P\upL}_{M}$ is not homeomorphic to a closed Euclidean ball.
\end{ex}

The following example answers another question by Br\"and\'en \cite{Branden21}*{page 7} in the negative:

\begin{ex}[\Cref{ex:betsyrosstable}]
    Let $M=B_{11}$ be the Betsy Ross matroid, and denote by $\upS_M$ the space of stable polynomials with support $M$. Then the Euler characteristic of the space $\overline{\P\upS}_{B_{11}}$ is equal to $17$. In particular, this space is not homeomorphic to a closed Euclidean ball.
\end{ex}

The close connection between spaces of Lorentzian polynomials and Dressians also allows us to deduce new statements about Dressians. For example, the following can be deduced from the corresponding statement about $\P\upL(d,n)_\sqfree$. 

\begin{thm}[\Cref{rem:eulerdressian}]\label{t0spaces}
         The Euler characteristic of $\Gr(d,n)(\T_0)$ is equal to $1$ for any $d \le n$.
\end{thm}

We do not know how to prove this combinatorial statement without the detour to Lorentzian polynomials.

The spaces $\ulineL_M$ can also be compactified in a natural way:

\begin{thm}[\Cref{cor:compactification}]
    Let $H$ be the space of compact subsets of $\overline{\P\upL}_M$ equipped with the topology induced by the Hausdorff metric. The map $\ulineL_M\to H$ that sends an orbit to its closure is a homeomorphism onto its image. In particular, the closure $\HC({\ulineL}_M)$ in $H$ of the image of this map is a compactification of $\ulineL_M$.
\end{thm}

For uniform matroids, we set $\HC({\ulineL}(d,n)_\sqfree)\coloneq \HC({\ulineL}_{U_{d,n}})$. This compactification interacts nicely with the \emph{Chow quotient} ${\Gr}(d,n)(\C)/\!\!/(\C^\times)^n$ of the complex Grassmannian, as introduced in \cite{chowquotientsI}. The Chow quotient is the compactification of $\ulineGr_{U_{d,n}}(\C)$ obtained as the closure of the image of the map from $\ulineGr_{U_{d,n}}(\C)$ to the Chow variety of $\Gr(d,n)(\C)$ that sends a torus orbit to the cycle corresponding to its closure. We will prove in \Cref{thm:chowhausdorffmap} that there is a continuous map
\begin{equation}\label{eq:chowmap}
    {\Gr}(d,n)(\C)/\!\!/(\C^\times)^n\longrightarrow\HC({\ulineL}(d,n)_\sqfree)
\end{equation}
which sends a cycle to the image of its underlying set under the map from \Cref{eq:mapfromgrassmannian}. By the same construction, we obtain a compactification $\HC(\ulineGr_M(\T_q))$ of $\ulineGr_M(\T_q)$.
\begin{ex}[\Cref{ssec:stablecurves}]
    The Chow quotient ${\Gr}(2,n)(\C)/\!\!/(\C^\times)^n$ is isomorphic to the Grothendieck--Knudson moduli space $\overline{\cM}_{0,n}$ of stable rational curves with $n$ marked points \cite{chowquotientsI}*{Chapter IV}. We examine the map
    \begin{equation*}
        \overline{\cM}_{0,n}\longrightarrow\HC({\ulineL}(2,n)_\sqfree)
    \end{equation*}
    from \Cref{eq:chowmap} for small $n$:
    \begin{enumerate}[(1)]\itemsep 5pt
        \item For $n<4$, both the source and target are a point.
        \item The space $\HC({\ulineL}(2,4)_\sqfree)$ is the disc obtained by adding the three missing points to the boundary of $\ulineL_{U_{2,4}}$, see \Cref{fig: Dressian and log-Lorentzian for U24}. The space $\overline{\cM}_{0,4}$ is the complex projective line and the map $\overline{\cM}_{0,4}\to\overline\HC({\ulineL}(2,4)_\sqfree)$ is the quotient by the action of complex conjugation.
        \item The image of the map $\overline{\cM}_{0,5}\to\HC({\ulineL}(2,5)_\sqfree)$ is the closure of the boundary $\partial\ulineL_{U_{2,5}}$ in $\HC({\ulineL}(2,5)_\sqfree)$, which we denote by $\partial\HC({\ulineL}(2,5)_\sqfree)$. Again, the map $\overline{\cM}_{0,5}\to\partial\HC({\ulineL}(2,5)_\sqfree)$ is the quotient by the action of complex conjugation on $\overline{\cM}_{0,5}$.
        \item For $n>5$, the map $\overline{\cM}_{0,n}\to\HC({\ulineL}(2,n)_\sqfree)$ is constant on complex conjugate pairs, but there are fibers that contain more than one such pair.
    \end{enumerate}
\end{ex}

Just as for the Chow quotient of the Grassmannian \cite{chowquotientsI}*{Section 1.2}, the points in $\HC({\ulineL}_M)$ and $\HC(\ulineGr_M(\T_q))$ can be characterized in terms of regular matroid subdivisions of the base polytope of $M$. Namely, every point in our compactification is the union of orbit closures, one for each maximal cell of a regular matroid subdivision, where points in $\ulineL_M$ and $\ulineGr_M(\T_q)$ correspond to the trivial subdivision. However, unlike the case of classical Chow quotients, there are points in our compactification corresponding to every such subdivision, see \Cref{thm:bondarypoints}. We end with a question on the topology of the spaces $\HC({\ulineL}_M)$ and $\HC(\ulineGr_M(\T_q))$.

\begin{question}
    Are the spaces $\HC({\ulineL}_M)$ and $\HC(\ulineGr_M(\T_q))$ homeomorphic to a closed Euclidean ball for every matroid $M$ and every $q>0$?
\end{question}

In \Cref{ssec:examplescompactification}, we provide positive evidence for a few matroids, including the matroid $M=\cT_{11}$ from \Cref{ex:ellipticballness}, for which the compactification $\overline{\P\upL}_M$ is ill-behaved.

\subsection{Acknowledgements} 
 Matt Baker was partially supported by NSF grant DMS-2154224 and a Simons Fellowship in Mathematics (1037306, Baker). June Huh was partially supported by the Oswald Veblen Fund and the Simons Investigator Grant. Mario Kummer was partially supported by DFG grant 502861109. Oliver Lorscheid was partially supported by NSF grant DMS-1926686 and by the Institute for Advanced Study.
 We thank Donggyu Kim, Matt Larson, and Curtis McMullen for helpful comments on a previous version of the manuscript.
 Mario Kummer would like to thank Georg Loho for helpful discussions on topics related to this paper.
 We thank the Korea Institute for Advanced Study in Seoul and the Institute for Advanced Study in Princeton for hosting us during our collaboration.

\section{Elementary properties of star-shaped sets}\label{sec:starshaped}
We fix a finite dimensional $\R$-vector space $V$ equipped with a norm $\|\cdot\|$. 

Recall that a triple $(x_*,X,V)$, where $X\subseteq V$ and $x_*\in X$, is called \emph{star-shaped} if for every $x\in X$ and every $t\in[0,1]$, the point $x_*+t\cdot (x-x_*)$ lies in $X$.
 
\begin{df}
 A triple $(x_*,X,V)$, where $X\subseteq V$ and $x_*\in X$, is called \emph{strongly star-shaped} if $X$ is closed in $V$ and if for every $x\in X$ and every $t\in[0,1)$, the point $x_*+t\cdot (x-x_*)$ lies in the interior of $X$ (with respect to the topology of $V$).
\end{df}

Note that $(x_*,X,V)$ being strongly star-shaped implies in particular that $x_*$ is in the interior of $X$. Our first goal is to show that if $(x_*,X,V)$ is strongly star-shaped, then it is homeomorphic to an open subset of a Euclidean ball.

\begin{lemma}\label{lemma: cont}
    Let $(x_*,X,V)$ be strongly star-shaped.
    The following map is continuous:
    \begin{equation*}
        \psi\colon V\longrightarrow\R_{\geq0},\qquad x\longmapsto\inf\big\{t>0\, \big| \, x_*+ \frac{x-x_*}{t}\in X\big\}.
    \end{equation*}
    Furthermore, we have for all $x\in X$:
    \begin{enumerate}[(1)]\itemsep 5pt
        \item $\psi(x)\leq1$.
        \item If $\psi(x)>0$, then $x_*+ \frac{x-x_*}{\psi(x)}\in X$.
        \item If $\psi(x)=0$, then $x_*+s\cdot (x-x_*)$ is in the interior of $X$ for all $s\geq0$.
    \end{enumerate}
\end{lemma}

\begin{proof}    
    Without loss of generality, we can assume that $x_*=0$ is the origin.
    In order to prove that $\psi$ is continuous, let $(x_i)_{i\in\N}\subseteq V$ be a sequence converging to $x_0\in V$. We first let $t_0>\liminf_{i\to\infty}\psi(x_i)$. Then there is a subsequence $(y_i)_{i\in\N}$ of $(x_i)_{i\in\N}$ and a sequence $(t_i)_{i\in\N}\subseteq\R$ with $t_i> \psi(y_i)$ for all $i\in\N$ which converges to $t_0$. Then the sequence $(\frac{y_i}{t_i})_{i\in\N}$ is in $X$ and converges to $\frac{x_0}{t_0}$, which is therefore also in $X$. Thus $t_0\geq \psi(x_0)$, which implies $\liminf_{i\to\infty}\psi(x_i)\geq \psi(x_0)$.
    Next we prove that $\limsup_{i\to\infty}\psi(x_i)\leq \psi(x_0)$. Let $t_0>\psi(x_0)$. Then $\frac{x_0}{t_0}$ is in the interior of $X$. If $U\subseteq X$ is an open neighbourhood of $\frac{x_0}{t_0}$, then $U'=t_0\cdot U$ is an open neighbourhood of $x_0$ such that $\psi(x)\leq t_0$ for all $x\in U'$. This implies that $\limsup_{i\to\infty}\psi(x_i)\leq \psi(x_0)$.

    The first of the three additional statements is trivial, the second follows because $X$ is closed, and the third follows from the definition of a strongly star-shaped set.
\end{proof}

We denote by $S$ the unit sphere in $V$ around the origin. Let $(x_*,X,V)$ be strongly star-shaped and consider the following set of directions:
\begin{equation*}
    S(x_*,X)=\{x\in S\mid \psi(x_*+x)=0\}.
\end{equation*}

\begin{cor}\label{cor: ballminusdirections}
    Let $\upB\subseteq V$ be the closed unit ball. Then $X$ is homeomorphic to $\upB\smallsetminus S(x_*,X)$.
\end{cor}

\begin{proof}
    After replacing $X$ by $X-x_*$ we can assume without loss of generality that $x_*$ is the origin.
    We consider the map
    \begin{equation*}
        \varphi\colon \upB\smallsetminus S(0,X)\longrightarrow X,\qquad x\longmapsto\frac{x}{1+\psi(x)-\|x\|},
    \end{equation*}
    which is continuous by \Cref{lemma: cont}. Its inverse is given by the continuous map
    \begin{equation*}
        \phi\colon X\longrightarrow \upB\smallsetminus S(0,X),\qquad x\longmapsto\frac{x}{1-\psi(x)+\|x\|}.\qedhere
    \end{equation*}
\end{proof}

\begin{cor}\label{cor: manifoldwithboundary}
    Every strongly star-shaped set is a topological manifold with boundary.
\end{cor}

\begin{proof}
    Let $(x_*,X,V)$ be a strongly star-shaped. The set $S(x_*,X)$ is closed because $\psi$ is continuous. Therefore, \Cref{cor: ballminusdirections} shows that $X$ is homeomorphic to an open subset of the closed unit ball, and thus to a manifold with boundary.
\end{proof}

Finally,  we provide some criteria to show that a set is strongly star-shaped.

\begin{lemma}\label{lemma: starcrit1}
    Assume that $V$ is contained in a finite dimensional $\R$-vector space $U$. Let $U_i$ for $i\in[n]=\{1,\ldots,n\}$ be some further finite dimensional $\R$-vector spaces. Let $\pi_i\colon U\to U_i$ be linear maps and $V_i\subseteq U_i$ linear subspaces such that $V=\cap_{i=1}^n\pi_i^{-1}(V_i)$. Furthermore, let $X_i\subseteq V_i$ for all $i\in[n]$ be subsets and set $X=\cap_{i=1}^n\pi_i^{-1}(X_i)$. Finally, let $x_*\in V$ be such that $(\pi_i(x_*),X_i,V_i)$ is strongly star-shaped for all $i\in[n]$. Then $(x_*,X,V)$ is strongly star-shaped.
\end{lemma}

\begin{proof}
  Let $x\in X$ and $t\in[0,1)$. Because $(\pi_i(x_*),X_i,V_i)$ is strongly star-shaped, it follows that $\pi_i(x_*+t\cdot (x-x_*))$ lies in the interior (relative to $V_i$) of $X_i$ for all $i\in[n]$. Since $V=\cap_{i=1}^n\pi_i^{-1}(V_i)$ and $X=\cap_{i=1}^n\pi_i^{-1}(X_i)$, this implies that $x_*+t\cdot (x-x_*)$ lies in the interior (relative to $V$) of $X$.
\end{proof}

\begin{lemma}\label{lemma: starcrit2}
    Consider the triple $(x_*,X,V)$ consisting of a closed subset $X\subseteq V$ which is star-shaped with respect to a point $x_*$ in the interior of $X$. Assume that for every $x\in X$, there exists $0<t_0<1$ such that for every $t\in(t_0,1)$, the point $x_*+t\cdot (x-x_*)$ lies in the interior of $X$. Then $(x_*,X,V)$ is strongly star-shaped.
\end{lemma}

\begin{proof}
    Let $x\in X$ and $t \in [0,1)$. We need to prove that the point $x_1=x_*+t\cdot (x-x_*)$ lies in the interior of $X$. If $t=0$, then we are done because $x_*$ is in the interior of $X$. If $0<t<1$, then by assumption there exists $c\geq 1$ such that $x_c=x_*+c\cdot t\cdot (x-x_*)$ lies in the interior of $X$. Let $U\subseteq X$ be an open neighbourhood of $x_c$. Then the set
    \begin{equation*}
        \Big\{x_*+\frac{1}{c}(y-x_*)\mid y\in U\Big\}
    \end{equation*}
    is an open neighbourhood of $x_1$ which is contained in $X$ (because $X$ is star-shaped).
\end{proof}

\begin{lemma}\label{lemma: starshapedlineality}
    Let $X\subseteq V$ be a closed subset and $W\subseteq V$ a linear subspace such that $W+X=X$. For every $x_*\in X$ and $w\in W$,  $(x_*,X,V)$ is strongly star-shaped if and only if $(w+x_*,X,V)$ is strongly star-shaped.
\end{lemma}

\begin{proof}
    This follows directly from the definitions.
\end{proof}

\begin{lemma}\label{lemma: starshapedlineality2}
    Let $X\subseteq V$ be a closed subset and $W\subseteq V$ a linear subspace such that $W+X=X$, and consider the projection $\pi\colon V\to V/W$. We choose some norm on $V/W$. Let $x_*\in X$ be such that $(x_*,X,V)$ is strongly star-shaped. Then:
    \begin{enumerate}[(1)]\itemsep 5pt
        \item $(\pi(x_*),\pi(X),V/W)$ is strongly star-shaped.
        \item The set of directions $S(\pi(x_*),\pi(X))$ is equal to 
        \begin{equation*}
            \big\{\pi(x)\in S\subseteq V/W\mid \forall t\geq0: x_*+tx\in X\big\},
        \end{equation*}
        where $S$ is the unit sphere in $V/W$.
    \end{enumerate}
\end{lemma}

\begin{proof}
    It follows from $W+X=X$ that $\pi(X)$ is closed. Let $x\in X$ and $0\leq t<1$. Then $x_*+t(x-x_*)$ is in the interior of $X$. Thus $\pi(x_*)+t(\pi(x)-\pi(x_*))$ is in the interior of $\pi(X)$, since $\pi$ is open. This establishes part (1).

    For part (2), the inclusion $\supset$ is clear. For the other inclusion, let $x\in V$ such that $\pi(x)\in S(\pi(x_*),\pi(X))$. Then for all $t\geq0$, there is $x_t\in X$ such that $\pi(x_*)+t\pi(x)=\pi(x_t)$. This is equivalent to $x_*+tx=x_t+w_t$ for some $w_t\in W=\ker(\pi)$. Because $W+X=X$, this implies the claim.
\end{proof}

\section{The topology of representation spaces over triangular hyperfields}
We first recall some definitions and results on polymatroid representations over tracts.
\subsection{Definition of tracts}\label{sec:tractdef}
A \emph{pointed monoid} is a multiplicatively written commutative monoid $F$ with unit $1$ and a distinguished element $0$ that satisfies $0\cdot a=0$ for all $a\in F$. The \emph{unit group of $F$} is the group
\[
 F^\times \ = \ \{a\in F\mid ab=1\text{ for some }b\in F\}
\]
of all invertible elements in $F$. A \emph{pointed group} is a pointed monoid $F$ such that $F^\times=F-\{0\}$. The \emph{ambient semiring} of a pointed group $F$ is the group semiring
\[
 F^+ \ = \ \N[F^\times].
\]
We denote its elements by $\sum n_a.a$ (where $n_a\in\N$ and $n_a=0$ for all but finitely many $a\in A$), $n_1.a_1+\dotsb+n_r.a_r$, or $\sum a_i$ (where $a$ appears $n_a$ times as a summand). The pointed group $F$ embeds into $F^+$ by sending $0$ to the additively neutral element of $F^+$ and $a\in F^\times$ to $a=1.a\in F^+$. An \emph{ideal of $F^+$} is a subset that contains $0$ and is closed under addition and under multiplication by elements of $F^+$.

\begin{df}
A \emph{tract}\footnote{What we call a tract in this text is, in the language of \cite{Baker-Lorscheid21b}, an \emph{ideal tract} or an \emph{idyll}.} is a pointed group $F$ together with an ideal $N_F$ of $F^+$, called the \emph{null set of $F$}, such that for every $a\in F$ there is a unique $b\in F$ with $a+b\in N_F$. 
A \emph{homomorphism of tracts} is a multiplicative map $f\colon F_1\to F_2$ with $f(0)=0$ and $f(1)=1$ such that the induced map $F_1^+\to F_2^+$ sends every element of $N_{F_1}$ to an element of $N_{F_2}$.
\end{df}
We write $-a$ for the unique element $b$ with $a+b\in N_F$ and call it the \emph{additive inverse of $a$}, and we use $a-b$ for $a+(-b)$. Typically, we denote a tract by $F$ and suppress its null set $N_F$ from the notation. 
\begin{ex}
    The most basic example of a tract is a field $F$, which is a pointed group with respect to multiplication, with null set $N_F$ consisting of all formal sums of elements that sum to zero in $F$.
\end{ex}

\subsection{Triangular hyperfields}\label{sec:triangbasics}
Viro introduces triangular hyperfields in \cites{Viro10,Viro11} as an algebraic framework for Maslov dequantization, which we review in \Cref{subsubsection: Maslov dequantization}. In this text, we consider triangular hyperfields (whose definition we give in \Cref{lemma: definition of triangular hyperfield}) as tracts. 

For $a,b\geq0$, we define
\begin{equation*}
    a\boxplus_q b = \big\{c\geq0\mid |a^{1/q}-b^{1/q}|\leq c^{1/q} \leq a^{1/q}+b^{1/q}\big\}.
\end{equation*}
In other words, $c\in a\boxplus_q b$ if and only if $a^{1/q}$, $b^{1/q}$ and $c^{1/q}$ are the side lengths of a Euclidean triangle.

If $A,B$ are subsets of $\R_{\geq0}$, we define
\begin{equation*}
    A\boxplus_q B = \bigcup_{a\in A,\, b\in B} a\boxplus_q b.
\end{equation*}

It was shown in \cite{Viro10}*{§5} that $\boxplus_q$, together with the usual multiplication on $\R_{\geq0}$, defines a so-called \emph{hyperfield}. We will not define hyperfields here, but only note the following consequence of the definition:
\begin{lemma}\label{lemma: definition of triangular hyperfield}
    The pointed group $\R_{\geq0}$, together with the null set
    \begin{equation*}
        N_{\T_q}=\big\{ 0 \big\} \cup \big\{\sum_{i=1}^n a_i \mid n\in\N_{\geq 2} \mathrm{ \; and \; } 0\in a_1\boxplus_q\cdots\boxplus_q a_n\big\},
    \end{equation*}
    is a tract, called the \emph{triangular hyperfield} $\T_q$.
\end{lemma}
\begin{proof}
    This follows from the fact that $\boxplus_q$, together with the usual multiplication on $\R_{\geq0}$, defines a hyperfield (cf.~\cite{Viro10}*{Section 5}, \cite{Baker-Bowler19}*{Section 2.3}, and \cite{Baker-Lorscheid21b}*{Theorem 2.21}). It can also be verified directly without any difficulties.
\end{proof}

In order to avoid confusion with the usual addition of nonnegative real numbers, we will from now on denote addition in the group semiring $\T_q^+$ by $+_q$. 

\begin{rem}\label{rem:scaling}
    For all $a\geq0$ and $q>0$, we have $a+_q a\in N_{\T_q}$. Furthermore,  we  have $a_1+_q\cdots +_q a_n\in N_{\T_q}$ if and only if $a_1^{1/q}+_1\cdots +_1 a_n^{1/q}\in N_{\T_1}$.
\end{rem}

We will also make use of the following description of elements of length larger than three in the null set of $\T_1$. Recall that a \emph{cyclic $n$-gon} is an $n$-gon whose vertices lie on a circle in the Euclidean plane.

\begin{lemma}\label{lemma: characterization of higher null sums in T1}
 Let $a_1,\dotsc,a_n\in\R_{\geq0}$. Then the following are equivalent:
 \begin{enumerate}[(1)]\itemsep 5pt
  \item\label{triangle1} $a_1+_1\cdots+_1a_n\in N_{\T_1}$;
  \item\label{triangle2} $a_i\leq \sum_{j \neq i} a_j$ for all $i=1,\dotsc,n$;
  \item\label{triangle3} $a_1,\dotsc,a_n$ are the side-lengths (proceeding clockwise from a fixed vertex) of a (possibly degenerate) $n$-gon in the Euclidean plane;
  \item\label{triangle4} $a_1,\dotsc,a_n$ are the side-lengths  (proceeding clockwise from a fixed vertex) of a (possibly degenerate) convex $n$-gon in the Euclidean plane;
  \item\label{triangle5} $a_1,\dotsc,a_n$ are the side-lengths  (proceeding clockwise from a fixed vertex) of a (possibly degenerate) cyclic $n$-gon in the Euclidean plane.
 \end{enumerate}
\end{lemma}

\begin{proof}
    It is clear that (\ref{triangle5}) implies (\ref{triangle4}), which implies (\ref{triangle3}). 
    And, since the shortest distance between two points is a straight line,    (\ref{triangle3}) implies (\ref{triangle2}). 
    Moreover, it is proved in \cite{Penner87}*{Theorem 6.2} (see also \cite{Schreiber93}) that (\ref{triangle5}) is equivalent to (\ref{triangle2}). Therefore (\ref{triangle2}), (\ref{triangle3}), (\ref{triangle4}), and (\ref{triangle5}) are all equivalent.
    (We note that one can show in a much simpler and more direct way that (\ref{triangle2}) implies (\ref{triangle3}), and with a bit of additional work that (\ref{triangle2}) implies (\ref{triangle4}).)
    
    So it suffices to show that (\ref{triangle1}) is equivalent to (\ref{triangle2}). For this, we use the elementary observation (cf.~\cite{Baker-Zhang23}*{Proposition 1.10 and Lemma 2.3}) that if $H$ is a hyperfield with hyperaddition $\boxplus$, we have $0\in a_1\boxplus \cdots\boxplus a_n$ iff there exists $a \in H$ such that $0 \in a_1 \boxplus a_2 \boxplus \cdots \boxplus a_{n-2} \boxplus a$ and $0 \in a_{n-1} \boxplus a_n \boxplus (-a)$.
    Since $-1 = 1$ in $\T_1$, it follows that $a_1+_1\cdots+_1a_n\in N_{\T_1}$ iff there exists $a \in \T_1$ such that $a_1 +_1 a_2 +_1 \cdots +_1 a_{n-2} +_1 a \in N_{\T_1}$ and $a_{n-1} +_1 a_n +_1 a \in N_{\T_1}$. (A tract with this property is called a \emph{fusion tract} in \cite{Baker-Zhang23}.)

    We know that (\ref{triangle1}) is equivalent to (\ref{triangle2}) when $n = 3$, so we may assume that $n \geq 4$. Assume  (\ref{triangle1}) holds. Then there exists $a \in \T_1$ such that $a_1 +_1 a_2 +_1 \cdots +_1 a_{n-2} +_1 a \in N_{\T_1}$ and $a_{n-1} +_1 a_n +_1 a \in N_{\T_1}$. By induction, we may assume that $a_i\leq \sum_{j \neq i, j \leq n-2} a_j + a$ for all $i=1,\dotsc,n-2$, and $a \leq a_{n-1} + a_n$, and therefore (\ref{triangle2}) holds. 
    
    Conversely, assume (\ref{triangle2}) holds. One can give an elementary direct algebraic proof of (\ref{triangle1}), but it is perhaps more enlightening to proceed as follows. Since (\ref{triangle2}) implies (\ref{triangle4}), there is a convex $n$-gon $P$ with side lengths $a_1,\dotsc,a_n$. 
    Then $a_{n-1}$ and $a_n$ are consecutive side lengths in $P$, and we draw a diagonal of $P$ of length $a$ which splits $P$ into a triangle with side lengths $a,a_{n-1}$, and $a_n$ and a convex $(n-1)$-gon $P'$ with side lengths $a_1,\ldots,a_{n-2}$ and $a$. Since these side lengths satisfy (\ref{triangle2}), it follows by induction that both $a_{n-1} +_1 a_n +_1 a$ and $a_1 +_1 a_2 +_1 \cdots +_1 a_{n-2} +_1 a$ belong to $N_{\T_1}$. Thus $a_1+_1\cdots+_1a_n\in N_{\T_1}$ as desired.   
\end{proof}

\begin{rem}
Penner also proves in \cite{Penner87}*{Theorem 6.2} that the cyclic polygon in (\ref{triangle5}) is \emph{unique} in the non-degenerate case where $a_i < \sum_{j \neq i} a_j$ for all $i=1,\dotsc,n$.
\end{rem}

\begin{lemma}\label{lemma:subadditive}
    Let $n\geq2$, $0\leq q<1$, and $x_1,\ldots,x_n>0$. Then
    \begin{equation*}
        (x_1+\cdots+x_n)^q < x_1^q+\cdots + x_n^q.
    \end{equation*}
\end{lemma}

\begin{proof}
    By induction, it suffices to prove the case $n=2$. We have
    \begin{multline*}
        (x_1+x_2)^q \ = \ \frac{x_1}{x_1+x_2}\cdot(x_1+x_2)^q \ + \ \frac{x_2}{x_1+x_2}\cdot(x_1+x_2)^q \\[5pt]
        = \ \Big(\frac{x_1}{x_1+x_2}\Big)^{1-q} \cdot x_1^q \ + \ \Big(\frac{x_2}{x_1+x_2}\Big)^{1-q} \cdot x_2^q \ < \ x_1^q+x_2^q,
    \end{multline*}
    where the last the inequality follows because $\frac{x_1}{x_1+x_2},\frac{x_2}{x_1+x_2}<1$.
\end{proof}

\begin{cor}\label{cor: inclusion of triangular hyperfields}
 Let $0<p\leq q<\infty$. Then  $N_{\T_{p}}\subseteq N_{\T_{q}}$ as subsets of $\N[\R_{>0}]$.
\end{cor}

\begin{proof}
 The claim is clear for $p=q$, so we may assume that $p<q$. Consider $\sum a_i\in N_{\T_p}$, i.e.\ $a_j^{1/p}\leq\sum_{i\neq j} a_i^{1/p}$ for all $j$. After taking $p$-th powers of all $a_i$, we can assume that $p=1<q$, i.e., $a_j\leq\sum a_i$ and $0<1/q<1$. By \Cref{lemma:subadditive}, we have $a_j^{1/q}\leq\sum_{i\neq j} a_i^{1/q}$ for all $j$, which shows that $\sum a_i\in N_{\T_q}$, as claimed.
\end{proof}

We define the \emph{tropical hyperfield} as the tract $\T_0=\R_{\geq0}$ with null set $N_{\T_0}=\bigcap_{q>0}N_{\T_q}$ and the \emph{degenerate triangular hyperfield} as the tract $\T_\infty=\R_{\geq0}$ with null set $N_{\T_\infty}=\bigcup_{q>0}N_{\T_q}$. It follows formally from \Cref{cor: inclusion of triangular hyperfields} that both null sets satisfy the axioms of a tract.

\begin{lemma}\label{lemma:zeroinfdesc}
    Let $a_1,\ldots,a_n\geq0$ and let $q>0$. Then:
    \begin{enumerate}[(1)]\itemsep 5pt
        \item\label{nullset1} $\sum_{i=1}^n a_i\in N_{\T_0}$ if and only if the sum is identically zero or the maximum among $a_1,\dotsc,a_n$ appears at least twice.
        \item\label{nullset2} $\sum_{i=1}^n a_i\in N_{\T_\infty}$ if and only if the sum is identically zero, the maximum among $a_1,\dotsc,a_n$ appears at least twice, or at least three of the $a_i$ are nonzero.
    \end{enumerate}
\end{lemma}

\begin{proof}
 For part \eqref{nullset1}, assume that $a_1>a_2,\ldots,a_n$. Then $\frac{a_2}{a_1},\cdots,\frac{a_n}{a_1}<1$ and there exists $p>0$ small enough such that $\left(\frac{a_i}{a_1}\right)^{1/p}<\frac{1}{n}$ for all $i=2,\ldots,n$. It follows that $a_1^{1/p}>a_2^{1/p}+\cdots+a_n^{1/p}$, contradicting (1) for $q=1$ by \Cref{lemma: characterization of higher null sums in T1}. The other direction of \eqref{nullset1} follows directly from \Cref{lemma: characterization of higher null sums in T1}. 
 
 The ``only if'' direction of part \eqref{nullset2} is clear. For the other direction, it suffices to prove $a_1^{1/p}\leq a_2^{1/p}+\cdots+a_n^{1/p}$ for a suitable $p>0$. We can further assume that $a_1>a_2,\ldots,a_n$. Again, this implies $\frac{a_2}{a_1},\cdots,\frac{a_n}{a_1}<1$ and there exists $p>0$ large enough such that $\left(\frac{a_i}{a_1}\right)^{1/p}\geq\frac{1}{2}$ for all $i=2,\ldots,n$ with $a_i\neq0$. Since, by assumption, this is the case for at least two $i\in\{2,\ldots,n\}$, the claim follows.
\end{proof}

Finally, we study tract homomorphisms into triangular hyperfields. The situation is particularly simple for the degenerate triangular hyperfield $\T_\infty$:

\begin{ex}\label{rem:tinftyreps}
    Let $F$ be a tract and $f\colon F^\times\to\T_\infty^\times=\R_{>0}$ a group homomorphism. We claim that the induced map $F^+\to\T_\infty^+$ sends $N_F$ to $N_{\T_\infty}$. Indeed, by part \eqref{nullset2} of \Cref{lemma:zeroinfdesc}, this is true for all sums of length at least three in the null set of $F$. Thus it only remains to show that $f(a)=f(-a)$ for all $a\in F^\times$. Since $a^2=(-a)^2$ and thus $f(a)^2=f(-a)^2$, this follows from $\R_{>0}$ being torsion-free. Therefore, tract homomorphisms $F\to\T_\infty$ are canonically in bijection with group homomorphisms $F^\times\to\R_{>0}$, cf. \cite{BHKL0}*{Proposition~9.2}.
\end{ex}

Next we study tract homomorphisms $\C\to\T_q$.

\begin{lemma}\label{lemma: morphisms from c to tq}
For $t\geq0$, define $\psi_t\colon\C\to\R_{\geq0}$ by $x\mapsto |x|^t$. The following holds for all $q\geq0$:
\begin{enumerate}[(1)]\itemsep 5pt
    \item If $0\leq t\leq q$, then $\psi_t$ is a homomorphism of tracts $\C\to\T_q$. 
    \item Conversely, if $t>q$, then for all $a_1,a_2,a_3\in\R^\times$ with $a_1+a_2+a_3=0$, we have  $\psi_t(a_1)+_q\psi_t(a_2)+_q\psi_t(a_3)\not\in N_{\T_q}$. 
\end{enumerate}
\end{lemma}

\begin{proof}
    The case $q=0$ is clear, so we may assume that $q>0$.
    It is clear that $\psi_t$ is a group homomorphism. Next, let $a_1,\ldots,a_n\in\C$ be such that $a_1+\cdots+a_n=0$. 
    By the triangle inequality, we have
    \[
        |a_1|\leq |a_2|+\cdots+|a_n|
        \Longrightarrow |a_1|^{t/q}\leq (|a_2|+\cdots+|a_n|)^{t/q}
        \leq |a_2|^{t/q}+\cdots+|a_n|^{t/q},
        \]
    where the last inequality follows from \Cref{lemma:subadditive}. By \Cref{lemma: characterization of higher null sums in T1}, this is exactly what we need to show for part (1). For part (2), assume without loss of generality $a_1>0$ and $a_2,a_3<0$. Assume that $\psi_t(a_1)+_q\psi_t(a_2)+_q\psi_t(a_3)\in N_{\T_q}$. This implies that $|a_1|^{t/q}\leq |a_2|^{t/q}+|a_3|^{t/q}$. This implies, by \Cref{lemma:subadditive}, that $a_1<a_2+a_3$ since $\frac{q}{t}<1$, which contradicts our assumption.
\end{proof}

\subsection{Polymatroid representations}\label{subsec:polymatroid}
There are several cryptomorphic definitions of an (integral) polymatroid \cites{Welsh, Murota03,Edmonds70}. For us, a polymatroid will always be given by its set of bases: A subset $J\subseteq\Delta_n^d=\{\alpha\in\N^n \mid \alpha_1+\dotsb+\alpha_n=d\}$ is the set of bases of a polymatroid if and only it is M-convex in the following sense.
\begin{df}
    A subset $J$ of $\Delta^d_n$ is \emph{M-convex} if $J \neq \emptyset$ and for all $\alpha,\beta\in J$ and every $i\in[n]$ with $\alpha_i<\beta_i$, there exists $j\in[n]$ such that $\alpha_j > \beta_j$ and $J$ contains both $\alpha+e_i-e_j$ and $\beta-e_i+e_j$. The \emph{rank} of the M-convex set $J$ is $r$.
\end{df}

\begin{ex}
    If we identify a subset of $[n]$ with its indicator function, regarded as an element of $\{0,1\}^n$, then the set of bases of a rank $d$ matroid on $[n]$ is an M-convex set in $\Delta_n^d$.
\end{ex}

Recall the following definitions from \cite{BHKL0}. We consider $\N^n$ as a partially ordered set with respect to the partial order $\alpha\leq\beta$ if and only if $\alpha_i\leq\beta_i$ for all $i\in[n]$.

\begin{df}
    Let $J\subseteq\Delta_n^d$ be an M-convex set. We define $\delta^-_J=\inf(J)\in\N^n$ to be the vector whose $i$-th coordinate is $\min\{\alpha_i\mid\alpha\in J\}$. Similarly, we define $\delta^+_J=\sup(J)\in\N^n$ to be the vector whose $i$-th coordinate is $\max\{\alpha_i\mid\alpha\in J\}$.
\end{df}

\begin{df}
    Let $F$ be a tract and $J\subseteq\Delta_n^d$ an M-convex set. 
    A function $\rho\colon\Delta^d_n\to F$ is a \emph{strong $F$-representation }of $J$ if its support is $J$ and if it satisfies the Pl\"ucker relations
\begin{equation}\label{eq:pluecker relations}
 \sum_{k=0}^s (-1)^{k+\epsilon(k)}\cdot \rho(\alpha +e_{i_0}+\dotsb \widehat{e_{i_k}}\dotsb+e_{i_s}) \cdot \rho(\alpha +e_{i_k}+e_{j_2}+\dotsb +e_{j_s})  \in  N_{F}
\end{equation}
for all $2\leq s\leq d$, all $\alpha\in\Delta^{d-s}_n$, all $1\leq i_0\leq \dotsc\leq i_s\leq n$ and all $1\leq j_2\leq\dotsc\leq j_s\leq n$ with 
\[
 \delta_J^- \ \leq \ \alpha \qquad \text{and} \qquad \alpha+e_{i_0}+\cdots+e_{i_s}+e_{j_2}+\cdots+e_{j_s} \ \leq \ \delta_J^+, 
\]
where $\epsilon(k)$ is the number of $k\in\{2,\ldots,s\}$ with $i_k<j_s$.

It is a \emph{weak $F$-representation} of $J$ if its support is $J$ and if it satisfies the $3$-term Pl\"ucker relations
\begin{multline*}
 \rho(\alpha+e_j+e_k) \cdot \rho(\alpha +e_i+e_l) \ - \ \rho(\alpha +e_i+e_k) \cdot \rho(\alpha +e_j+e_l) \\ + \ \rho(\alpha +e_i+e_j) \cdot \rho(\alpha +e_k+e_l)  \in  N_{F}
\end{multline*}
for all $\alpha\in\Delta^{d-2}_n$ and $1\leq i\leq j\leq k\leq l\leq n$  with $\delta_J^-\leq\alpha$ and $\alpha+e_i+e_j+e_k+e_l\leq\delta_J^+$.
\end{df}

\begin{rem}
    If $F$ is a field, then \Cref{eq:pluecker relations} describes the usual Pl\"ucker relations. If $F$ is \emph{idempotent}, i.e., $1+1,\ 1+1+1\in N_F$, then the signs can be ignored. For example the triangular hyperfields are all idempotent.
\end{rem}

\begin{df}[Representation spaces]
Let $F$ be a tract and $J\subseteq\Delta_n^d$ an M-convex set.
The \emph{(strong) representation space of $J$ over $F$} is defined as the set $\upR_J(F)$ of all strong $F$-representations of $J$. The \emph{weak representation space of $J$ over $F$} is the set $\upR_J^{\rm w}(F)$ of all weak $F$-representations of $J$. 
The group $\R_{>0}$ acts diagonally on the (weak and strong) representation space. The \emph{(strong) thin Schubert cell of $J$ over $F$} is $\Gr_J(F)=\upR_J(F)/\R_{>0}$. The \emph{weak thin Schubert cell of $J$ over $F$} is $\Gr^{\rm w}_J(F)=\upR^{\rm w}_J(F)/\R_{>0}$. 
\end{df}

\begin{rem}\label{rem:universal tract}
 Let $J\subseteq\Delta^d_n$ be an M-convex set. There exists a tract $T_J$, called the \emph{universal tract} of $J$, and, for every tract $F$, a bijection
 \[
   \Hom(T_J,  F) \ \stackrel\sim\longrightarrow \ \Gr_J(F)
 \]
   which is functorial in $F$, see \cite{BHKL0}*{Proposition~D}.
\end{rem}

\begin{rem}
    It follows from our considerations in \Cref{sec:triangbasics} that we have, for all $0\leq q_1\leq q_2\leq\infty$, the inclusion $\upR_J(\T_{q_1})\subseteq\upR_J(\T_{q_2})$, and similarly for weak representations and (weak) thin Schubert cells.
\end{rem}

\begin{rem}\label{rem: excellent tracts}
    If $F$ is a field or $F\in\{\T_0,\T_\infty\}$, then we have $\upR_J(F)=\upR_J^{\rm w}(F)$ by \cite{BHKL0}*{Theorem~J}.
    This is not the case for $0<q<\infty$, see \cite{Baker-Bowler19}*{Example 3.37}.
\end{rem}

Representations over $\T_0$ are essentially the same thing as M-convex functions.

\begin{df}
 A function $f\colon\Delta^d_n\to \R\cup\{\infty\}$ with non-trivial support $J=\{\alpha\in\Z^n\mid f(\alpha)\neq\infty\}$ is \emph{M-convex} if it satisfies the following \emph{exchange property}: for $\alpha,\beta\in J$ and $k\in[n]$ with $\alpha_k>\beta_k$, there is an $l\in[n]$ with $\alpha_l<\beta_l$ and
\begin{equation}\label{eq: Murata's exchange relations}
 f(\alpha) \ + \ f(\beta) \ \geq \ f(\alpha-e_k+e_l) \ +  f(\beta+e_k-e_l).
\end{equation}
\end{df}

\begin{lemma}[\cite{BHKL0}*{Proposition~4.11}]\label{lemma:mconvexist0}
    Let $J$ be an M-convex set and $\rho\colon\Delta^d_n\to\T_0=\R_{\geq0}$ a function with support $J$. Then $\rho$ is a (weak or strong) $\T_0$-representation of $J$ if and only if $f=-\log(\rho)$ is M-convex.
\end{lemma}

The set $\upR_J(F)$ is also invariant under rescaling by elements of $(F^\times)^{n+1}$, see \cite{BHKL0}*{Lemma~6.1}. In more detail, given an $F$-representation $\rho\colon\Delta^d_n\to F$ of $J$ and $t=(t_0,\dotsc,t_n)\in (F^\times)^{n+1}$, we define
\begin{equation}\label{eq:rescalingaction}
 (t.\rho)(\alpha) \ = \ \Big(t_0\cdot\prod_{i=1}^n \ t_{i}^{\alpha_i} \Big) \cdot \rho(\alpha).
\end{equation}
The same holds true for weak representations.

\begin{df}[Realization spaces]
    The \emph{realization space of $J$ over $F$} is defined as the quotient set $\ulineGr_J(F)=\upR_J(F)/(F^\times)^{n+1}$. Similarly, we define $\ulineGr^{\rm w}_J(F)=\upR^{\rm w}_J(F)/(F^\times)^{n+1}$.
\end{df}

\begin{rem}
 If the abelian group $F^\times$ is divisible, then every orbit under the action of $(F^\times)^{n+1}$ defined in \Cref{eq:rescalingaction} is also an orbit of $(F^\times)^{n}$ under the action
 \begin{equation*}\label{eq:rescalingaction2}
 (t.\rho)(\alpha) \ = \ \Big(\prod_{i=1}^n \ t_{i}^{\alpha_i} \Big) \cdot \rho(\alpha)
\end{equation*}
for $t=(t_1,\dotsc,t_n)\in (F^\times)^{n}$.
\end{rem}
The representation spaces $\upR_J(\T_q)$ and $\upR_J^{\rm w}(\T_q)$ can naturally be considered as subsets of $\R_{>0}^J$, and as such they inherit the Euclidean topology of $\R_{>0}^J$. We equip the (weak) Schubert cells as well as the (weak) realization spaces with the quotient topology. The goal of this section is to prove that these spaces are topological manifolds with boundary by employing the results from \Cref{sec:starshaped}.

Let $J\subseteq\Delta_n^d$ be an M-convex set. The coordinate-wise logarithm map
\begin{equation*}
    \log\colon\R_{>0}^J\to\R^J
\end{equation*}
is a homeomorphism. By the description in part \eqref{nullset2} of \Cref{lemma:zeroinfdesc}, it maps the set $\upR_J(\T_\infty)=\upR_J^{\rm w}(\T_\infty)$ (see \Cref{rem: excellent tracts} for this equality) to a linear subspace $V_J$ of $\R^J$.

\begin{prop}\label{prop:repstar}
    Let $0<q<\infty$, and let $T\subseteq V_J$ be either $\log(\upR_J(\T_q))$ or $\log(\upR_J^{\rm w}(\T_q))$. Then $(0,T,V_J)$ is strongly star-shaped.
\end{prop}

\begin{proof}
    We first prove the case of strong representations.
    It is clear from the definitions that $\upR_J(\T_q)\subseteq\R^J_{>0}$, and hence $\log(\upR_J(\T_q))$ is closed. Let $\rho\in\upR_J(\T_q)$ and $0\leq t<1$. We have to show that 
    \begin{equation*}
    \rho_t\colon J\longrightarrow \R_{>0},\qquad \alpha\longmapsto\rho(\alpha)^t
    \end{equation*}
    is in the interior of $\upR_J(\T_q)$ relative to $\upR_J(\T_\infty)$. Let $s\leq r$, $\delta_J^-\leq\alpha\in\Delta^{r-s}_n$, $i_0,\ldots,i_s\in[n]$, and $j_2,\ldots,j_s\in[n]$. If in the corresponding Pl\"ucker relation at most two terms are nonzero, it is satisfied, as a relation over $\T_q$, by all elements of $\upR_J(\T_\infty)$ by \Cref{lemma:zeroinfdesc}. 
    
    Assume otherwise. We have, for all $0\leq l\leq s$,
    \begin{align*}
        &\rho_t(\alpha +e_{i_0}+\dotsb \widehat{e_{i_l}}\dotsb+e_{i_s})^{1/q} \cdot \rho_t(\alpha +e_{i_l}+e_{j_2}+\dotsb +e_{j_s})^{1/q}\\
        =&\rho(\alpha +e_{i_0}+\dotsb \widehat{e_{i_l}}\dotsb+e_{i_s})^{t/q} \cdot \rho(\alpha +e_{i_l}+e_{j_2}+\dotsb +e_{j_s})^{t/q}\\\leq & \left(\sum_{k=0,k\neq l}^s  \rho(\alpha +e_{i_0}+\dotsb \widehat{e_{i_k}}\dotsb+e_{i_s}) \cdot \rho(\alpha +e_{i_k}+e_{j_2}+\dotsb +e_{j_s})\right)^{t/q}\\
        <&\sum_{k=0,k\neq l}^s  \rho(\alpha +e_{i_0}+\dotsb \widehat{e_{i_k}}\dotsb+e_{i_s})^{t/q} \cdot \rho(\alpha +e_{i_k}+e_{j_2}+\dotsb +e_{j_s})^{t/q}\\
        =&\sum_{k=0,k\neq l}^s  \rho_t(\alpha +e_{i_0}+\dotsb \widehat{e_{i_k}}\dotsb+e_{i_s})^{1/q} \cdot \rho_t(\alpha +e_{i_k}+e_{j_2}+\dotsb +e_{j_s})^{1/q},
    \end{align*}
    where the first inequality holds because $\rho\in\upR_J(\T_q)$ and the second inequality follows from \Cref{lemma:subadditive}. Thus, every $\rho'\in\upR_J(\T_\infty)$ in the proximity of $\rho_t$ satisfies all Pl\"ucker relations over $\T_q$. This shows the claim for strong representations. The proof for weak representations is identical after restricting to the case $s=2$.
\end{proof}

The coordinate-wise logarithm map $\log\colon\R_{>0}^J\to\R^J$ induces a homeomorphism
\begin{equation*}
    \log\colon\R_{>0}^J/\R_{>0}\longrightarrow \R^J/\R\mathbf{1},
\end{equation*}
where $\mathbf{1}$ is the all-ones vector. The dimension $\tau(J)$ of $V_J/\R\mathbf{1}$ is called the \emph{Tutte rank}. By \cite{BHKL0}*{Section~10.2} this is the rank of the multiplicative group of the universal tract of $J$.
The image of $\Gr_J(\T_0)$ under this map is called the \emph{(local) Dressian} $\Dr_J$ of $J$. The orbit of the constant-one map $J\to\R_{>0}$ under the action from \Cref{eq:rescalingaction} is mapped by the coordinate-wise logarithm map $\log\colon\R_{>0}^J\to\R^J$ to a linear subspace $W_J$ of $V_J$, which is called the \emph{lineality space of $\Dr_J$}. We obtain an induced map
\begin{equation*}
    \log\colon\R_{>0}^J/\R^n_{>0}\longrightarrow\R^J/W_J
\end{equation*}
which satisfies $\log(\ulineGr_J(\T_q))=\log(\upR_J(\T_q)) / W_J$ for all $0\leq q\leq\infty$. We write $\ulineDr_J=\log(\ulineGr_J(\T_0))$, and call this set the \emph{reduced Dressian} of $J$.

\begin{thm}\label{thm:ATq}
    Let $\upB\subseteq V_J/\R\mathbf{1}$ be the unit ball with respect to some norm and let $X=\upB\smallsetminus(\Dr_J\cap\partial \upB)$. For $0<q<\infty$, the spaces $\Gr_J(\T_q)$ and $\Gr^{\rm w}_J(\T_q)$ are both homeomorphic to $X$. In particular, these spaces are topological manifolds with boundary which can be compactified to a closed ball  of dimension  $\tau(J)$.
\end{thm}

\begin{proof}
    By \Cref{prop:repstar} and part (1) of \Cref{lemma: starshapedlineality2},  both $(0,\log(\Gr_J(\T_q)),V_J/\R\mathbf{1})$ and $(0,\log(\Gr^{\rm w}_J(\T_q)),V_J/\R\mathbf{1})$ are strongly star-shaped. It follows from the definition of $\T_0$ and part (2) of \Cref{lemma: starshapedlineality2} that $\Dr_J\cap\partial \upB$ is equal to $S(0,\log(\Gr_J(\T_q)))$ (resp. $S(0,\log(\Gr^{\rm w}_J(\T_q)))$). Thus the claim follows from \Cref{cor: ballminusdirections}.
\end{proof}

We have a similar result for the spaces $\ulineGr_J(\T_q)$ and $\ulineGr^{\rm w}_J(\T_q)$. The proof is the same as for \Cref{thm:ATq}.

\begin{thm}\label{thm:ATqreduced}
    Let $\upB\subseteq V_J/W_J$ be the unit ball with respect to some norm and let $X=\upB\smallsetminus(\ulineDr_J\cap\partial \upB)$. For $0<q<\infty$, the spaces $\ulineGr_J(\T_q)$ and $\ulineGr^{\rm w}_J(\T_q)$ are both homeomorphic to $X$. In particular, these spaces are topological manifolds with boundary which can be compactified to a closed Euclidean ball.
\end{thm}

\subsection{Foundations}\label{rem:foundation}
 Let $J\subseteq\Delta^d_n$ be an M-convex set. There exists a tract $F_J$, called the \emph{foundation} of $J$, such that for every tract $F$ there is a bijection
 \[
   \Hom(F_J,  F) \ \stackrel\sim\longrightarrow \ \ulineGr^{\rm w}_J(F)
 \]
   which is functorial in $F$; cf.~\cite{BHKL0}*{Section~6}. We call the image of the identity map $F_J\to F_J$ in $\ulineGr^{\rm w}_J(F_J)$ under this bijection the \emph{universal weak rescaling class} of $J$.
   \begin{ex}\label{ex:finitefoundation}
    The foundation $F_J$ of $J$ is \emph{finite} (in the sense that $F_J^\times$ is finite) if and only if $\ulineGr_J(\T_\infty)$ is a singleton. Indeed, the space $\ulineGr_J(\T_\infty)=\ulineGr^{\rm w}_J(\T_\infty)$ corresponds bijectively to the set of tract homomorphisms $F_J\to\T_\infty$, which by \Cref{rem:tinftyreps} corresponds bijectively to the set of group homomorphisms $F_J^\times\to\R_{>0}$. Since $F_J^\times$ is a finitely generated abelian group, there is a non-trivial group homomorphism $F_J^\times\to\R_{>0}$ if and only if $|F_J^\times|=\infty$. 
\end{ex}  

   The theory of foundations, as developed in \cites{Baker-Lorscheid20,Baker-Lorscheid-Zhang24}, gives us an explicit way of calculating $\ulineGr^w_J(\T_q)$, and hence (by \Cref{thm:ATqreduced}) $\ulineGr_J(\T_q)$.
   
 For $0< q<\infty$ and every tract $F$, we define on $\Hom(F,\T_q)$ the  compact-open topology, where we consider $F$ with the discrete topology and $\T_q=\R_{\geq0}$ with the Euclidean topology. We recall from  \cite{BHKL0}*{Section~12} the following facts:
 \begin{prop}
    Let $J$ be an M-convex set and let $F_J$ be its foundation. Then the bijection $\Hom(F_J,  \T_q) \ \stackrel\sim\longrightarrow \ \ulineGr^{\rm w}_J(\T_q)$ is a homeomorphism.
 \end{prop}
 \begin{prop}\label{prop:limitscolimits}
    If a tract $F$ is the colimit of a finite diagram of tracts $F_i$, $i\in I$, then the topological space $\Hom(F,\T_q)$ is the limit of the corresponding diagram of topological spaces $\Hom(F_i,\T_q)$, $i\in I$.
 \end{prop}

 Recall from \cite{BHKL0}*{Proposition~2.29} that every M-convex set $J$ can be decomposed (as a Cartesian product) into indecomposable M-convex sets, and that such a decomposition is unique up to permuting the factors. In particular, the number $c(J)$ of indecomposable components of $J$ is well-defined.
 
 \begin{prop}\label{prop:product}
    Let  $J\subseteq\Delta^d_n$ be an M-convex set with $c(J)$ indecomposable components and $s=n-c(J)$ (see \cite{BHKL0}*{Section~2.5.1}). Then there is a (non-canonical) homeomorphism \[\Gr^w_J(\T_q)\simeq\ulineGr_J^w(\T_q)\times\R^{s}.\]
 \end{prop}
 We draw some consequences for the spaces $\ulineGr^{\rm w}_{J}(\T_q)$ from this. Note that by  \Cref{thm:ATq}, \Cref{thm:ATqreduced}, and \Cref{prop:product}, we immediately get corresponding statements for the spaces $\ulineGr_{J}(\T_q)$, $\Gr^w_{J}(\T_q)$, and $\Gr_{J}(\T_q)$:
 
 \begin{cor}\label{cor:combinatorialequivalent}
    If two M-convex sets $J_1$ and $J_2$ have the same foundation, then $\ulineGr^{\rm w}_{J_1}(\T_q)$ and $\ulineGr^{\rm w}_{J_2}(\T_q)$ are homeomorphic. In particular, this happens in the following cases:
    \begin{enumerate}[(1)]\itemsep 5pt
    \item $J_1$ and $J_2$ are combinatorially equivalent in the sense of \cite{BHKL0}*{Definition~2.23}, e.g., if $J_1$ and $J_2$ are matroids that are dual to each other.
    \item $J_1$ is a matroid and $J_2$ is obtained from $J_1$ by a segment-cosegment exchange in the sense of \cite{Baker-Lorscheid-Walsh-Zhang}.
    \end{enumerate}
 \end{cor}
 
 \begin{proof}
    In \cite{BHKL0}*{Corollary~7.5} and \cite{Baker-Lorscheid-Walsh-Zhang}*{Theorem D}, it was shown that the foundations agree in these cases.
 \end{proof}
 
 \begin{thm}\label{cor:ternary}
    Let $M$ be a matroid that does not have a $U_{2,5}$ or $U_{3,5}$ minor. (This happens, for example, whenever $M$ is a ternary matroid.)
    Then $\ulineGr^{\rm w}_{M}(\T_q)$ is homeomorphic to the 
product of finitely many half-open intervals and discs with three points removed from the boundary. 
 \end{thm}
 
\begin{proof}
    We consider the five tracts $\mathbb{F}_2,\mathbb{F}_3,\mathbb{D},\mathbb{H},\mathbb{U}$ defined as follows. Each of these tracts is a \emph{partial field}. More precisely, the unit group is the unit group of a certain ring $R$ and the null set is the ideal generated by all formal sums of three elements $a_1,a_2,a_3\in R^\times$ that sum to zero in $R$. The rings in question are $\F_2, \F_3, \Z[\frac{1}{2}], \Z[\zeta_6]$, and $\Z[x,\frac{1}{x},\frac{1}{1-x}]$, where $\zeta_6\in\C$ is a primitive root of unity.
    It was shown in \cite{Baker-Lorscheid20}*{Theorem B} that the foundation of $M$ is the coproduct of finitely many tracts $F_1,\ldots,F_r$ from $\{\mathbb{F}_2,\mathbb{F}_3,\mathbb{D},\mathbb{H},\mathbb{U}\}$. For each of these five tracts, we determine the topological type of $\Hom(F,\T_q)$. 
    
    The unit group of each of the rings $\F_2, \F_3$, and $\Z[\zeta_6]$ is finite. Thus, by \Cref{ex:finitefoundation}, the space $\Hom(F,\T_q)$ is a singleton for $F\in\{\F_2,\F_3,\H\}$. The tract $\U$ is the foundation of the matroid $U_{2,4}$ \cite{Baker-Lorscheid-Zhang24}*{Section 3.2}, which by \Cref{thm:ATqreduced} shows that $\Hom(\U,\T_q)=\ulineGr_{U_{2,4}}(\T_q)$ is a disc with three points removed from the boundary. Finally, since $\Z[\frac{1}{2}]^\times$ is generated by $2$ and $-1$, and since $\R_{>0}$ is torsion-free, for every $x>0$ there is a unique group homomorphism $\varphi_x\colon\Z[\frac{1}{2}]^\times\to\R_{>0}$  with $\varphi_x(2)=x$. The null set of $\D$ is the ideal generated by $2+(-1)+(-1)$. Hence $\varphi_x\in\Hom(\D,\T_q)$ if and only if $x^{1/q}\leq2$. This identifies $\Hom(\D,\T_q)$ with the half-open interval $(0,2^q]$. The claim now follows from \Cref{prop:limitscolimits}.
 \end{proof}
 
  \begin{thm}\label{cor:directlimit}
    Let $M$ be a matroid. Then $\ulineGr^{\rm w}_{M}(\T_q)$ is homeomorphic to the inverse
limit of a finite  diagram of topological spaces, each of which is homeomorphic to a disc with three points removed from the boundary or a five-dimensional ball with a copy of the Petersen graph (\Cref{fig: Peterson graph}) removed from the boundary.
 \end{thm}
 
\begin{figure}[t]
 \includegraphics{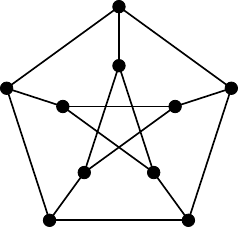}
 \caption{A $2$-dimensional projection of the Petersen graph}
 \label{fig: Peterson graph}
\end{figure}

 \begin{proof}
    It follows from \Cref{prop:limitscolimits} and \cite{Baker-Lorscheid-Zhang24}*{Theorem B} that $\ulineGr^{\rm w}_{J}(\T_q)$ is homeomorphic to the direct limit of a finite directed system of topological spaces $\ulineGr^w_J(\T_q)$, where $J$ is one of the  matroids $U_{2,4}, U_{2,5}, C_5, U_{2,4}\oplus U_{1,2},F_7$ or their duals. Here $C_5$ is a series extension of $U_{2,4}$ and $F_7$ is the Fano matroid. By \cite{Baker-Lorscheid-Zhang24}*{Remark 5.2}, the spaces $\ulineGr_J^w(\T_q)$ are all homeomorphic to each other for $J\in\{U_{2,4}, C_5, U_{2,4}\oplus U_{1,2}\}$, and thus homeomorphic to a disc with three points from the boundary removed by \Cref{thm:ATqreduced}. By \Cref{thm:ATqreduced} and \cite{Maclagan-Sturmfels15}*{Example 4.3.2} (together with the fact that the tropical Grassmannian equals the Dressian for $U_{2,n}$, see p.184 of \emph{loc.~cit.}), 
    the space $\ulineGr_{U_{2,5}}^w(\T_q)$ is homeomorphic to  a five-dimensional ball with a copy of the Petersen graph removed from the boundary. Finally, since $F_7$ is binary, the space $\ulineGr_{F_7}^w(\T_q)$ is a singleton, cf.~\Cref{ex:finitefoundation}.
 \end{proof}
 
 \begin{rem}
    The results in \cite{Baker-Lorscheid-Zhang24} give an explicit description of the directed system appearing in \Cref{cor:directlimit} in terms of certain embedded minors of the matroid $M$.
 \end{rem}

\subsection{Maslov dequantization}
\label{subsubsection: Maslov dequantization}
The realization of triangular hyperfields as tracts allows for the following reformulation of Maslov dequantization (cf.~Viro's papers \cites{Viro10,Viro11}). In a nutshell, the idea behind Maslov dequantization is to deform the usual addition $(a,b)\mapsto a+b$ of nonnegative real numbers via the rule
\[
 (a,b)\mapsto \big(a^{1/q} + b^{1/q} \big)^q,
\]
which in the limit as $q$ goes to zero becomes the tropical sum:
\[
 \ \lim\limits_{q\to0}\big(a^{1/q} + b^{1/q} \big)^q \ = \ \max\{a,b\}.
\]

In the language of tracts, Maslov dequantization is reflected in the fact that the tropical hyperfield $\T_0$ is the limit as $q \to 0$ of the triangular hyperfields $\T_q$ for $q>0$, in the sense that $N_{\T_0}=\bigcap_{q>0} N_{\T_q}$. This limiting process extends to $\T_q$-rational points of varieties, leading to a novel framework for amoebas and their ``tropical'' limit, as explained in the following.

Let $X$ be a complex variety embedded into the torus $(\C^\times)^n$. The \emph{$q$-amoeba of $X$} is the image $\cA_q(X)$ of $X$ under the map $(\C^\times)^n\to\R^n$ given by
\[
 (a_1,\dotsc,a_n) \ \longmapsto \ (\log|a_1|^{q},\dotsc,\log|a_n|^{q}). 
\]
The \emph{tropicalization of $X$} is the Hausdorff limit 
\[
 X^\trop \ \coloneq \ \cA_0(X) \ \coloneq \ \lim_{q\to 0} \ \cA_q(X)
\]
as a subset of $\R^n$.

These spaces are controlled by a certain tract $F=\past{\C(x_1,\dotsc,x_n)}{\pi^{-1}(I)}$ associated with $X\subseteq(\C^\times)^n$, which we define in the following. Let $R=\C[x_1^{\pm1},\dotsc,x_n^{\pm1}]$ be the ring of Laurent series over $\C$ and let $I\subseteq R$ be the vanishing ideal of $X$. As a pointed monoid, $F$ is the submonoid $\{\overline{ax^n}\mid a\in\C,\, n\in\Z^n\}$ of $R/I$, where $\overline p$ is the class of $p\in R$ in $R/I$. The inclusion $F\to R$ extends linearly to a semiring homomorphism $\pi\colon F^+\to R$, where $F^+=\N[\overline{ax^n}\mid a\in\C^\times,\, n\in\Z^n]$ is the ambient semiring of $F$. The null set of $F$ is the pullback $N_F=\pi^{-1}(I)$ of the vanishing ideal of $X$.

For the purposes of the following construction, we define a complex tract to be a tract $T$ together with a tract morphism $\alpha_T\colon \C\to T$. For example, $\T_q$ is a complex tract (for $q\geq0$) with respect to the map $z\mapsto|z|^q$, which is a tract morphism $\alpha_{\T_q}\colon\C\to\T_q$ since whenever $\sum a_i\in N_\C$ (i.e., $\sum a_i=0$ in $\C)$, the $\norm{a_i}$ form the side lengths of a (possibly degenerate) $n$-gon in the Euclidean plane, and thus $\sum\norm{a_i}^q\in N_{\T_q}$ by \Cref{lemma: characterization of higher null sums in T1}. A $\C$-linear map between complex tracts $S$ and $T$ is a tract morphism $f\colon S\to T$ such that $\alpha_T=f\circ\alpha_S$. We denote by $\Hom_\C(S,T)$ the set of $\C$-linear maps from $S$ to $T$. The natural inclusion $\alpha_F\colon\C\to \past{\C(x_1,\dotsc,x_n)}{\pi^{-1}(I)}=F$ provides $F$ with the structure of a complex tract. 

Since $F^\times$ is generated by $\overline{x_1},\dotsc,\overline{x_n}$, the evaluation map
 \[
  \begin{array}{cccl}
  \iota_{F,T}\colon & \Hom_\C(F,T) & \longrightarrow & (T^\times)^n \\[5pt]
              & {}[f\colon F\to T] & \longmapsto     & \big(f(x_1),\dotsc,f(x_n)\big)
  \end{array}             
 \]
is an injection for every complex tract $T$. Also, by construction, the image of the map $\iota_{F,\C}\colon\Hom_\C(F,\C)\to(\C^\times)^n$ is equal to $X$.

\begin{prop}\label{prop: amoebas and tropicalization as subset of rational point sets}
 Let $X\subseteq(\C^\times)^n$ be a subvariety. Let $\cA_q(X)\subseteq\R^n$ be its $q$-amoeba (for $q>0$) and let $\cA_0(X)=X^\trop\subseteq\R^n$ be its tropicalization. Let $F=\past{\C(x_1,\dotsc,x_n)}{\pi^{-1}(I)}$ be as above. Then $\cA_q(X)$ is contained in the image of the embedding
 \[
  \begin{array}{cccl}
  \log(\iota_{X,\T_q})\colon & \Hom_\C(F,\T_q) & \longrightarrow & \R^n \\[5pt]
              & {}[f\colon F\to\T_q] & \longmapsto     & \big(\log f(x_1),\dotsc,\log f(x_n)\big)
  \end{array}             
 \]
 for all $q\geq0$, and $X^\trop = \im\big(\log(\iota_{X,\T_0})\big)$ as subsets of $\R^n$.
\end{prop}

\begin{proof} 
 The inclusions $\iota_{F,T}$ are evidently functorial in $T$, and thus $\C\to\T_q$ induces the commutative diagram:
  \[
  \begin{tikzcd}[column sep=60pt]
   X \ = \ \Hom(F,\C) \ar[r,hook,"\iota_{F,\C}"] \ar[d]   & (\C^\times)^n \ar[d] \\
   \qquad \Hom(F,\T_q) \ar[r,hook,"\iota_{F,\T_q}"] & (\T_q^\times)^n \ar[r,"\log"',"\sim"''] & \R^n
  \end{tikzcd}
 \]
 Thus $\cA_q$, the image of $X\to\R^n$, is contained in the image of $\Hom(F,\T_q) \to\R^n$, as claimed.
 
 We turn to the claim that $X^\trop=\im\big(\log(\iota_{X,\T_0})\big)$. By \cite{Mikhalkin04}*{Corollary 6.4}, the Hausdorff limit $X^\trop=\lim_{q\to0}\cA_q(X)$ is equal to the bend locus of $I^\trop$ in $\R^n$, where $I^\trop$ is the tropicalization of $I$ with respect to the trivial absolute value $\norm\cdot_0:\C\to\R_{\geq0}$. By Kapranov's theorem (\cite{Maclagan-Sturmfels15}*{Theorem 3.1.3}) and a result of Payne \cite{Payne09}*{Proposition 2.2}, the bend locus of $I^\trop$ is equal to the image $\widetilde X^\trop$ of the Berkovich space $X^\an\subseteq\G_m^{n,\an}$ under the Payne map $\G_m^{n,\an}\to \R_{>0}^n$. As explained in \cite{Baker-Jin-Lorscheid24}*{Example 3.8} and \cite{Lorscheid22}*{Theorem 3.5}, the Kajiwara--Payne tropicalization $\widetilde X^\trop$ is equal to $\Hom_\C(F,\T_0)$, which concludes the proof.
\end{proof}

\begin{rem}
 In contrast to tropicalizations, the amoebas $\cA_q(X)$ (for $q>0$) are typically lower-dimensional subspaces of $\log\Hom_\C(F,\T_q)$. For example, we will see in \Cref{sec:smalluniform} that the image of the thin Schubert cell $X=\Gr_M(\C)$ for $M=U_{2,5}$ under the coordinate-wise squared absolute value map $\norm{\cdot}^2\colon X\to\R_{>0}^{10}/\R_{>0}$ is exactly the boundary $\partial\P \upL_M$ of the projective space $\P\upL_M$ of Lorentzian polynomials with support $M$ (\Cref{lemma:imageofsquaremap}.(3)). We will see in \Cref{subsec:conclusiontopo} that $\partial\P \upL_M$, and hence the amoebas $\cA_q(X)$ for $q>0$, all have dimension eight. 
 
On the other hand, $\log\Hom_\C(F,\T_q)$ is always star-shaped (with respect to the origin of $\R^{10}/\R$), which follows from \Cref{lemma:subadditive} by the same arguments as in the proof of \Cref{prop:repstar}. Thus $\log\Hom_\C(F,\T_2)$ contains the minimal star-shaped set containing $\cA_2(X)$. In the example $X=\Gr_{U_{2,5}}(\C)$, this shows that $\log\Hom_\C(F,\T_2)$ has nonempty interior in $\R^{10}/\R$, and thus has dimension nine.
\end{rem}

\begin{rem}[Difference between the tropicalized Grassmannian and the Dressian]
 The thin Schubert cell $X=\Gr_J(\C)$ embeds into the torus $(\C^\times)^J/\C^\times$ by means of its non-vanishing Pl\"ucker coordinates indexed by $J$. Let $F=\past{\C(x_\alpha\mid\alpha\in J)}{\pi^{-1}(I)}$ be as in \Cref{prop: amoebas and tropicalization as subset of rational point sets}. Then the tropicalization $\Gr_J(\C)^\trop$ is equal to the subset $\log\Hom_\C(F,\T_0)$ of $\R^J/\R$.
 
 As explained in \Cref{rem:universal tract}, we have $\Gr_J(\C)=\Hom_\C(F,\C)=\Hom(T_J,\C)$ for the universal tract $T_J$, which is, roughly speaking, the tract generated over $\Funpm$ by variables $x_\alpha$ for $\alpha\in J$ and whose null set is generated by the Pl\"ucker relations. The space $\log\Hom(T_J,\T_0)$ equals the Dressian $\Dr_J$ of $J$ as subsets of $\R^J$ (cf.\ the passage before \Cref{thm:ATq}). Under these identifications, the inclusion $\Gr_J(\C)^\trop\to\Dr_J$ corresponds to the map $\Hom_\C(F,\T_0)\to\Hom(T_J,\T_0)$ given by pre-composition with $T_J\to F$.
\end{rem}

\section{The topology of spaces of Lorentzian polynomials}\label{sec:lortopo}
\subsection{Definitions}\label{sec:lorentziandefinitions}
We denote by $\upH(d,n)$ the space of homogeneous polynomials of degree $d$ in $x_1,\ldots,x_n$ with nonnegative real coefficients.
Let $\mathring{\upL}^2_n$ be the space of homogeneous quadratic polynomials $f \in \upH(2,n)$ with strictly positive coefficients whose Hessian ${\mathcal H}_f$ has {\em Lorentzian signature}, that is, ${\mathcal H}_f$ has one positive eigenvalue and $n-1$ negative eigenvalues, counting multiplicities.
A {\em strictly Lorentzian polynomial} of degree $d$ in $x_1,\ldots,x_n$ is a  polynomial $f \in \upH(d,n)$ with strictly positive coefficients such that $\partial^\alpha f \in \mathring{\upL}^2_n$ for all $\alpha\in\Delta^{d-2}_n$. A {\em Lorentzian polynomial} is a limit of strictly Lorentzian polynomials. It was shown in \cite{Branden-Huh20} that a nonzero polynomial $f\in \upH(d,n)$ with nonnegative coefficients is Lorentzian if and only if it satisfies one of the following equivalent conditions:
\begin{enumerate}[(1)]\itemsep 5pt
\item The support of $f$ is an M-convex set and the Hessian of $\partial^\alpha f \in \upH(2,n)$ has at most one positive eigenvalue for all $\alpha\in\Delta^{d-2}_n$.
\item The partial derivative $\partial^\alpha f$ is identically zero or log-concave on $\mathbb{R}^n_{>0}$ for all $\alpha \in \mathbb{Z}^n_{\ge 0}$.
\end{enumerate}
For an M-convex set $J$, we denote by $\upL_J$ the set of Lorentzian polynomials with support $J$. 

\begin{ex}\label{ex:stable}
    A polynomial $f \in \R[x_1,\ldots,x_n]$ is called {\em (real) stable} if $f=0$ or $f(w) \neq 0$ for all $w=(w_1,\ldots,w_n) \in {\mathbb H}^n$, where ${\mathbb H} \coloneq \{ z \in \C \; | \; {\rm Im}(z)>0 \}$ is the complex upper half-plane. Homogeneous stable polynomials with nonnegative coefficients are Lorentzian by \cite{Branden-Huh20}*{Proposition 2.2}. Moreover, in degree two the notions of stable and Lorentzian coincide by \cite{Branden-Huh20}*{Lemma 2.5}. We denote the set of stable polynomials with nonnegative coefficients and support $J$ by $\upS_J$.
\end{ex}

The multiplicative group $\R_{>0}$ acts on $\upL_J$ by scalar multiplication, and we denote the quotient space by $\P\upL_J$. Similarly, the multiplicative group $\R^n_{>0}$ acts on $\upL_J$ by scaling each variable, and we denote by $\ulineL_J$ the quotient space. The same notation applies to spaces of stable polynomials.
If 
\[
f = \sum_{\alpha \in \Delta^d_n} c_\alpha \frac{x^\alpha}{\alpha !} \in \upH(d,n), 
\]
we define $\rho_f \colon \Delta^d_n \to \R_{\geq 0}$ by $\rho_f(\alpha)=c_\alpha$.

The space of Lorentzian polynomials with support $J$ can be related to representations of $J$ over $\T_0$ and $\T_\infty$ as follows.

\begin{lemma}\label{lemma: petterlem}
   If $f$ is a Lorentzian polynomial with support $J$, then $\rho_f\in\upR^{\rm w}_J(\T_\infty)$.
\end{lemma}

\begin{proof}
Let $J\subseteq\Delta^d_n$ be an M-convex set, let $f\in \upL_J$, and let $\rho=\rho_f$. Let $\alpha\in\Delta_n^{d-2}$ and choose $i,j,k,l\in[n]$ such that one of the three terms
    \begin{equation*}
        \rho(\alpha+e_i+e_j)\rho(\alpha+e_k+e_l),\rho(\alpha+e_i+e_k)\rho(\alpha+e_j+e_l),\rho(\alpha+e_i+e_l)\rho(\alpha+e_j+e_k)
    \end{equation*}
    is zero. Without loss of generality, we can assume that 
    \begin{equation*}
     \rho(\alpha+e_i+e_j)=0.   
    \end{equation*}
    To show that $\rho_f\in\upR^{\rm w}_J(\T_\infty)$, we need to prove that the other two terms are equal to one another. 
    If $i=j$, this is clear, and so we may assume that $i \neq j$.
    The multi-affine part of $\partial^\alpha f$ is Lorentzian, cf.~\cite{Branden-Huh20}*{Corollary 3.5}. Because Lorentzian polynomials of degree two are real stable, we can apply \cite{Branden07}*{Lemma 6.1} and the claim follows.
\end{proof}

\begin{rem}\label{rem: Lorentzian polynomials satisfy all degenerate Plucker relations}
 As a consequence of \Cref{lemma: petterlem} and the equality $\upR^{\rm w}_J(\T_\infty)=\upR_J(\T_\infty)$ (see \Cref{rem: excellent tracts}), if $f=\sum_{\alpha \in \Delta^d_n} c_\alpha \frac{x^\alpha}{\alpha !}\in\upL_J$ is Lorentzian then $\rho_f$ belongs to  $\upR_J(\T_\infty)$, which means that $f$ satisfies \emph{all} degenerate Pl\"ucker relations.
 More precisely,
 \[
  c_{\beta+e_{i_0}} \ \cdot \ c_{\gamma+e_{i_1}} \ = \ c_{\beta+e_{i_1}} \ \cdot \ c_{\gamma+e_{i_0}} 
 \]
 for all $\beta,\gamma\in\Delta^{d-1}_n$ such that
$c_{\beta+e_{i_0}+e_{i_1}-e_{k}} \cdot c_{\gamma+e_{k}} = 0$
whenever $k\in[n]$ and $\beta_k>\gamma_k$. 

 By \cite{Branden07}*{Lemma 6.1}, this fact is known for stable polynomials when $\beta-e_k=\gamma-e_l$ for some $k,l\in[n]$, i.e., stable polynomials satisfy the degenerate $3$-term Pl\"ucker relations. That Lorentzian polynomials, or even stable polynomials, satisfy all degenerate Pl\"ucker relations appears to be a new observation.
 
 An example of a degenerate Pl\"ucker relation with more than $3$ terms is 
 \[
  c_{123} \cdot c_{456} \ + \ c_{124} \cdot c_{356} \ + \ \underbrace{c_{134} \cdot c_{256}}_{=0} \ + \ \underbrace{c_{234} \cdot c_{156}}_{=0} \quad \in \quad N_{\T_\infty}
 \]
 for the matroid $J=\{\alpha\in \Delta^3_6 \mid \alpha_i\leq1\text{ for all $i\in[n]$ and }e_3+e_4\not\leq\alpha\}$, which is a parallel extension of $U_{3,5}$ with parallel elements $3$ and $4$. Consequently, we have $c_{123} \cdot c_{456} = c_{124} \cdot c_{356}$ for any Lorentzian polynomial $f=\sum_{\alpha \in \Delta^d_n} c_\alpha \frac{x^\alpha}{\alpha !}$ with support $J$.
\end{rem}

For $f=\sum_{\alpha \in \Delta^d_n} c_\alpha \frac{x^\alpha}{\alpha !}$ and $p>0$, we define
$R_p(f)=\sum_{\alpha \in \Delta^d_n} c^p_\alpha \frac{x^\alpha}{\alpha !}$.

\begin{thm}[\cite{Branden-Huh20}]\label{thm:314325}
     Let $f$ be a Lorentzian polynomial with support $J$. Then:
     \begin{enumerate}[(1)]\itemsep 5pt
        \item For every $0\leq p\leq 1$, the polynomial $R_p(f)$ is Lorentzian.
        \item $\rho_f$ is a $\T_0$-representation of $J$ if and only if the polynomial $R_p(f)$ is Lorentzian for all $p\geq0$ .
     \end{enumerate}
\end{thm}

\begin{proof}
    Part (1) is \cite{Branden-Huh20}*{Proposition 3.25}. Part (2) follows from \cite{Branden-Huh20}*{Theorem 3.14}, together with \Cref{lemma:mconvexist0}.
\end{proof}

\subsection{Simplification of Lorentzians in degree two}
We first recall the following property of M-convex sets $J\subseteq\Delta_n^2$.
    
\begin{lemma}\label{lemma: 2supp}
    Let $J\subseteq\Delta_n^2$ be an M-convex set. Let $V\subseteq[n]$ be the set of all $i\in[n]$ such that $e_i+e_j\in J$ for some $j\in [n]$, and let $E$ be the set of all two-element subsets $\{i,j\}\subseteq V$ such that $e_i+e_j\in J$. Finally, let $D$ the set of all $i\in J$ such that $2e_i\in J$. Then:
    \begin{enumerate}[(1)]\itemsep 5pt
        \item We have $e_i+e_j\in J$ for all $i\in D$ and $j\in V$.
        \item The graph $G=(V,E)$ is a complete multipartite graph where $V$ is partitioned into the equivalence classes of the equivalence relation $\sim$ generated by $i\sim j$ if $e_i+e_j\notin J$.
    \end{enumerate}
\end{lemma}

\begin{proof}
    Let $i\in V$ such that $2e_i\in J$ and let $j\in V$. Because $j\in V$, there exists $k\in V$ such that $e_j+e_k\in J$. If $k=i$, then we are done. Otherwise, it follows from M-convexity that $e_i+e_j\in J$. This proves part (1). For part (2), we note that $E$ is the set of bases of a rank $2$ matroid on the ground set $V$, which is a parallel extension of a uniform matroid $U_{2,r}$. The claim then follows from \cite{Choe-Oxley-Sokal-Wagner04}*{Corollary 5.4}, noting that the equivalence class of $i\in V$ is the flat (or parallel class) of $i$.
\end{proof}

\begin{df}\label{df: jbar}
    Let $J\subseteq\Delta_n^2$ be an M-convex set. We use the notation from \Cref{lemma: 2supp}. By part (2)  of \Cref{lemma: 2supp}, there is a partition 
        $V=\coprod_{i=1}^r V_i$ 
    such that $\{i,j\}$ is an edge of $G$ if and only if $i$ and $j$ are contained in two different $V_l$. For $i\in V$, we let $p(i)\in[r]$ be such that $i\in V_{p(i)}$. We then define
    \begin{equation*}
        J^{\rm simp}=\{2e_{p(i)}\mid i\in D\}\cup \{e_i+e_j\mid i,j\in[r]\textrm{ and }i\neq j\}\subseteq\Delta_r^2.
    \end{equation*}
\end{df}

\begin{thm}[Simplification]\label{thm: simplification}
    Let $J\subseteq\Delta_n^2$ be an M-convex set and $J^\simp$ as above.
    Let $f\in \upH(d,n)$ with $\supp(f)=J$ such that $\rho_f\in\upR_J(\T_\infty)$. Then:
    \begin{enumerate}[(1)]\itemsep 5pt
        \item     There exists a unique polynomial $g\in\upH(d,r)$ with $\supp(g) = J^{\rm simp}$ and unique $\lambda_1,\ldots,\lambda_n>0$ such that:
    \begin{enumerate}[(1)]\itemsep 5pt
        \item For all $j\in[r]$ we have $\sum_{i\in V_j}\lambda_i=1$.
            \item As polynomials in $H(d,n)$, we have
            \begin{equation*}
        f(x_1,\dotsc,x_n) \ = \ g\Big(\sum_{i\in V_1} \lambda_i x_i,\ldots,\sum_{i\in V_r} \lambda_i x_i\Big).
    \end{equation*}
    \end{enumerate}
    \item $f$ is Lorentzian if and only if $g$ is Lorentzian.
    \end{enumerate}
\end{thm}

\begin{proof}
    Let $k,l\in[r]$ be two different indices, and let $f_{kl}$ the polynomial obtained from $f$ by setting all variables to zero that are not in $V_k\cup V_l$. We claim that there are $a_{kk},a_{ll}\geq0$, $a_{kl}>0$, and $\mu^{k,l}_i>0$, $i\in V_k\cup V_l$, such that
    \begin{equation*}
     \sum_{i\in V_k}\mu^{k,l}_i=\sum_{i\in V_l}\mu^{k,l}_i=1   
    \end{equation*}
    and
    \begin{equation}\label{eq: smallprod}
        f_{kl}=a_{kk}\cdot\sum_{i\in D\cap V_k}x_{i}^2+a_{ll}\cdot\sum_{i\in D\cap V_l}x_{i}^2+a_{kl}\cdot(\sum_{i\in V_k}\mu^{k,l}_i x_i)\cdot(\sum_{i\in V_l}\mu^{k,l}_i x_i).
    \end{equation}
    We further require $a_{kk}=0$ (resp. $a_{ll}=0$) if $D\cap V_k=\emptyset$ (resp. $D\cap V_l=\emptyset$). Then all $a_{kk},a_{ll}, a_{kl}, \mu^{k,l}_i$ are uniquely determined by \Cref{eq: smallprod} if they exist. It follows from part (1) of \Cref{lemma: 2supp} that $D\cap V_k\neq\emptyset$ implies $|V_k|=1$, and the same holds true for $V_l$. Therefore, the claim is clear when $D\cap(V_k\cup V_l)\neq\emptyset$, and for the remainder of the proof we assume that $D\cap(V_k\cup V_l)=\emptyset$. In this case, the claim is equivalent to the matrix $(\rho_f(e_i+e_j))_{i\in V_k, j\in V_l}$ having rank one. However, this is true because $\rho_f\in \upR_J(\T_\infty)$ implies that every $2\times2$ minor of this matrix vanishes.

    In order to prove part (1), it remains to show that for all pairwise different $k,l,l'\in[r]$ and $i\in V_k$, we have
        $\mu^{k,l}_i=\mu^{k,l'}_i$.
 To this end, let $i_1,i_2\in V_k$, $j_1\in V_l$, and $j_2\in V_{l'}$. Then
 \begin{equation*}
     \underbrace{\rho_f(e_{i_1}+e_{i_2})}_{=0}\cdot\rho_f(e_{j_1}+e_{j_2})=0.
 \end{equation*}
 Because $\rho_f$ is a representation over $\T_\infty$, this implies that
 \begin{equation*}
     \rho_f(e_{i_1}+e_{j_1})\rho_f(e_{i_2}+e_{j_2})=\rho_f(e_{i_1}+e_{j_2})\rho_f(e_{i_2}+e_{j_1}).
 \end{equation*}
 If we plug in the expressions from \Cref{eq: smallprod} for the coefficients of $f$, we obtain
    \begin{equation*}
        a_{kl}a_{kl'}\mu^{k,l}_{i_1}\mu_{j_1}^{k,l}\mu_{i_2}^{k,l'}\mu_{j_2}^{k,l'}=
        a_{kl}a_{kl'}\mu^{k,l'}_{i_1}\mu_{j_2}^{k,l'}\mu_{i_2}^{k,l}\mu_{j_1}^{k,l}
    \end{equation*}
    which shows that $\mu_{i_1}^{k,l}\mu_{i_2}^{k,l'}=\mu_{i_1}^{k,l'}\mu_{i_2}^{k,l}$. This means that the vectors $(\mu^{k,l}_i\mid i\in V_k)$ and $(\mu^{k,l'}_i\mid i\in V_k)$ are linearly dependent. Because both are in the standard simplex, these vectors must be equal. This implies the claim of part (1).

    For part (2), assume first that $g$ is Lorentzian. Then $f$ is Lorentzian by \cite{Branden-Huh20}*{Theorem 2.10}. Conversely, assume that $f$ is Lorentzian. We obtain $g$ by setting all but one variable from each group equal to zero and appropriately scaling the remaining one. This operation also preserves the Lorentzian property by \cite{Branden-Huh20}*{Theorem 2.10}.
\end{proof}

\begin{df}
    Given $f\in\upH(d,n)$ with $\supp(f)=J$ such that $\rho_f\in\upR_J(\T_\infty)$, we call the unique polynomial $g$ from \Cref{thm: simplification} the \emph{simplification} of $f$.
\end{df}

\begin{cor}\label{cor: biregular}
    Let $J\subseteq\Delta_n^2$ be an M-convex set. We use the notation from \Cref{df: jbar}.
    For $i\in[r]$, let $\Delta^\circ_{V_i}$ be the open standard simplex in $\R^{V_i}$, and consider the map
    \begin{equation*}
        \psi_J\colon\upR_{J^{\rm simp}}(\T_\infty)\times\Delta^\circ_{V_1}\times\cdots\times\Delta^\circ_{V_r}\longrightarrow \upR_J(\T_\infty),\qquad (\rho_g,\lambda_1,\ldots,\lambda_n)\longmapsto\rho_f
    \end{equation*}
    where $f=g(\sum_{i\in V_1} \lambda_i x_i,\ldots,\sum_{i\in V_r} \lambda_i x_i)$. Then:
    \begin{enumerate}[(1)]\itemsep 5pt
        \item $\psi_J$ is bijective.
        \item Both $\psi_J$ and $\psi_J^{-1}$ are regular in the sense that they are given by rational functions without poles on their domains.
        \item $\psi_J$ maps $\upL_{J^{\rm simp}}\times\Delta^\circ_{V_1}\times\cdots\times\Delta^\circ_{V_r}$ onto $\upL_{{J}}$.
    \end{enumerate}
\end{cor}

\begin{proof}
    It is clear that $\psi_J$ is regular, and by \Cref{thm: simplification} it is bijective. This proves part (1) and for part (2) it remains to show that the inverse of $\psi_J$ is regular on $\upR_J(\T_\infty)$. This follows from the fact that $g$ can be recovered from $f$ by equating all variables from one group, and the $\lambda_i$ can be recovered by dividing suitable coefficients of $f$ by a corresponding coefficient of $g$. Part (3) is covered by \Cref{thm: simplification}.
\end{proof}

\begin{cor}\label{cor: fullrankinterior}
    Let $f$ be a Lorentzian polynomial of degree two and let $J=\supp(f)$. If the Hessian of the simplification of $f$ has full rank, then $f$ is in the interior of $\upL_J$ relative to $\upR_J(\T_\infty)$.
\end{cor}

\begin{proof}
    If the Hessian of the simplification of $f$ has full rank, then it is in the interior of $\upL_{J^{\rm simp}}$. Hence the claim is implied by \Cref{cor: biregular}.
\end{proof}

We will later need the following partial converse of \Cref{cor: fullrankinterior}.

\begin{lemma}\label{lemma:singularboundary}
    If $f$ is in the interior of $\upL_{M}$ for $M=U_{2,n}$, then the Hessian of $f$ has full rank.
\end{lemma}

\begin{proof}
    Let $f\in\upL_M$ be such that the Hessian $A$ of $f$ has rank less than $n$. Note that this implies $n>2$. Let $0\neq v\in\ker(A)$. Because every off-diagonal entry of $A$ is positive, there are at least two indices $j,k\in[n]$ such that $v_j$ and $v_k$ are nonzero. Because $n>2$, there exists $i\in[n]\smallsetminus\{j,k\}$. Now let $B=(b_{rs})_{r,s}$ be the symmetric $n\times n$ matrix such that $b_{ik}=b_{ki}=1$, $b_{jk}=b_{kj}=-\frac{v_i+v_k}{v_j}$, and all other entries are zero. Then 
    \begin{equation*}
        v^tBv=2\cdot\left(v_i v_k-v_j v_k \frac{v_i+v_k}{v_j}\right)=-2v_k^2.
    \end{equation*}
    Furthermore, letting $e=e_i+e_k$, we have $e^tBe=2$ and
    \begin{equation*}
        e^tBv=v^tBe=v_i+v_k-v_j \cdot \frac{v_i+v_k}{v_j}=0.
    \end{equation*}
    Finally, we compute
    \begin{equation*}
        (\lambda\cdot e +\mu\cdot v)^t (A-\epsilon\cdot B) (\lambda e+ \mu v)=\lambda^2(e^tAe-2\epsilon)+\mu^2\cdot \epsilon\cdot v_k^2.
    \end{equation*}
    Therefore, for small enough $\epsilon>0$, the matrix $A-\epsilon\cdot B$ is positive definite on the two dimensional span of $e$ and $v$. This shows that $f$ is not in the interior of $\upL_M$.
\end{proof}

\subsection{Strong star-shapedness in degree two}
\label{subsection: strong star-shapedness in degree 2}

For $J\subseteq\Delta_n^2$ an M-convex set, we now examine the image $\log(\upL_J)$ of $\upL_J$ under the coefficient-wise logarithm map. 
Note that we take coefficients with respect to the normalized monomial basis, i.e., 
\[
\log(\upL_J) = \Big\{ (\log c_\alpha)_{\alpha \in J} \; | \; \sum_{\alpha \in J} c_\alpha \frac{x^\alpha}{\alpha !} \in \upL_J \Big\} \subseteq \R^J.
\]

By \Cref{lemma: petterlem}, $\log(\upL_J)$ is contained in the linear space $V_J=\log(\upR_J(\T_\infty))$. Furthermore, it follows from \Cref{thm:314325} that 
$\log(\upL_J)$ is star-shaped with respect to the origin. However, in general the origin does not lie in the interior of $\log(\upL_J)$. In particular, the triple $(0,\log(\upL_J),V_J)$ is not strongly star-shaped in general. However, we will prove that $\log(\upL_J)$ is a strongly star-shaped subset of $V_J$ with respect to another point.

\begin{lemma}\label{lemma: diagonalsmaller}
    Let $A=(a_{ij})_{ij}$ be a symmetric $n\times n$ matrix with entries in $\R_{\geq0}$ which has exactly one positive eigenvalue. If $0\leq d_{i}\leq a_{ii}$ for all $i\in[n]$, then the matrix
    \begin{equation*}
        A'=A-\textnormal{Diag}(d_{1},\ldots,d_{n})
    \end{equation*}
    has exactly one positive eigenvalue or is the zero matrix.
\end{lemma}

\begin{proof}
    Assume that the matrix $A'$ is not the zero matrix. Then $A'$ has at least one positive eigenvalue, because it has only nonnegative entries. It cannot have more than one positive eigenvalue, because then the same would be true for the matrix
    \begin{equation*}
        A=A'+{\textnormal{Diag}(d_{1},\ldots,d_{n})}
    \end{equation*}
    as ${\textnormal{Diag}(d_{1},\ldots,d_{n})}$ is positive semidefinite.
\end{proof}

\begin{thm}\label{thm: strongstardeg2}
    Let $J\subseteq\Delta_n^2$ be an M-convex set. For every $i\in[n]$, let $a_{ii}\in[0,1]$, and consider the polynomial
    \begin{equation*}
        f=\sum_{2e_i\in J} a_{ii}\cdot\frac{x_i^2}{2}+\sum_{e_i+e_j\in J,\,i\neq j} x_ix_j.
    \end{equation*}
    \begin{enumerate}[(1)]\itemsep 5pt
        \item The set $\log(\upL_J)$ is star-shaped with respect to $\log(f)$.
        \item If $a_{ii}<1$ for every $i\in[n]$, then the Hessian of the simplification of $f$ has full rank.
        \item If $a_{ii}<1$ for every $i\in[n]$, then $(\log(f),\log(\upL_J),V_J)$ is strongly star-shaped.
    \end{enumerate}
\end{thm}

\begin{proof}
    For part (1), we need to prove that for all $g\in\upL_J$ and $p\in[0,1]$ we have
    \begin{equation} \label{eq:logLJstarshaped}
        \log(f)+p\cdot(\log(g)-\log(f))\in\log(\upL_J).
    \end{equation}
    Let $(a_{ij})_{ij}$ and $(b_{ij})_{ij}$ be the Hessian matrices of $f$ and $g$, respectively. Then \eqref{eq:logLJstarshaped} is equivalent to the matrix
    \begin{equation*}
        \left(a_{ij}\cdot\left(\frac{b_{ij}}{a_{ij}}\right)^p\right)_{ij}=\left(b_{ij}^p\right)_{ij}-\textnormal{Diag}(d_{1},\ldots,d_{n})
    \end{equation*}
    having exactly one positive eigenvalue, where $d_i=b_{ii}^p\cdot(1-a_{ii}^{1-p})$. This follows from \Cref{lemma: diagonalsmaller} together with \cite{Branden-Huh20}*{Lemma 3.24}.

    For part (2), we observe that the Hessian of the simplification of $f$ has the form
    \begin{equation*}
        H=\textbf{1}_r-D,
    \end{equation*}
    where $\textbf{1}_r$ is the $r\times r$ all-ones matrix and $D$ is a diagonal matrix whose diagonal $v$ has strictly positive entries that are at most $1$. Furthermore, if $r=1$, then the entry of $v$ is strictly less than $1$. Letting $w$ be the vector whose $i$-th entry is $v_i^{-\frac{1}{2}}$, the matrix $H$ has the same rank as
    \begin{equation*}
        H'=w\cdot w^t- I_r.
    \end{equation*}
    The matrix $w\cdot w^t$ has only one nonzero eigenvalue, which is the norm of $w$. Because $0<v_i\leq 1$ and $0<v_1<1$ if $r=1$, the norm of $w$ is strictly larger than $1$, which implies that $H'$ has full rank.

    In order to prove part (3), we first note that $\log(\upL_J)$ is closed because $\upL_J$ is closed. By part (2) and \Cref{cor: fullrankinterior}, the point $\log(f)$ is in the interior of $\log(\upL_J)$ relative to $V_J$. In order to conclude that $(\log(f),\log(\upL_J))$ is strongly star-shaped as a subset of $V_J$, we would like to apply \Cref{lemma: starcrit2}. To this end, we need to prove that for every $g\in \upL_J$ there exists $0<p_0<1$ such that for every $p\in(p_0,1)$, the point 
    \begin{equation*}
     x_p=\log(f)+p\cdot (\log(g)-\log(f))   
    \end{equation*}
    lies in the interior of $\log(\upL_J)$ relative to $V_J$. We can write $x_p=\log(h_p)$ for some polynomial $h_p$. The determinant of the Hessian of the simplification of $h_p$ is an analytic function in $p$ by \Cref{cor: biregular}. It does not vanish for $p=0$ by part (2). Therefore, it has only isolated zeros. This implies that there exists $0<p_0<1$ such that for every $p\in(p_0,1)$, the Hessian of the simplification of $h_p$ has full rank. By \Cref{cor: fullrankinterior}, this implies the desired statement.
\end{proof}

\subsection{Conclusion about the topology}\label{subsec:conclusiontopo}
Now we are ready to complete the proof of \Cref{thm:topologystratum}.
Consider the linear operator $\upN$ on polynomials defined by the condition $\upN(x^\alpha)=\frac{x^\alpha}{\alpha!}$. (We call $N$ the \emph{normalization operator}.)
If a polynomial $f$ is Lorentzian, then $\upN(f)$ is also Lorentzian by \cite{Branden-Huh20}*{Corollary 3.7}. We will prove that $(\log(\upN(f_J)),\log(\upL_J),V_J)$ is strongly star shaped for every M-convex $J\subseteq\Delta^d_n$, where $V_J=\log(\upR_J(\T_\infty))$, as above, and $f_J=\sum_{\alpha\in J}\frac{x^\alpha}{\alpha!}$ is the (exponential) generating function of $J$, i.e., 
\begin{equation*}
        \upN(f_J)=\sum_{\alpha\in J}\frac{x^\alpha}{(\alpha!)^2}.
\end{equation*}

More generally, for any real number $t$, we write $\upN_t$ for the linear operator on polynomials defined by the condition
\begin{equation*}
    \upN_t(x^\alpha)=\frac{x^\alpha}{(\alpha!)^t}.
\end{equation*}
The operator $\upN_0$ is the identity and the operator $\upN_1$ is the normalization operator $\upN$. Note furthermore that $\upN_s\circ\upN_t=\upN_{s+t}$. The following lemma generalizes \cite{Branden-Huh20}*{Corollary 3.5} and \cite{Branden-Huh20}*{Corollary 3.7} by interpolating between them.

\begin{lemma}\label{Normalization}
The operator $\upN_t$ preserves the Lorentzian property for all $t \ge 0$.
\end{lemma}

\begin{proof}
The proof of \cite{Branden-Huh20}*{Corollary 3.7} admits a straightforward generalization and shows that the symbol of $\upN_t$ is a Lorentzian polynomial in $2n$ variables.
\end{proof}

We will need the following lemma.

\begin{lemma}\label{lemma: deriofnfj}
    Let $f=\sum_{\beta\in\Delta^d_n}\frac{c_\beta}{\beta!}\cdot x^\beta$, $a\in\Delta^{d-2}_n$, and $g=\partial^a \upN_t(f)$ for some $t\in\R$.
    Then
    \begin{align*}
        &g((1+a_1)^t\cdot x_1,\ldots,(1+a_n)^t\cdot x_n)\\=&\frac{1}{(a!)^t}\cdot\left(\sum_{i=1}^n c_{a+2e_i}\cdot\left(\frac{1+a_i}{2+a_i}\right)^t\cdot\frac{x_i^2}{2}+\sum_{1\leq i<j\leq n}c_{a+e_i+e_j}\cdot x_ix_j\right).
    \end{align*}
\end{lemma}

\begin{proof}
    This is a straight-forward calculation.
\end{proof}

Next, we note that the linear subspace $W_J\subseteq\R^J$ spanned by the vectors $(\alpha_i)_{\alpha\in J}$ for $i\in[n]$ is contained in $V_J$. Adding a multiple of one of these vectors in log-coordinates corresponds to scaling one variable of the corresponding polynomial by a positive constant. As the latter preserves Lorentzians, we conclude that $W_J+\log(\upL_J)=\log(\upL_J)$.

\begin{thm}\label{thm:logljstronglystarshaped}
    Let $t>0$, let $J\subseteq\Delta^d_n$ be an M-convex set, and let $f_J$ be the generating function of $J$. Then $(\log(\upN_t(f_J)),\log(\upL_J),V_J)$ is strongly star-shaped. 
\end{thm}

\begin{proof}
    We will apply \Cref{lemma: starcrit1} to $(x_*,X)=(\log(\upN_t(f_J)),\log(\upL_J))$. To this end, we consider $V=V_J$ as a linear subspace of $U=\R^J$. For every $a\in\Delta^{d-2}_n$, we consider the M-convex set 
    \begin{equation*}
     J_a=\{\alpha\in\Delta^2_n\mid \alpha+a\in J\}.
    \end{equation*}
    We let $U_a=\R^{J_a}$ and $\pi_a=\partial^a\colon U\to U_a$. Then $V_a=V_{J_a}=\log(\upR_{J_a}(\T_\infty))\subseteq U_a$ and $X_a=\log(\upL_{J_a})$ satisfy the assumptions of \Cref{lemma: starcrit1}. Therefore, it suffices to show that 
    \begin{equation*}
    (\pi_a(x_*),X_a)=(\log(\partial^a \upN_t(f_J)), \log(\upL_{J_a}))    
    \end{equation*}
    is strongly star-shaped as a subset of $V_{J_a}$ for all $a\in\Delta^{d-2}_n$. By \Cref{lemma: deriofnfj} and \Cref{lemma: starshapedlineality}, this follows from part (3) of \Cref{thm: strongstardeg2}.
\end{proof}

\begin{rem}\label{remark stable not star shaped}
    \Cref{thm:logljstronglystarshaped} is in general not true for the space of stable polynomials $\upS_J$ with support $J$, even if $\upS_J$ is nonempty. For example, for the non-Fano matroid $M=F_7^-$, the basis generating polynomial $f_M$ is not stable, although $\upS_{M}\neq\emptyset$ \cite{Choe-Oxley-Sokal-Wagner04}*{Example 11.5}. In particular, the set $\log(\upS_M)$ is not even star-shaped with respect to $\log(f_M)=\log(\upN(f_M))$. In \Cref{thm:betsyross} we will consider a matroid $M$ for which $\upS_M$ is not even connected.
\end{rem}

\begin{lemma}\label{lem: diraremconvex}
    Let $J\subseteq\Delta^d_n$ be an M-convex set and let $f_J$ be the generating function of $J$. For $t\geq0$ and $\nu\in V_J\subseteq\R^J$, we have
    \begin{equation*}
        \log(\upN_t(f_J)) \ + \ s\cdot\nu \ \ \in \ \ \log(\upL_J)
    \end{equation*}
    for all $s\geq0$ if and only if $\nu\colon J\to\R$ is M-concave.
\end{lemma}

\begin{proof}
    The ``if'' direction follows from \Cref{thm:314325} and the fact that the operator $\upN_t$ preserves Lorentzians (\Cref{Normalization}). The ``only if'' direction is implied by \cite{Branden-Huh20}*{Theorem 3.20} and Tarski's principle which says that a first-order sentence in the language of ordered fields holds in a given real closed field if and only if it holds in $\R$ \cite{Prestel84}*{Section 5}.
\end{proof}

\begin{thm}\label{thm:Alor}
    Let $\upB\subseteq V_J/\R\mathbf{1}$ be the unit ball with respect to some norm, and let $X=\upB\smallsetminus(\Dr_J\cap\partial B)$. Then the space $\P\upL_J$ is homeomorphic to $X$.
\end{thm}

\begin{proof}
    Let $t>0$. By \Cref{thm:logljstronglystarshaped} and part (1) of \Cref{lemma: starshapedlineality2},  it follows that 
    \begin{equation*}
(\log(\upN_t(f_J)),\log(\P\upL_J),V_J/\R\mathbf{1})
    \end{equation*}
    is strongly star-shaped. It follows from \Cref{lem: diraremconvex} and part (2) of \Cref{lemma: starshapedlineality2} that $\Dr_J\cap\partial B$ is equal to $S(\log(\upN_t(f_J)),\log(\P\upL_J))$. 
    Thus the claim follows from \Cref{cor: ballminusdirections}.
\end{proof}

\begin{cor}\label{cor:lorismanifoldwithboundary}
    For an M-convex set $J$, the space $\P\upL_J$ is a manifold with boundary.
\end{cor}
\begin{proof}
    This follows from \Cref{cor: manifoldwithboundary} and \Cref{thm:logljstronglystarshaped}.
\end{proof}

\begin{cor}\label{cor:allarehomeo}
    The spaces $\P\upL_J$, $\Gr_J(\T_q)$, and $\Gr^{\rm w}_J(\T_q)$ are all homeomorphic to each other.
\end{cor}
\begin{proof}
    This is \Cref{thm:Alor} together with \Cref{thm:ATq}.
\end{proof}

By the  same argument, we obtain the corresponding results for orbit spaces.

\begin{thm}\label{thm:Alorreduced}
    Let $\upB\subseteq V_J/W_J$ be the unit ball with respect to some norm and let $X=\upB\smallsetminus(\ulineDr_J\cap\partial \upB)$. Then the space $\ulineL_J$ is homeomorphic to $X$.
\end{thm}

\begin{cor}
    The spaces $\ulineL_J$, $\ulineGr_J(\T_q)$, and $\ulineGr^{\rm w}_J(\T_q)$ are all homeomorphic to each other.
\end{cor}

The preceding results allow us to transfer all the results on the topology of $\ulineGr^{\rm w}_J(\T_q)$ and $\Gr^{\rm w}_J(\T_q)$ from \Cref{rem:foundation} to $\ulineL_J$ and $\P\upL_J$. 

\section{Some detailed examples}
\subsection{Small uniform matroids}\label{sec:smalluniform}
In this section we fix a matroid $M$ on $[n]$. The map
\begin{equation*}
    \C \longrightarrow \R_{\geq0}, \qquad
    z \longmapsto |z|^2, 
\end{equation*}
defines a homomorphism of tracts from the field $\C$ to $\T_2$ by \Cref{lemma: morphisms from c to tq}. From this, we obtain a (continuous) map $\upR_M(\C)\to\upR_M(\T_2)$. Our goal is to prove the following:

\begin{thm}\label{thm:maptoboundary}
    The image of the map $\upR_M(\C)\to\upR_M(\T_2)$ is contained in $\upL_M$. In fact, every polynomial in the image is stable. Moreover, we have:
    \begin{enumerate}[(1)]\itemsep 5pt
        \item If $M$ has a $U_{2,4}$ minor, then the image of $\upR_M(\R)$ is contained in $\partial\upL_M$.
        \item If $M$ has a $U_{2,5}$ minor, then the image of $\upR_M(\C)$ is contained in $\partial\upL_M$.
        \newcounter{enumTemp}
    \setcounter{enumTemp}{\theenumi}
    \end{enumerate}
    (Here $\partial\upL_M$ is the boundary of $\upL_M$ relative to $\upR_M(\T_\infty)$.)
    \begin{enumerate}[(1)]\itemsep 5pt
        \setcounter{enumi}{\theenumTemp}
        \item If $M=U_{2,4}$, then the induced map
        \begin{equation*}
            \ulineGr_M(\R)\longrightarrow\partial\ulineL_M
        \end{equation*}
        is a homeomorphism.
        \item The map $\Gr(2,4)(\R)\to\partial\P\upL(2,4)_\sqfree$ is the quotient by the multiplicative group $\{-1,1\}^4$ acting on $\Gr(2,4)(\R)$ via rescaling.
    \end{enumerate}
\end{thm}

\begin{rem}
    By \cite{Branden21}, the space $\P\upL(2,4)_\sqfree$ is homeomorphic to a five-dimensional ball. Therefore, part (4) of \Cref{thm:maptoboundary} implies that the quotient space
    \begin{equation*}
        \Gr(2,4)(\R)/\{-1,1\}^4
    \end{equation*}
    is homeomorphic to a four-dimensional sphere. This non-obvious fact was independently proven in \cite{Buchstaber-Terzic}.
\end{rem}

\begin{rem}
    Let $\pi\colon\Gr^+(2,4)(\R)\to\Gr(2,4)(\R)$ be the universal cover. It is well known that $\Gr^+(2,4)(\R)$ is homeomorphic to $S^2\times S^2$ and $\pi$ is the quotient map by the action of $(t,t)$ on $S^2\times S^2$, where $t\colon S^2\to S^2$ is the antipodal map. In order to describe the map from part (4) of \Cref{thm:maptoboundary}, for $i=1,2,3$, let $h_i\colon S^2 \to S^2$ be the reflection along the hyperplane $x_i=0$, where we think of $S^2$ being embedded in the standard way into $\R^3$, and let $g_i=(h_i, h_i)$ be the diagonal action on $S^2\times S^2$. Furthermore, let $g_4$ exchange the two copies of $S^2\times S^2$. Then $g_1, ..., g_4$ generate a group $K$ isomorphic to $\{-1,1\}^4$ that acts on $S^2\times S^2$. The map $\Gr^+(2,4)(\R)\to\partial\P\upL(2,4)_\sqfree$, obtained by precomposing $\pi$ with the map from  part (4) of \Cref{thm:maptoboundary} is the quotient map $S^2\times S^2 \to (S^2\times S^2) / K$. Since the componentwise antipodal map $(t, t)$ is in $K$, it factors through the map $(S^2\times S^2) / (t,t) \to (S^2\times S^2) / K$, which is the map  from  part (4) of \Cref{thm:maptoboundary}. As in the previous remark, we obtain as a corollary the non-obvious statement that $(S^2\times S^2)/K$ is homeomorphic to $S^4$.
\end{rem}

The proof of \Cref{thm:maptoboundary} requires some preparation.

\begin{lemma}\label{lemma:propermap}
    Let $K$ be $\R$ or $\C$. The map $\upR_M(K)\to\upR_M(\T_2)$ is proper.
\end{lemma}

\begin{proof}
    Let $\cB$ be the set of bases of $M$. The map $K^{\cB}\to\R^{\cB}$ defined by taking the square of the absolute value of each coordinate is clearly proper. The restriction to the preimage of the set where all coordinates are nonzero remains proper. The same is then true for the restriction to the closed subset of points satisfying the Pl\"ucker relations.
\end{proof}

\begin{cor}\label{cor:mapisclosed}
    Let $K$ be $\R$ or $\C$. The map $\ulineGr_M(K)\to\ulineGr_M(\T_2)$ is closed.
\end{cor}

\begin{proof}
    Let $A\subseteq\ulineGr_M(K)$ be a closed subset and let $B\subseteq\ulineGr_M(\T_2)$ be its image. Then the preimage $A'$ of $A$ under the map $\upR_M(K)\to\ulineGr_M(K)$ is sent by the map $\upR_M(K)\to\upR_M(\T_2)$ to the preimage $B'$ of $B$ under the map $\upR_M(\T_2)\to\ulineGr_M(\T_2)$. The set $A'$ is closed, as the preimage of a closed set, and $B'$ is closed by \Cref{lemma:propermap}. Finally, because $\ulineGr_M(\T_2)$ carries the quotient topology of $\upR_M(\T_2)$, we see that $B$ is closed.
\end{proof}

\begin{lemma}\label{lemma:detrepstable}
    Let $0\neq f\in\R[x_1,\ldots,x_n]$ be a homogeneous polynomial of degree $d$ with nonnegative coefficients.
    \begin{enumerate}[(1)]\itemsep 5pt
        \item If there exist positive semidefinite Hermitian $d\times d$ matrices $A_1,\ldots,A_n$ such that
        \begin{equation*}
            f=\det(x_1A_1+\cdots+x_nA_n),
        \end{equation*}
        then $f$ is stable.
        \item If $f$ is stable, then $f$ is Lorentzian.
    \end{enumerate}
\end{lemma}

\begin{proof}
    Part (1) is \cite{Branden07}*{Lemma 4.1} and part (2) is \cite{Branden-Huh20}*{Proposition 2.2}.
\end{proof}

\begin{lemma}\label{cor:cauchybinet}
    Let $A$ be a complex $d\times n$ matrix with columns $v_1,\ldots,v_n\in\C^d$, let $X$ be the $n\times n$ diagonal matrix with diagonal entries $x_1,\ldots,x_n$, and let $f=\det(AXA^*)$. Here $B^*$ denotes the conjugate transpose of a matrix $B$.
    \begin{enumerate}[(1)]\itemsep 5pt
        \item We have the following equalities:
        \begin{equation*}
          f=\det\left(x_1\cdot v_1\cdot v_1^*+\cdots+x_n\cdot v_n\cdot v_n^*\right)=\sum_{S\in\binom{[n]}{d}} |A(S)|^2 \cdot \prod_{i\in S} x_i  
        \end{equation*}
        where $A(S)$ denotes the determinant of the $d\times d$ submatrix of $A$ whose columns are indexed by $S$.
        \item The polynomial $f$ is stable (and hence Lorentzian).
    \end{enumerate}
    In particular, the image of the map 
     $\upR_M(\C)\to\upR_M(\T_2)$    
    is contained in $\upS_M\subseteq\upL_M$.
\end{lemma}

\begin{proof}
    The first equality in part (1) is straight-forward and the second one follows from the Cauchy--Binet formula. Now part (2) follows from \Cref{lemma:detrepstable} applied to the positive semidefinite Hermitian $d\times d$ matrices $v_i\cdot v_i^*$.
\end{proof}

The proof of \Cref{cor:cauchybinet} gives a more precise description of the image of the map $\upR_M(\C)\to\upL_M$.

\begin{lemma}\label{lemma:imageiffdetrep}
    A Lorentzian polynomial $f\in\upL_M$ is in the image of $\upR_M(\C)\to\upL_M$ if and only if there are positive semidefinite Hermitian $d\times d$ matrices $A_1,\ldots,A_n$ such that $A_1+\cdots+A_n$ is positive definite and
    \begin{equation}\label{eq:detrep2}
        f=\det(x_1A_1+\cdots+x_nA_n).
    \end{equation}
    It is in the image of $\upR_M(\R)\to\upL_M$ if and only if these matrices can be chosen to be real. 
\end{lemma}

\begin{proof}
    The `only if' direction for both claims follows from part (1) of \Cref{cor:cauchybinet}. For the `if' direction, we first observe that the matrices $A_i$ must have rank one by \cite{Branden11}*{page 1207} (since $f$ is multi-affine). Thus, we can write $A_i=v_i\cdot v_i^*$ for suitable $v_i\in\C^d$ (resp. $v_i\in\R^d$ in the real case). The coefficients of $f$ are given by the square of the absolute value of the maximal minors of the matrix whose columns are $v_1,\ldots,v_n$, which proves the claim.
\end{proof}

Whether a multi-affine polynomial $f$ has a representation as in \Cref{eq:detrep2} was characterized in \cite{Kummer-Plaumann-Vinzant15} for real matrices and in \cite{Ahmadieh-Vinzant24} for Hermitian matrices. Namely, $f$ has a representation of this form if and only if the polynomial
\begin{equation*}
    \frac{\partial f}{\partial x_i}\cdot\frac{\partial f}{\partial x_j}-f\cdot \frac{\partial^2 f}{\partial x_i \partial x_j}
\end{equation*}
is a (Hermitian) square for all $i,j\in[n]$. We will derive another characterization when $f$ has degree two.

\begin{lemma}\label{lemma:detreplor}
    Consider $r\geq1$ and the polynomial
    \begin{equation*}
        f=x_1^2-(x_2^2+\cdots+x_r^2).
    \end{equation*}
    There are Hermitian $2\times2$ matrices $A_1,\ldots,A_r$ such that
    \begin{equation}
        f=\det(x_1A_1+\cdots+x_rA_r)
    \end{equation}
    if and only if $r\leq4$. These matrices can be chosen to be real if and only if $r\leq3$. Furthermore, whenever such matrices exist, they can be chosen in such a way that $A_1$ is the identity matrix.
\end{lemma}

\begin{proof}
    In both cases, the `if' direction can be deduced from the following identity:
    \begin{equation}\label{eq:detlor4}
        x_1^2-x_2^2-x_3^2-x_4^2=\det\begin{pmatrix}
            x_1+x_2& x_3+ i x_4\\
            x_3-i x_4 & x_1-x_2
        \end{pmatrix}.
    \end{equation}
    Conversely, assume there are Hermitian (resp. real symmetric) $2\times 2$ matrices $A_1,\ldots,A_r$ with $f=\det(A(x))$ where $A(x)=x_1A_1+\cdots+x_rA_r$. If $r>4$ (resp. $r>3$), then, for dimension reasons, there exists $0\neq\lambda\in\R^r$ such that $A(\lambda)=0$. Then the gradient of $\det(A(x))$ vanishes at $\lambda$, which contradicts the assumption that $f=\det(A(x))$  because the gradient of $f=x_1^2-(x_2^2+\cdots+x_r^2)$ does not vanish at nonzero points in $\R^r$.
\end{proof}

\begin{prop}\label{prop:imagedeg2}
    Assume that the rank of $M$ is two.
    A Lorentzian polynomial $f\in\upL_M$ is in the image of $\upR_M(\C)\to\upL_M$ if and only if its Hessian has rank at most four. It is in the image of $\upR_M(\R)\to\upL_M$ if and only if its Hessian has rank at most three. 
\end{prop}
    
\begin{proof}
   We prove the statement for $\upR_M(\C)\to\upL_M$; the statement about  $\upR_M(\R)\to\upL_M$ follows analogously. 
   
   Let $f\in\upL_M$. Because $f$ has degree two and nonnegative coefficients, and because its Hessian has exactly one positive eigenvalue, there is a linear change of coordinates $\psi\colon\R^n\to\R^n$ with $\psi(e_1)=\sum_{i=1}^ne_i$ such that
   \begin{equation*}
       f\circ\psi=x_1^2-(x_2^2+\cdots+x_r^2),
   \end{equation*}
   where $r$ is the rank of the Hessian of $f$. 
   
   Assume that $f$ is in the image $\upR_M(\C)\to\upL_M$. Then, by \Cref{lemma:imageiffdetrep}, there are Hermitian $2\times 2$ matrices $A_1,\ldots,A_n$ such that $f=\det (A(x))$, where $A(x)=x_1A_1+\cdots+x_nA_n$. This implies that
   \begin{equation*}
       x_1^2-(x_2^2+\cdots+x_r^2)=\det (A(\psi(x))).
   \end{equation*}
   \Cref{lemma:detreplor} now implies that $r\leq4$. 
   
   Conversely, assume that $r\leq4$. Then, by \Cref{lemma:detreplor}, there are Hermitian $2\times2$ matrices $A_1,\ldots,A_n$ such that $f\circ\psi=\det(A(x))$ and $A_1=I_2$ is the identity matrix. (Here we choose $A_{r+1}=\cdots=A_n=0$.) We have
   \begin{equation*}
       f=\det(A(\psi^{-1}(x)))
   \end{equation*}
   and $A(\psi^{-1}(\sum_{i=1}^ne_i))=I_2$ is positive definite. By \Cref{lemma:imageiffdetrep}, it remains to prove that $A(\psi^{-1}(e_i))$ is positive semidefinite for $i=1,\ldots,n$. This follows since
   \begin{equation*}
       \det(t\cdot I_2+A(\psi^{-1}(e_i)))=f(t,\ldots,t,\underbrace{t+1}_{i\textnormal{th position}},t,\ldots,t)
   \end{equation*}
   has only non-positive zeros (which follows from the fact that $f$ has nonnegative coefficients).
\end{proof}

\begin{lemma}\label{lemma:imageofsquaremap}
    Let $M$ be a matroid of rank two and let $n\in\N$ be maximal such that $M$ has an $U_{2,n}$-minor. 
    \begin{enumerate}[(1)]\itemsep 5pt
        \item If $n\leq3$, then the map $\upR_{M}(\R)\to\upL_{M}$ is surjective.
        \item If $n=4$, then the map $\upR_{M}(\C)\to\upL_{M}$ is surjective.
        \item If $n=5$, then the map $\upR_{M}(\C)\to\partial\upL_{M}$ is surjective.
    \end{enumerate}
\end{lemma}

\begin{proof}
    Let $f$ be an element of the target of one of the maps in question. It follows from \Cref{cor: biregular} that  the Hessian of $f$ has rank at most three in case (1) and rank at most four in case (2). In case (3), it follows from \Cref{cor: biregular} and \Cref{cor: fullrankinterior} that the rank of the Hessian of $f$ is at most four. Now the claim follows from \Cref{prop:imagedeg2}.
\end{proof}

\begin{proof}[Proof of \Cref{thm:maptoboundary}(1)--(2)]
    We prove part (2); part (1) can be proved analogously. 
    Suppose $M$ has a $U_{2,5}$ minor, i.e., there are disjoint, possibly empty, subsets $X,Y\subseteq[n]$ such that $(M/X)\smallsetminus Y$ is isomorphic to $U_{2,5}$. Without loss of generality, we can assume that $[n]\smallsetminus(X\cup Y)=[5]$. We consider the commutative diagram
    \begin{equation}\label{eq:diagramminor}
        \begin{tikzcd}
            \upR_M(\C)\arrow{r}\arrow{d}&\upL_M\arrow{d}\\
            \upR_{U_{2,5}}(\C)\arrow{r}& \upL_{U_{2,5}}
        \end{tikzcd}
    \end{equation}
    where the horizontal arrows are the coefficient-wise absolute square maps $z\mapsto|z|^2$ and the vertical arrows correspond on the level of polynomials to
    \begin{equation*}
        h\mapsto\left(\prod_{i\in X}\frac{\partial}{\partial x_i} h\right)|_{x_j=0\textrm{ for }j\in Y}.
    \end{equation*}
    Let $f\in\upL_M$ be in the image of $\upR_{M}(\C)\to\upL_{M}$. It suffices to show that for every $\epsilon>0$, we have $(1+\epsilon)\cdot\log(f)\not\in\log(\upL_M)$. For this, it suffices to show that for every $\epsilon>0$, we have $(1+\epsilon)\cdot\log(g)\not\in\log(\upL_{U_{2,5}})$, where $g\in\upL_{U_{2,5}}$ is the image of $f$ under the map $\upL_M\to\upL_{U_{2,5}}$ from \eqref{eq:diagramminor}. 
    Because $g$ is in the image of the map $\upR_{U_{2,5}}(\C)\to\upL_{U_{2,5}}$ from \eqref{eq:diagramminor}, the Hessian of $g$ has rank at most four by \Cref{prop:imagedeg2}. Thus, \Cref{lemma:singularboundary} implies that $g$ is on the boundary of $\upL_{U_{2,5}}$ (inside $\upR_{U_{2,5}}(\T_\infty)$). By \Cref{thm:logljstronglystarshaped}, this shows that $(1+\epsilon)\cdot\log(g)\not\in\log(\upL_{U_{2,5}})$ for all $\epsilon>0$.
\end{proof}

We now prepare for the proof of parts (3) and (4) of \Cref{thm:maptoboundary}.

\begin{lemma}\label{lemma:complexslorbits}
    Let $f=x_1^2-x_2^2-x_3^2-x_4^2$ and
    \[
     A_1 \ = \ \big(\begin{smallmatrix} 1& 0\\ 0 & 1\end{smallmatrix}\big), \qquad A_2 \ = \ \big(\begin{smallmatrix} 1& 0\\ 0 & -1 \end{smallmatrix}\big), \qquad A_3 \ = \ \big(\begin{smallmatrix} 0& 1\\ 1 & 0 \end{smallmatrix}\big), \qquad A_4 \ = \ \big(\begin{smallmatrix} 0& i\\ -i & 0 \end{smallmatrix}\big).     
    \]
    The following holds true:
    \begin{enumerate}[(1)]\itemsep 5pt
        \item $f=\det(x_1A_1+x_2A_2+x_3A_3+x_4A_4)$.
        \item The set of tuples $(B_1,B_2,B_3,B_4)$ of Hermitian $2\times 2$ matrices with 
        \begin{equation*}
         f=\det(x_1B_1+x_2B_2+x_3B_3+x_4B_4) 
        \end{equation*}
        consists of exactly four orbits under the $\textnormal{SL}_2(\C)$-action of coordinate-wise conjugation, i.e., $(B_1,B_2,B_3,B_4)\mapsto(SB_1S^*,SB_2S^*,SB_3S^*,SB_4S^*)$ for $S\in\textnormal{SL}_2(\C)$. Again $S^*$ denotes the conjugate transpose of $S$.
        \item The following tuples are representatives of the four orbits:
        \begin{equation*}
            (A_1,A_2,A_3,A_4),\quad (-A_1,-A_2,-A_3,A_4),\quad (A_1,A_2,A_3,-A_4), \quad (-A_1,-A_2,-A_3,-A_4).
        \end{equation*}
    \end{enumerate}
\end{lemma}

\begin{proof}
    Part (1) is a direct calculation, see also \Cref{eq:detlor4}. For part (2), we first observe that such a tuple $(B_1,B_2,B_3,B_4)$ is necessarily a basis of the real vector space $H$ of Hermitian $2\times 2$ matrices. Indeed, if not, there would be some $0\neq\lambda\in\R^4$ such that $B(\lambda)$ is the zero matrix, but then the gradient of $\det(B(x))$ at $\lambda$ would be zero, contradicting $f=\det(B(x))$. Therefore, there exists an automorphism of the vector space $H$ which maps $A_i$ to $B_i$. Because it preserves $f$, this automorphism is an element of the Lorentz group $\textnormal{O}(1,3)$. On the other hand, the action $Y\mapsto SYS^*$ for $S\in\textnormal{SL}_2(\C)$ defines a group homomorphism $\textnormal{SL}_2(\C)\to\textnormal{O}(1,3)$ whose image is the identity component $\textnormal{SO}^+(1,3)$ of $\textnormal{O}(1,3)$. Because $\textnormal{O}(1,3)$ has four connected components \cite{Hall03}*{page 15}, this shows that there are four orbits. For part (3), it remains to show that the four given matrices are in different orbits. This can be seen by observing that the $\textnormal{SL}_2(\C)$-action preserves the property of being positive definite, and that elements in the image of $\textnormal{SL}_2(\C)\to\textnormal{O}(1,3)$ have positive determinant.
\end{proof}

\begin{cor}\label{cor:twoorbitcomplex}
    The set of all tuples $(B_1,B_2,B_3,B_4)$ of Hermitian $2\times 2$ matrices such that $B_1$ is positive definite and
        \begin{equation*}
         \det(x_1B_1+x_2B_2+x_3B_3+x_4B_4) = \lambda\cdot(x_1^2-x_2^2-x_3^2-x_4^2)
        \end{equation*}
    for some $\lambda>0$ consists of exactly two orbits under the $\textnormal{GL}_2(\C)$-action of coordinate-wise conjugation. Complex conjugation exchanges these two orbits.
\end{cor}

\begin{proof}
    By \Cref{lemma:complexslorbits}, the set in question is the union of the $\textnormal{GL}_2(\C)$-orbits of the following pair of complex conjugated tuples:
    \begin{equation*}
            (A_1,A_2,A_3,A_4) \textnormal{ and } (A_1,A_2,A_3,-A_4).
        \end{equation*}
    Because $\textnormal{GL}_2(\C)$ is connected, for every $S\in\textnormal{GL}_2(\C)$ the automorphism $Y\mapsto SYS^*$ of the space of Hermitian matrices has positive determinant. This shows that these are two different orbits.
\end{proof}

\begin{lemma}\label{lemma:realslorbits}
    Let $f=x_1^2-x_2^2-x_3^2$ and
    \[
     A_1=\big(\begin{smallmatrix} 1& 0\\0 & 1\end{smallmatrix}\big), \qquad A_2=\big(\begin{smallmatrix} 1& 0\\ 0 & -1 \end{smallmatrix}\big), \qquad A_3=\big(\begin{smallmatrix} 0& 1\\ 1 & 0 \end{smallmatrix}\big).
    \]
    The following holds true:
    \begin{enumerate}[(1)]\itemsep 5pt
        \item $f=\det(x_1A_1+x_2A_2+x_3A_3)$.
        \item The set of tuples $(B_1,B_2,B_3)$ of real symmetric $2\times 2$ matrices with
        \begin{equation*}
         f=\det(x_1B_1+x_2B_2+x_3B_3  ) 
        \end{equation*}
        consists of exactly four orbits under the $\textnormal{SL}_2(\R)$-action of coordinate-wise conjugation.
        \item The following tuples are representatives of the four orbits:
        \begin{equation*}
            (A_1,A_2,A_3),\quad  (-A_1,-A_2,A_3),\quad  (A_1,A_2,-A_3), \quad (-A_1,-A_2,-A_3).
        \end{equation*}
    \end{enumerate}
\end{lemma}

\begin{proof}
    This follows from an argument similar to the proof of the previous lemma.
\end{proof}

\begin{cor}\label{cor:oneorbitreal}
    The set of all tuples $(B_1,B_2,B_3)$ of real symmetric $2\times 2$ matrices such that $B_1$ is positive definite and
        \begin{equation*}
         \det(x_1B_1+x_2B_2+x_3B_3) = \lambda\cdot(x_1^2-x_2^2-x_3^2)
        \end{equation*}
    for some $\lambda>0$ consists of exactly one orbit under the $\textnormal{GL}_2(\R)$-action  of coordinate-wise conjugation.
\end{cor}

\begin{proof}
    By \Cref{lemma:realslorbits}, the set in question is the union of the $\textnormal{GL}_2(\R)$-orbits of the two tuples
    \begin{equation*}
            (A_1,A_2,A_3) \textnormal{ and } (A_1,A_2,-A_3).
        \end{equation*}
    The claim then follows from the identity
    \begin{equation*}
            (A_1,A_2,A_3) = \begin{pmatrix}
                -1&0\\0&1
            \end{pmatrix} (A_1,A_2,-A_3)
            \begin{pmatrix}
                -1&0\\0&1
            \end{pmatrix}.\qedhere
        \end{equation*}
\end{proof}

\begin{lemma}\label{lemma:rankatleastthree}
    Let $M=U_{2,n}$ for $n\geq3$, and let $f\in\upL_M$. The Hessian of $f$ has rank at least three.
\end{lemma}

\begin{proof}
    After scaling the first three variables appropriately, we can assume that the cofficients of the monomials $x_ix_j$ for $1\leq i<j\leq 3$ are all one. Now the claim is apparent.
\end{proof}

\begin{proof}[Proof of \Cref{thm:maptoboundary}(3)]
    Let $M=U_{2,4}$.
    We first prove that the map $\ulineGr_M(\R)\to\partial\ulineL_M$ is surjective. For this, it suffices to prove that the map $\upR_M(\R)\to\partial\upL_M$ is surjective. To this end, let $f\in\partial\upL_M$. By \Cref{cor: fullrankinterior}, this implies that the rank of the Hessian of $f$ is at most three, and therefore $f$ is in the image of $\upR_M(\R)\to\partial\upL_M$ by \Cref{prop:imagedeg2}.

    As a next step, we prove injectivity of the map $\ulineGr_M(\R)\to\partial\ulineL_M$. To this end, let $R$ and $R'$ be real $2\times4$ matrices that both represent the same $f\in\partial\upL_M$. We have to show that, after scaling their columns by nonzero scalars, their rows span the same two dimensional subspace of $\R^4$.
    Let $v_1,v_2,v_3,v_4\in\R^2$ and $v_1',v_2',v_3',v_4'\in\R^2$ be the columns of $R$ and $R'$, respectively. Letting $A_i=v_i\cdot v_i^t$ and $A_i=v_i'\cdot v_i'^t$ for $i=1,2,3,4$, \Cref{cor:cauchybinet} shows that
    \begin{equation*}
        f=\det(x_1A_1+x_2A_2+x_3A_3+x_4A_4)=\det(x_1A_1'+x_2A_2'+x_3A_3'+x_4A_4').
    \end{equation*}
    By \Cref{cor: fullrankinterior} and \Cref{lemma:rankatleastthree}, the Hessian of $f$ has rank three. Thus, by a linear change of coordinates, one can bring $f$ to the form $x_1^2-x_2^2-x_3^2$. Therefore, it follows from \Cref{cor:oneorbitreal} that there exists $S\in\textnormal{GL}_2(\R)$ such that $A_i=SA_i'S^t$ for all $i=1,2,3,4$. Replacing $R'$ by $SR$ does not change the row span, so we can assume from now on that $v_iv_i^t=v_i'v_i'^t$ holds for $i=1,2,3,4$. This implies that $v_i=\pm v_i'$ for $i=1,2,3,4$ as desired.

    Therefore, the map $\ulineGr_M(\R)\to\partial\ulineL_M$ is continuous and bijective. Because it is also closed by \Cref{cor:mapisclosed}, it is a homeomorphism.
\end{proof}

\begin{proof}[Proof of \Cref{thm:maptoboundary}(4)]
    We denote by $G=(\R^\times)^4$ and $H=(\R^\times)^4$ the groups that act on $\Gr(2,4)(\R)$ and $\P\upL(2,4)_\sqfree$, respectively. We also write $f\colon\Gr(2,4)(\R)\to\partial\P\upL(2,4)_\sqfree$ and $\varphi\colon G\to H$ for the map that takes squares coordinate-wise. Finally, we let $K=\ker(\varphi)=\{-1,1\}^4$. 
    
    Since $\Gr(2,4)(\R)$ is compact, it suffices to show that the fibers of $f$ are exactly the orbits of the action of $K$ on $\Gr(2,4)(\R)$. By construction, we have $\varphi(g).f(x)=f(g.x)$ for all $g\in G$ and $x\in\Gr(2,4)(\R)$. This shows in particular that every fiber of $f$ is invariant under $K$. We also note that every fiber $f^{-1}(y)$ of $f$ is contained in an orbit of $G$; this follows from part (3) of \Cref{thm:maptoboundary} for $y\in\partial\P\upL_{U_{2,4}}$ and the fact that $G$ acts transitively on $\Gr_M(\R)$ for every other matroid $M$ of rank $2$ on $[4]$. 
    
    Next, we claim that it suffices to prove for all $x\in\Gr(2,4)(\R)$ that
    \begin{equation}\label{eq:stabandkerngeneratepreim}
        \varphi^{-1}(H_{f(x)})\subseteq K\cdot G_x.
\end{equation} In fact, if $f(x)=f(y)$ then there exists $g\in G$ with $y=g.x$ and thus $\varphi(g)\in H_{f(x)}$. By \Cref{eq:stabandkerngeneratepreim}, we can write $g=ab$ with $a\in K$ and $b\in G_x$. Thus, $y=g.x=a.(b.x)=a.x$, which proves the claim. 

In order to prove \Cref{eq:stabandkerngeneratepreim}, let $x\in\Gr_M(\R)$. We first note that $H_{f(x)}$ consists of all tuples for which all the entries corresponding to the same connected component of $M$ agree. Then $\varphi^{-1}(H_{f(x)})$ consists of all tuples for which all the entries corresponding to the same connected component of $M$ agree up to a sign. This proves \Cref{eq:stabandkerngeneratepreim}, and hence \Cref{thm:maptoboundary}(4).
\end{proof}

We conclude with a lemma for future reference.

\begin{lemma}\label{lemma:orbitfibers}
    Let $M$ be a matroid of rank two. Then the fiber over every element in the image of the map $\ulineGr_{M}(\C)\to\ulineL_{M}$ is an orbit under complex conjugation.
\end{lemma}

\begin{proof}
    The statement is trivial if the simplification of $M$ is $U_{2,2}$. Hence we can assume that $M$ has a $U_{2,3}$-minor.
    We first prove that the fiber over an element in the image of $\ulineGr_{M}(\R)$ is a singleton.
    By \Cref{prop:imagedeg2} and \Cref{lemma:rankatleastthree}, the Hessian of every element $f$ in the image of $\upR_M(\R)\to\upL_M$ has rank three. 
    The desired statement now follows from the same argument we used to prove injectivity in \Cref{thm:maptoboundary}(3).

    Next, we prove that the fiber over an element in the image of $\ulineGr_{M}(\C)$ which is not in the image of $\ulineGr_{M}(\R)$ has cardinality two, and that the two fibers are exchanged by complex conjugation. Let $R$ and $R'$ be complex $2\times n$ matrices that both represent the same $f$ in the image, which cannot be represented over $\R$. 
    Let $v_1,\ldots,v_n\in\C^2$ and $v_1',\ldots,v_n'\in\C^2$ be the columns of $R$ and $R'$, respectively. Letting $A_i=v_i\cdot v_i^*$ and $A_i=v_i'\cdot v_i'^*$ for $i=1,\ldots,n$, \Cref{cor:cauchybinet} shows that
    \begin{equation*}
        f=\det(x_1A_1+\cdots+x_nA_n)=\det(x_1A_1'+\cdots+x_nA_n').
    \end{equation*}
    By \Cref{prop:imagedeg2}, the Hessian of $f$ has rank four. Thus, by a linear change of coordinates, one can bring $f$ to the form $x_1^2-x_2^2-x_3^2-x_4^2$. Using \Cref{cor:twoorbitcomplex}, we can then argue as in the proof of injectivity in \Cref{thm:maptoboundary}(3).
\end{proof}

\subsection{The Betsy Ross matroid}
\label{subsection: Betsy-Ross}
The Betsy Ross matroid $B_{11}$ is the rank $3$ matroid on $11$ elements whose point-line arrangement is illustrated in \Cref{fig: Betsy-Ross}. The goal of this section is to prove the following theorem.

\begin{thm}\label{thm:betsyross}
    We have $\ulineGr_{B_{11}}(\T_q)=\ulineGr^{\rm w}_{B_{11}}(\T_q)$ for all $0\leq q\leq\infty$. Moreover, there is an explicit homeomorphism $\ulineGr_{B_{11}}(\T_\infty)\to\R$ which maps $\ulineGr_{B_{11}}(\T_q)$ to the closed interval $[-q,q]$ for all $0\leq q<\infty$. It identifies the space $\ulineL_{B_{11}}$ with the closed interval $[-2,2]$ and $\ulineS_{B_{11}}$ with its boundary points $\{-2,2\}$.
\end{thm}

By \cite{Baker-Lorscheid-Zhang24}*{Appendix A.3.5}, the foundation of $B_{11}$ is the \emph{golden ratio partial field $\G$}, considered as a partial field or as a pasture.\footnote{As a partial field, $\G$ is the pair $(\G^\times,R)$ where $R=\Z[\varphi]$ is the subring of $\R$ that is generated by the golden ratio $\varphi=\frac{1+\sqrt{5}}{2}$ and $\G^\times=R^\times$ is the unit group of $R$.} Partial fields and, more generally, pastures can be realized in different ways as tracts, and a suitable interplay of two such realizations $\G_1$ and $\G_2$ allows us to determine both $\ulineGr^w_{B_{11}}(\T_q)$ and $\ulineGr_{B_{11}}(\T_q)$. 

The tract $\G_1$ is the foundation of the matroid $B_{11}$ in the sense of \Cref{rem:foundation} and satisfies $\ulineGr^w_{B_{11}}(F)=\Hom(\G_1,F)$ for every tract $F$. Its unit group $\G_1^\times=\G^\times$ is the multiplicative subgroup of $\R^\times$ generated by $-1$ and the golden ratio $b=\frac{1+\sqrt{5}}{2}$. The null set of $\G_1$ is the ideal of $\N[\G_1^\times]$ generated by all formal sums of three elements $a_1,a_2,a_3\in\G=\G^\times\cup\{0\}$ that sum to zero in $\R$.

The tract $\G_2$ has the same unit group $\G_2^\times=\G^\times$ as $\G_1$ and the null set of $\G_2$ consists of all formal sums of elements that sum to zero in $\R$, without any restriction on their lengths. Since $N_{\G_1}\subseteq N_{\G_2}$, the identity map $\G_1\to\G_2$ is a tract morphism, which defines a weak rescaling class $\rho$ in $\ulineGr^w_{B_{11}}(\G_2)$. By \cite{Baker-Zhang23}*{Section 1.5}, $\G_2$ is perfect and thus $\ulineGr^w_{B_{11}}(\G_2)=\ulineGr_{B_{11}}(\G_2)$, which shows that $\rho$ is in fact a strong rescaling class. Consequently, composing $\rho$ with a tract homomorphism $\psi\colon\G_2\to F$ yields a strong rescaling class of $\ulineGr_{B_{11}}(F)$ for every tract $F$.

\begin{rem}
 A concrete matrix representation of  $\rho$ is given by 
 \begin{equation*}
   A= \left(
 \begin{array}{ccccccccccc}
  0 & 0 & 1 & 1 & 1 & 1 & 1 & 1 & 1 & 1 & 1\\
  1 & 0 & 1 & \varphi+1 & \varphi & \varphi+1 & 0 & \varphi & \varphi & \varphi+1&0 \\
  0 & 1 & 1 & \varphi & \frac{1}{\varphi} & 0 & \varphi & \varphi & 1 & 1&0 \\
 \end{array}
   \right),
 \end{equation*}
  in the sense that it is represented by the map $U_{3,11}\to \G$ that sends $e_i+e_j+e_k$ to the minor of the $3\times3$-submatrix of $A$ with columns $i<j<k$. This matrix was computed using an implementation of the algorithms developed in \cite{Chen-Zhang}.
\end{rem}

Now we determine $\ulineGr^w_{B_{11}}(\T_\infty)=\ulineGr_{B_{11}}(\T_\infty)$. Clearly, for all $t\in\R$ the map 
\begin{equation*}
    \psi_t\colon\G^\times\longrightarrow \R_{>0},\qquad x\longmapsto |x|^t
\end{equation*}
is a group homomorphism. Because the abelian group $\G^\times$ has rank one, every group homomorphism $\G^\times\to\R_{>0}$ is of this form. By \Cref{rem:tinftyreps}, this implies that 
\begin{equation}\label{eq:betsyrossline}
\R\longrightarrow\ulineGr_{B_{11}}(\T_\infty),\qquad t\longmapsto\rho_t\coloneq\psi_t\circ\rho
\end{equation}
is a bijection; in fact, it is a homeomorphism. Moreover, for every $0\leq q\leq\infty$, the map $\rho_t$ is in $\ulineGr^{\rm w}_{B_{11}}(\T_q)$ if and only if $\psi_t$ defines a tract homomorphism $\G_1\to\T_q$. Finally, the map $\rho_t$ is in $\ulineGr_{B_{11}}(\T_q)$ if $\psi_t$ defines a tract homomorphism $\G_2\to\T_q$. Hence the first part of \Cref{thm:betsyross} follows from the following lemma.

\begin{lemma}The following holds for all $q\geq0$:
\begin{enumerate}[(1)]\itemsep 5pt
    \item If $t\in[-q, q]$, then $\psi_t$ is a homomorphism of tracts $\G_2\to\T_q$. 
    \item If $t\not\in [-q,q]$, then $\psi_t$ is not a homomorphism of tracts $\G_1\to\T_q$. 
\end{enumerate}
\end{lemma}

\begin{proof}
    The unique nontrivial field automorphism of $\Q(b)$ sends $b$ to $-b^{-1}$ and thus restricts to an automorphism $\tau$ of both $\G_1$ and $\G_2$. For all $t\in\R$, we have $\psi_{-t}=\psi_t\circ\tau$. Hence it suffices to prove (1) and (2) for $t\geq 0$. By definition of $\G_2$, part (1) follows from part (1) of \Cref{lemma: morphisms from c to tq}. For part (2), it suffices by part (2) of \Cref{lemma: morphisms from c to tq} to find $a_1,a_2,a_3\in\G^\times$ with $a_1+a_2+a_3=0$. We can choose, for example, $a_1=b^2$, $a_2=-b$, and $a_3=-1$.
\end{proof}

The remaining parts of \Cref{thm:betsyross} are covered by the following proposition.

\begin{prop}
    Let $t\in\R$ and $f_t\in\R[x_1,\ldots,x_{11}]$ be such that $\rho_{f_t}=\rho_t$. Then $f_t$ is Lorentzian if and only if $t\in[-2,2]$ and $f_t$ is stable if and only if $t\in\{-2,2\}$.
\end{prop}

\begin{proof}
    \Cref{thm:maptoboundary} implies that $f_2$ is stable. Applying the unique nontrivial field automorphism of $\Q(b)$, as in the previous proof, we conclude that $f_{-2}$ is stable as well. Part (1) of \Cref{thm:maptoboundary} implies, moreover, that $f_2$ and $f_{-2}$ are boundary points of $\upL_{B_{11}}$. Hence \Cref{thm:314325} implies the claim on Lorentzian polynomials.

    For the claim on stable polynomials, recall that by \cite{Branden07}*{Theorem 5.6} the polynomial $f_t$ is stable only if the polynomial
    \begin{equation*}
        W_t=\frac{\partial f_t}{\partial x_1}\frac{\partial f_t}{\partial x_{11}}-f_t\cdot \frac{\partial^2f_t}{\partial x_1\partial x_{11}}
    \end{equation*}
    is nonnegative on $\R^{11}$. We define the polynomial $g_{t}$ to be the restriction of $W_t$ given by:
    \begin{equation*}
        x_1=x_2=x_3=x_5=x_7=x_{11}=0,\, x_4=1+b^t,\, x_6=b^{-2t}+b^{-t},\, x_8=x,\, x_9=b^t,\, x_{10}=-b^t.
    \end{equation*}
    One can calculate that
    \begin{equation*}
        g_t(x)=-\frac{\left(b^t+x+1\right) \left(b^{4t} x-b^{3t} x-4 b^{2t} x-b^{2t}-2 b^t x-b^t+x\right)}{b^t}.
    \end{equation*}    
    We note that $g_t(x)$ is nonnegative on $\R$ for every value of $t\geq0$ for which $f_t$ is stable. Because $g_t(x)$ has degree two, it is nonnegative if and only if its discriminant is non-positive. The discriminant is given by
    \begin{equation*}
        \frac{\left(b^t+1\right)^6 \left(b^{2t}-3b^{t}+1\right)^2}{b^{2t}}.
    \end{equation*}
    It is non-positive if and only if it is zero, and one checks that the only real values of $t$ for which this expression is zero are $t=2$ and $t=-2$. This shows the remaining part of the claim.
\end{proof}

\section{Compactifications and Euler characteristics}
\label{section:Compactifications and Euler characteristic}

Let $J\subseteq\Delta_n^d$ be an M-convex set.
\Cref{thm:Alor} shows that $\P\upL_J$ has a compactification that is homeomorphic to a closed Euclidean ball. On the other hand, one can consider the closure $\overline{\P\upL}_J$ of $\P\upL_J$ inside the projective space of homogeneous polynomials of degree $r$ in $n$ variables. 
It was conjectured in \cite{Branden-Huh20}*{Conjecture 2.29} that $\overline{\P\upL}_J$ is homeomorphic to a closed Euclidean ball in the case that $J=\Delta_n^d$. This was confirmed in \cite{Branden21}*{Theorem 3.3} for a larger class of polymatroids, also including all uniform matroids. It was asked in \cite{Branden21}*{Question 5.1} whether this is true for arbitrary polymatroids, or at least for matroids \cite{Branden21}*{Question 5.3}. The analogous question was also asked for spaces of stable polynomials, i.e., whether $\overline{\P\upS}_J$ is homeomorphic to a closed Euclidean ball for every polymatroid $J$ \cite{Branden21}*{page 7}.
The following summarizes our findings in connection with these questions:

\begin{thm}\label{thm:whenball}
    Let $J$ be an M-convex set.
    \begin{enumerate}[(1)]\itemsep 5pt
        \item If $\ulineGr_J(\T_\infty)$ is a singleton, then $\overline{\P\upL}_J$ is homeomorphic to a closed Euclidean ball. In this case, either $\overline{\P\upS}_J=\overline{\P\upL}_J$ or $\overline{\P\upS}_J=\emptyset$.
        \item If $J$ is \emph{rigid}, i.e., $\ulineGr_J(\T_0)$ is a singleton, then the Euler characteristic of $\overline{\P\upL}_J$ is equal to one.
        \item There is a rigid matroid $M$ with $\dim(\ulineGr_M(\T_\infty))=1$ such that $\overline{\P\upS}_M$ is nonempty with  Euler characteristic not equal to one. In particular, the set $\overline{\P\upS}_M$ is nonempty but not homeomorphic to a closed Euclidean ball.
        \item There is a matroid $M$ with $\dim(\ulineGr_M(\T_0))=1$ such that the Euler characteristic of $\overline{\P\upL}_M$ is not equal to one. In particular, the set $\overline{\P\upL}_M$ is not homeomorphic to a closed Euclidean ball.
    \end{enumerate}
\end{thm}

Part (1) is shown in  \Cref{prop:singletonthenball} and part (2) in \Cref{prop:rigideuler}. Parts (3) and (4) are \Cref{ex:T11counterexample} and \Cref{ex:betsyrosstable}, respectively.

\begin{rem}
It follows from \Cref{ex:finitefoundation} that the space $\ulineGr_J(\T_\infty)$ is a singleton if and only if the foundation $F_J$ of $J$ (see \Cref{rem:foundation}) is finite.
\end{rem}

\begin{rem}
    The condition that $\ulineGr_J(\T_\infty)$ is a singleton is satisfied, for example, when $J$ is a binary matroid or a projective geometry of projective dimension $d\geq2$ over a finite field \cite{Branden-deLeon10}*{Propositions 3.5 and 3.6}. If $M$ is a binary matroid, then $\overline{\P\upS}_M=\overline{\P\upL}_M$ if and only if $M$ is a regular matroid. This follows from the previous theorem in combination with \cite{Branden-deLeon10}*{Theorem 1.4}. 
\end{rem}

\subsection{Initial polymatroids and base polytopes}
For every $r,n\in\N$, the space of nonzero Lorentzian polynomials of degree $r$ in $n$ variables has a natural stratification:
\begin{equation*}
    \bigsqcup_{J\subseteq\Delta_n^d\textrm{ M-convex}} {\upL}_J .
\end{equation*}
We will need to understand when the closure of one such cell intersects another cell.

\begin{df}\label{df:initialsubset}
    Let $J\subseteq\Delta_n^d$ be M-convex and $\emptyset\neq J'\subseteq J$ a subset. We say that $J'$ is an \emph{initial subset} of $J$ if there exists an M-convex function $v\colon J\to\R$ which takes only nonnegative values and is zero exactly on $J'$. A polymatroid is an \emph{initial polymatroid} of another polymatroid if its corresponding M-convex set is an initial subset of the M-convex set of the other polymatroid.
\end{df}

\begin{rem}\label{rem:initialmconvex}
    It follows directly from the definition of M-convex sets and M-convex functions that every initial subset of an M-convex set is M-convex itself.
\end{rem}
\begin{ex}
 If a matroid $M'$ is a \emph{relaxation} of another matroid $M$ in the sense of \cite{OxleyMatroidTheory}*{Theorem 1.5.14}, then $M$ is an initial matroid of $M'$.
\end{ex}

\begin{lemma}\label{lemma:wheninitial}
    Let $\emptyset\neq J,J'\subseteq\Delta_n^d$ be M-convex sets. The following are equivalent:
    \begin{enumerate}[(1)]\itemsep 5pt
        \item $J'$ is an initial subset of $J$.
        \item $f_{J'}\in\overline{\upL}_J$.
        \item $\overline{\upL}_J\cap\upL_{J'}\neq\emptyset$.
    \end{enumerate}
    Here $f_{J'}= \sum_{\alpha\in J'} \frac{1}{\alpha!}x^\alpha$ is the is the (exponential) generating function of $J'$ as defined in \Cref{subsec:conclusiontopo}.
\end{lemma}
\begin{proof}
  It follows from \Cref{thm:314325} that (1) implies (2). It is clear that (2) implies (3). Now assume (3). The curve selection lemma from semi-algebraic geometry \cite{Basu-Pollack-Roy}*{Theorem 3.19} implies that there are algebraic Puiseux series $c_\alpha(t)$ for $\alpha\in J$ such that the polynomial
  \begin{equation*}
    \sum_{\alpha\in J} \frac{c_\alpha(t)}{\alpha!}x^\alpha
  \end{equation*}
  is in $\upL_J$ for  for sufficiently small $t>0$, and in $\upL_{J'}$ for $t=0$. This means that the function
  \begin{equation*}
        v\colon J\to\R,\qquad \alpha\mapsto{\rm val}(c_\alpha(t))
  \end{equation*}
  takes only nonnegative values and is zero exactly on $J'$. From \cite{Branden-Huh20}*{Theorem 3.20}, it follows that $v$ is M-convex, which proves (1).
\end{proof}
\begin{rem}
 Spaces of stable polynomials are less well-behaved with respect to initial polymatroids. For example, the non-Fano matroid $F_7^-$ is a relaxation of the Fano matroid $F_7$. By \cite{Choe-Oxley-Sokal-Wagner04}*{Example 11.5}, we have $\upS_{F_7^-}\neq\emptyset$, but $\upS_{F_7}=\emptyset$ by \cite{Branden07}. An example of a relaxation $M'$ of a matroid $M$ such that $\upS_{M'}=\emptyset$ but $\upS_{M}\neq\emptyset$ is given in \cite{Kummer-Sawall}*{Example 4.10}.
\end{rem}

We now take a more careful look at the set $\overline{\upL}_J\cap\upL_{J'}$ in the case where it is nonempty.
\begin{lemma}\label{lemma:cellstarshaped}
    Let $J'$ be an initial subset of $J$. Then the set $\log(\overline{\upL}_J\cap\upL_{J'})\subseteq\R^{J'}$ is star-shaped with respect to the origin.
\end{lemma}

\begin{proof}
    Let $f=\sum_{\alpha\in J'}\frac{c_\alpha}{\alpha!}x^\alpha\in \upL_{J'}$, and let $(f_i)_{i\in\N}\subseteq\upL_J$ be a sequence converging to $f$. Writing $f_i=\sum_{\alpha\in J}\frac{c_{i,\alpha}}{\alpha!}x^\alpha\in \upL_{J}$, it follows from \Cref{thm:314325} that for all $0\leq p\leq 1$ and all $i\in\N$, the polynomial $f_{i,p}=\sum_{\alpha\in J}\frac{c_{i,\alpha}^px^\alpha}{\alpha!}$ is Lorentzian as well. For fixed $0< p\leq 1$, the limit of the sequence $(f_{i,p})_{i\in\N}\subseteq\upL_J$ is the polynomial $\sum_{\alpha\in J'}\frac{c^p_\alpha}{\alpha!}x^\alpha\in \upL_{J'}$. This shows the claim. 
\end{proof}

For understanding the topology of a star-shaped set, we need to understand the rays it contains.

\begin{lemma}\label{lemma:rayinclosextends}
    Let $J'$ be an initial subset of $J$.
     If $\nu'\in\R^{J'}$ is such that $-\lambda\cdot\nu'\in\log(\overline{\upL}_J\cap\upL_{J'})$ for all $\lambda\geq0$, then for every $N\in\N$, there is an M-convex function $\nu\colon J\to\R$ which extends $\nu'$ and satisfies $\nu(\beta)\geq N$ for all $\beta\in J\smallsetminus J'$.
\end{lemma}

\begin{proof}
    Let $\nu'\in\R^{J'}$ be such that $-\lambda\cdot\nu'\in\log(\overline{\upL}_J\cap\upL_{J'})$ for all $\lambda\geq0$. This implies that the polynomial with coefficients in the field of Puiseux series
    \begin{equation*}
        g=\sum_{\alpha\in J'}\frac{t^{\nu'(\alpha)}}{\alpha!}x^\alpha\in \upL_{J'}
    \end{equation*}
    is Lorentzian and belongs to the closure of the set of Lorentzian polynomials with support $J$ over the field of Puiseux series. Now let $\epsilon>0$, and let $q\in\N$ be larger than $N$ and every $\nu'(\alpha)$ for $\alpha\in J'$. We consider the intersection of the ball of radius $t^q$ around $g$ in the max-norm on the coefficients with the set of Lorentzian polynomials over the Puiseux series with support $J$. By assumption, this intersection is nonempty. Therefore, there exist Puiseux series $c_\alpha(t)$ for $\alpha\in J$ such that ${\rm val}(c_\alpha(t))=\nu'(\alpha)$ if $\alpha\in J'$ and ${\rm val}(c_\alpha(t))\geq q$ otherwise, and such that $\sum_{\alpha\in J}\frac{c_{\alpha}(t)x^\alpha}{\alpha!}$ is Lorentzian. This implies that the function $\nu\colon J\to \R$ defined by $\nu(\alpha)={\rm val}(c_\alpha(t))$ has the desired properties.
\end{proof}

Before we continue, we note a useful consequence of \Cref{lemma:rayinclosextends}. For $d,n\in\N$ and a tract $F$, we define the \emph{Polygrassmannian} $\PolyGr(d,n)(F)$ to be the projectivization of the set all functions $\rho\colon\Delta_n^d\to F$ whose support is M-convex and that are a strong $F$-representation of their support. If $d\leq n$, then we define $\Gr(d,n)(F)$ to be the subset of $\PolyGr(d,n)(F)$ consisting of projective equivalence classes of functions whose support is contained in  $[0,1]^n$. If $F=\T_q$ for some $0\leq q\leq\infty$, then these spaces inherit a topology from the Euclidean topology on $\R_{\geq0}^{\Delta^d_n}$, see \cite{BHKL0}*{Section~10.1} for equivalent definitions and further properties.
We demonstrate that the connection between Lorentzian polynomials and M-convex functions can be used to recover some basic properties of Dressians \cite{Branden-Huh20}*{Lemma 3.27}.

\begin{cor}\label{cor:closofmax}
    For every $d,n\in\N$, the space $\Gr(d,n)(\T_0)$ is the closure of its maximal cell $\Gr_{U_{d,n}}(\T_0)$. The same is true for polymatroids, namely $\PolyGr(d,n)(\T_0)$ is the closure of its maximal cell $\Gr_{\Delta^d_n}(\T_0)$.
\end{cor}

\begin{proof}
    Let $\rho\colon\Delta^d_n\to\R_{\geq0}$ be a $\T_0$-representation of some polymatroid, and let $J$ be the support of $\rho$. Since $J$ is M-convex, the generating polynomial $f_J$ is Lorentzian. By definition, it is a limit of Lorentzian polynomials with full support $\Delta_n^d$. By \Cref{lemma:wheninitial}, $J$ is an initial subset of $\Delta^d_n$. By \Cref{lem: diraremconvex}, the function $-(\log\circ\rho)\colon J\to\R$ satisfies the hypotheses of \Cref{lemma:rayinclosextends}. Thus, for every $N\in\N$, there exists an M-concave function $\nu_N\colon\Delta_n^d\to\R$ which extends $\log\circ\rho$ and satisfies $\nu_N(\alpha)\leq-N$ for all $\alpha\in\Delta_n^d\smallsetminus J$. This shows that $\rho_N=\exp\circ\nu_N$ is in $\upR_{\Delta_n^d}(\T_0)$  and satisfies
    \begin{equation*}
        \rho_N(\alpha)=\begin{cases}
            \rho_N(\alpha)=\rho(\alpha)\textrm{ if }\alpha\in J,\\
            \rho_N(\alpha)\leq e^{-N}\textrm{ otherwise.}
        \end{cases}
    \end{equation*}
    It follows that $\rho=\lim_{N\to\infty}\rho_N$, which yields the claim for $\PolyGr(d,n)(\T_0)$. The claim for $\Gr(d,n)(\T_0)$ can be shown in the same way using the result from \cite{Branden21} that every multiaffine Lorentzian polynomial is a limit of multiaffine Lorentzian polynomials with full support.
\end{proof}

Initial polymatroids of a given polymatroid can be understood in terms of regular polymatroid subdivisions. We assume some familiarity with the theory of regular subdivisions.  As references we recommend \cite{TriangulationsBook}*{Section 5} or \cite{Joswig21}*{Section 1.2}.

\begin{df}
    Let $\emptyset\neq J\subseteq\Delta_n^d$ be an M-convex set. The \emph{base polytope} $\BP_J$ of $J$ is the convex hull of $J$ in $\R^n$. A \emph{regular polymatroid subdivision} of $J$ (or of $\BP_J$) is a regular subdivision of $J$ that is induced by an M-convex function $\nu\colon J\to\R$ (i.e., the subdivision equals the subset of $\BP_J$ where $\nu$ is not differentiable).
\end{df}

\begin{rem}\label{rem:polymatroidsubdivisions}
Let $\nu\colon J\to\R$ be an M-convex function.
    \begin{enumerate}[(1)]\itemsep 5pt
        \item  Let $\ell\colon\R^n\to\R$ be a linear function  and $a,b>0$. Then the M-convex functions $\nu$ and $a\cdot\nu+b\cdot\ell|_J$ induce the same regular polymatroid subdivision, see e.g. \cite{TriangulationsBook}*{Proposition 5.4.1}.
        \item Let $J'\subseteq J$ be the set of vertices of a cell\footnote{By a \emph{cell} of a regular polymatroid subdivision we mean one of the (not necessarily maximal) polytopes that comprise the regular polymatroid subdivision (which is a compact polyhedral complex).} in the regular polymatroid subdivision induced by $\nu$. Then there is a linear function $\ell\colon\R^n\to\R$ such that $\nu+\ell|_J$ takes only nonnegative values and is zero exactly on $J'$. This shows that the initial subsets of $J$ are exactly those subsets of $J$ that correspond to a cell in a regular polymatroid subdivision of $J$.  In particular, it follows  from \Cref{rem:initialmconvex} that subsets of $J$ which correspond to a cell in a regular polymatroid subdivision of $J$ are themselves M-convex.
        \item Let $\rho\colon J\to\T_0$ be a $\T_0$-representation of $J$ and let $t=(t_1,\ldots,t_n)\in(\R_{>0})^n$. By \Cref{lemma:mconvexist0}, the function $\nu=-\log(\rho)\colon J\to\R$ is M-convex. If $\ell\colon\R^n\to\R$ is the linear function that maps the $i$-th unit vector to $-\log(t_i)$, then $$-\log(t.\rho)=\nu+\ell|_J,$$ where $t.\rho$ denotes the action of $\rho$ defined in \Cref{eq:rescalingaction}. Thus $-\log(\rho)$ and $-\log(t.\rho)$ induce the same regular polymatroid subdivision (by part (1) of this remark), and we can speak of the regular polymatroid subdivision induced by an element of $\ulineGr_J(\T_0)$. By \Cref{lemma:mconvexist0}, every regular polymatroid subdivision of $J$ is induced by an element of $\ulineGr_J(\T_0)$ in this sense.
    \end{enumerate}
\end{rem}
The easiest special case where the intersection $\overline{\upL}_J\cap\upL_{J'}$ is always nonempty is when the base polytope of the initial subset $J'$ corresponds to a face of the base polytope of $J$.
\begin{lemma}\label{lemma:closoftorusorbit}
    Let $\emptyset\neq J'\subseteq J\subseteq\Delta_n^d$ be M-convex sets, and let $f\in\upL_J$. We consider the orbit
    \begin{equation*}
        \upO_f=\{f(\epsilon_1x_1,\ldots,\epsilon_nx_n)\mid \epsilon_1,\ldots,\epsilon_n>0\}\subseteq\upL_J.
    \end{equation*}
    The following are equivalent:
    \begin{enumerate}[(1)]\itemsep 5pt
        \item $\BP_{J'}$ is a face of $\BP_J$.
        \item The closure of $\upO_f$ intersects $\upL_{J'}$.
    \end{enumerate}
    In this case, the intersection of the closure of $\upO_f$ with $\upL_{J'}$ is
    \begin{equation*}
        \upO_g=\{g(\epsilon_1x_1,\ldots,\epsilon_nx_n)\mid \epsilon_1,\ldots,\epsilon_n>0\},
    \end{equation*}
    where $g$ is the sum over all terms in $f$ that correspond to points in $J'$.
    
    In particular, $\overline{\upL}_J\cap\upL_{J'}$ is nonempty whenever $\BP_{J'}$ is a face of $\BP_J$
\end{lemma}

\begin{proof}
    Assuming (1), there exists $\beta\in\R^n$ such that the linear form $\langle-,\beta\rangle$ takes its minimum $\gamma$ on $\BP_J$ exactly on $\BP_{J'}$. Then we have
    \begin{equation}\label{eq:facelimit}
        \lim_{\epsilon\to0} f(\epsilon^{\beta_1-\gamma/r}x_1,\ldots,\epsilon^{\beta_n-\gamma/r}x_n)=g\in\upL_{J'},
    \end{equation}
    which proves (2). By replacing each $x_i$ by $\epsilon_ix_i$ for some $\epsilon_i>0$ in \Cref{eq:facelimit}, we further see that the intersection of the closure of $\upO_f$ with $\upL_{J'}$ contains $\upO_g$. 
    
    Now assume (2), and let $h\in\upL_{J'}$ be in the closure of $\upO_f$. By the curve selection lemma  \cite{Basu-Pollack-Roy}*{Theorem 3.19} and the existence of semi-algebraic sections \cite{ScheidererBook}*{Proposition 4.5.9},     there are algebraic Puiseux series $c_i(t)$ for $i\in[n]$, which are positive for small enough $t>0$, such that 
    \begin{equation*}
        h=\lim_{t\to0}f(c_1(t)x_1,\ldots c_n(t)x_n).
    \end{equation*}
    Letting $\beta=({\rm val}(c_1(t)),\ldots,{\rm val}(c_n(t)))\in\R^n$, this implies that the linear form $\langle-,\beta\rangle$ takes its minimum on $\BP_J$ exactly on $\BP_{J'}$, where it is zero. This shows (1). Furthermore,
    \begin{equation*}
        h=\lim_{t\to0}f(c_1(t)x_1,\ldots c_n(t)x_n)=g(c_1^{\rm in}x_1,\ldots,c_n^{\rm in}x_n),
    \end{equation*}
    where $c_i^{\rm in}\in\R$ is the coefficient of smallest degree in $c_i(t)$. This shows that $h\in \upO_g$.
\end{proof}

\begin{cor}
    Let $\emptyset\neq J'\subseteq J\subseteq\Delta_n^d$ be M-convex sets such that $\BP_{J'}$ is a face of $\BP_J$. For $\nu'\in\R^{J'}$, the following are equivalent:
    \begin{enumerate}[(1)]\itemsep 5pt
        \item $-\lambda\cdot\nu'\in\log(\overline{\upL}_J\cap\upL_{J'})$ for all $\lambda\geq0$.
        \item $\nu'$ extends to an M-convex function $\nu\colon J\to\R$.
    \end{enumerate}
\end{cor}

\begin{proof}
 It follows from \Cref{lemma:rayinclosextends} that (1) implies (2). Assume (2) holds, and let $\nu\colon J\to\R$ be an M-convex function that extends $\nu'$. For every $q>0$, it follows from \Cref{thm:314325} that the polynomial $f=\sum_{\alpha\in J}\frac{q^{\nu(\alpha)}}{\alpha!}x^\alpha$ is Lorentzian. Hence, by \Cref{lemma:closoftorusorbit}, the polynomial 
 $g=\sum_{\alpha\in J'}\frac{q^{\nu'(\alpha)}}{\alpha!}x^\alpha$
 is in the closure of $\upL_J$. This establishes (1).
\end{proof}

If the reduced Dressian $\ulineDr_J=-\log(\ulineGr_J(\T_0))$ of $J$ has only finitely many rays, we have a good understanding of the rays in $\log(\overline{\upL}_J\cap\upL_{J'})$.

\begin{thm}\label{thm:eulerforrays}
    Let $J\subseteq\Delta_n^d$ be an M-convex set such that $\log(\ulineGr_J(\T_0))$ has only finitely many rays. Let $\emptyset\neq J'\subseteq J$ be an initial subset, and let $W_{J'}$ be the linear subspace of $\R^{J'}$ corresponding to the image under the orbit of the generating polynomial under the rescaling action of $(\R^\times)^n$. Finally, let $0\neq\nu'\in\R^{J'}$. Then $-\nu'$ generates a ray of $\log(\overline{\upL}_J\cap\upL_{J'})$ if and only if one of the following holds:
    \begin{enumerate}[(1)]\itemsep 5pt
        \item $\BP_{J'}$ is a face of $\BP_J$ and $\nu'$ is the restriction of an M-convex function $\nu\colon J\to\R$.
        \item $\BP_{J'}$ is not a face of $\BP_J$ and $\nu'\in W_{J'}$.
    \end{enumerate}
\end{thm}

\begin{proof}
    The case where $\BP_{J'}$ is a face of $\BP_J$ was settled in the preceding corollary. Thus, we may assume that $\BP_{J'}$ is not a face of $\BP_J$. If $\nu'\in W_{J'}$, then it follows from \Cref{lemma:closoftorusorbit} that $-\nu'$ generates a ray of $\log(\overline{\upL}_J\cap\upL_{J'})$. Conversely, assume that $-\nu'$ generates a ray of $\log(\overline{\upL}_J\cap\upL_{J'})$. By \Cref{lemma:rayinclosextends}, there exists an M-convex function $\nu_0\colon J\to\R$ extending $\nu'$. For every $i\in\N$, we recursively define $\nu_i\colon J\to\R$ to be an M-convex extension of $\nu'$ such that $\nu_i(\beta)>\max_{\gamma\in J}(v_{i-1}(\gamma))$ for all $\beta\in J\smallsetminus J'$. This is possible by \Cref{lemma:rayinclosextends}. Because $\log(\ulineGr_J(\T_0))$ has only finitely many rays, there are $i>j$, $\lambda>0$, and $\omega\in W_J$ such that
    \begin{equation*}
        \lambda\cdot \nu_i= \nu_j + \omega.
    \end{equation*}
    This show that the restriction of $\omega$ to $J'$, which lies in $W_{J'}$, is equal to $(\lambda-1)\cdot\nu'$. Thus, it remains to show that $\lambda\neq1$. Assume for the sake of a contradiction that $\lambda=1$, i.e., that $\omega=\nu_i-\nu_j$. By construction, $\nu_i(\beta)>\nu_j(\beta)$ for every $\beta\in J\smallsetminus J'$. This shows that $\omega$ is nonnegative on $J$ and is zero exactly on $J'$. Because $\omega\in W_{J}$, this implies that $\BP_{J'}$ is a face of $\BP_J$, contradicting our previous assumption.
\end{proof}

\subsection{Compactly supported Euler characteristic}
Building upon the results from the previous subsection, we will compute the Euler characteristic of $\overline{\P\upL}_M$ for some matroids $M$. In particular, we will see an example of a matroid $M$ for which $\chi(\overline{\P\upL}_M)\neq1$, which shows in particular that $\overline{\P\upL}_M$ is not homeomorphic to a ball.

To this end, for any M-convex $J$, we consider the decomposition
\begin{equation*}
    \overline{\P\upL}_J=\bigsqcup_{J'} (\overline{\P\upL}_J \cap \P\upL_{J'}),
\end{equation*}
where the disjoint union is over all initial subsets $J'$ of $J$. 

For computing $\chi(\overline{\P\upL}_J)$ in terms of the cells $\overline{\P\upL}_J \cap \P\upL_{J'}$, we will make use of some concepts and results from tame topology (see \cite{vanDries98} for an introduction to this topic). 

We say that a subset $S\subseteq\R^n$ is \emph{definable} if it is of the form
\begin{equation*}
    \{x\in\R^n\mid \exists y\in\R^k: P(x,y,e^x,e^y)=0\},
\end{equation*}
where $P$ is a real polynomial in $2(n+k)$ variables, and where $x=(x_1,\ldots,x_n)$, $y=(y_1,\ldots,y_k)$, $e^x=(e^{x_1},\ldots,e^{x_n})$, $e^y=(e^{y_1},\ldots,e^{y_k})$. A map $f\colon S\to\R^m$ from a definable set $S\subseteq\R^n$ is \emph{definable} if its graph is definable. It was shown by Wilkie \cite{wilkie} that the definable sets form an \emph{o-minimal structure}. This implies in particular that:
\begin{enumerate}[(1)]\itemsep 5pt
    \item Every semi-algebraic set and every semi-algebraic map is definable.
    \item Boolean combinations of definable sets are definable.
    \item If $A$ is definable, then so are $\R\times A$ and $A\times \R$.
    \item The image of a definable set under a definable map is definable.
\end{enumerate}

\begin{ex}
Note that the map $\R^n\to\R^n,\, x\mapsto e^x$ is definable but not semi-algebraic.
\end{ex}

\begin{rem}\label{rem:quantifyer elimination}
    Let $S\subseteq\R^{n+1}$ be a definable set. Then the set
    \begin{equation*}
        \{x\in\R^n\mid \exists t\in\R: (x,t)\in S\}
    \end{equation*}
    is the image of $S$ under the definable map 
    \begin{equation*}
    \pi\colon\R^{n+1}\to\R^n,\qquad (x,t)\mapsto x,
    \end{equation*}
    and thus is again definable by property (4) of o-minimality. Similarly, the set
    \begin{equation*}
        \{x\in\R^n\mid \forall t\in\R: (x,t)\in S\}=\R^{n}\smallsetminus \pi(\R^{n+1}\smallsetminus S)
    \end{equation*}
    is definable. We say that \emph{definable sets satisfy quantifier elimination}.
\end{rem}
\begin{thm}[§4.2 in \cite{vanDries98}]\label{thm:eulerchar}
    There is a unique integer-valued map $\chi_*$ on the set of definable sets that satisfies the following conditions:
    \begin{enumerate}[(1)]\itemsep 5pt
        \item $\chi_*(\R^n)=(-1)^n$,
        \item If $S,T\subseteq\R^n$ are disjoint definable sets, then $\chi_*(S\cup T)=\chi_*(S)+\chi_*(T)$,
        \item If $S\subseteq\R^n$ is a definable set and $f\colon S\to\R^m$ is an injective definable map, then $\chi_*(S)=\chi_*(f(S))$.
        \item If $S\subseteq\R^n$ and $T\subseteq\R^n$ are definable sets, then $\chi_*(S\times T)=\chi_*(S)\cdot \chi_*(T)$.
    \end{enumerate}
\end{thm}

\begin{rem}\label{rem:topologicaleuler}
\begin{enumerate}[(1)]\itemsep 5pt
    \item[]
    \item Let $S\subseteq\R^n$ be a compact definable set. 
    Because $S$ can be triangulated by means of definable homeomorphisms, see \cite{vanDries98}*{§8.2}, it follows that $\chi_*(S)$ is equal to the topological Euler characteristic $\chi(S)$ of $S$, defined as the alternating sum
    \begin{equation*}
        \chi(S)=b_0-b_1+b_2-b_3+\cdots,
    \end{equation*}
    where $b_i$ is the rank of the $i$th singular homology group of $S$.
    \item If $S\subseteq\R^n$ is definable but not compact, then $\chi_*(S)$ might differ from $\chi(S)$. For example $\chi(\R^n)=1$ because $\R^n$ is contractible, but $\chi_*(\R^n)=(-1)^n$ by part (1) of \Cref{thm:eulerchar}.
\end{enumerate}
\end{rem}

Applying these considerations to our situation, we obtain the following formula.

\begin{lemma}\label{eq:PLJeulerchar}
    Let $J$ be an M-convex set. Then 
    \begin{equation*} 
    \chi(\overline{\P\upL}_J)=\sum_{J'} \chi_*(\overline{\P\upL}_J \cap \P\upL_{J'})=-\sum_{J'} \chi_*(\overline{\upL}_J \cap \upL_{J'})=-\sum_{J'} \chi_*(\log(\overline{\upL}_J \cap \upL_{J'})),
\end{equation*}
where the sums are over all initial subsets $J'$ of $J$. 
\end{lemma}
\begin{proof}
    The first equality follows from part (2) of \Cref{thm:eulerchar}. For proving the second equality, we first note that $\overline{\P\upL}_J \cap \P\upL_{J'}$ can be identified with the subset of $\overline{\upL}_J \cap \upL_{J'}$ consisting of all polynomials for which the sum of the coefficients is equal to one. Under this identification we obtain the homeomorphism
    \begin{equation*}
        \left(\overline{\P\upL}_J \cap \P\upL_{J'}\right)\times\R_{>0}\to\overline{\upL}_J \cap \upL_{J'},\qquad (f,\lambda)\mapsto \lambda\cdot f.
    \end{equation*}
    Because $\chi_*(\R_{>0})=\chi_*(\R)=-1$ by part (1) of \Cref{thm:eulerchar}, the second equality of the claim now follows from part (4) of \Cref{thm:eulerchar}. Finally, the third equality follows from part (3) of \Cref{thm:eulerchar}.
\end{proof}

The sets $\log(\overline{\upL}_J \cap \upL_{J'})$ are star-shaped by \Cref{lemma:cellstarshaped}. Thus we can use \Cref{eq:PLJeulerchar} and the following lemma to compute the Euler characteristic of $\chi(\overline{\P\upL}_J)$.

\begin{lemma}\label{lemma:eulerstar}
    Let $X\subseteq\R^n$ be a closed definable set that is star-shaped with respect to the origin. Let  $ S\subseteq\R^n$ be the unit sphere. Consider the set of rays contained in $X$:
    \begin{equation*}
        S(X)=\{\nu\in  S\mid \forall t\geq0: t\cdot\nu\in X\}.
    \end{equation*}
    Then $S(X)$ is a definable set satisfying $\chi_*(S(X))+\chi_*(X)=1$.
\end{lemma}

\begin{proof}
    By \Cref{rem:quantifyer elimination}, definable sets satisfy quantifier elimination, and thus    the set $S(X)$ is definable. For the same reason, the following map is definable:
    \begin{equation*}
        \psi\colon X\to\R_{\geq0},\qquad x\mapsto\inf\{t>0\mid \frac{x}{t}\in X\}.
    \end{equation*}
    Furthermore, by \Cref{lemma: cont}, we have for all $x\in X$:
    \begin{enumerate}[(1)]\itemsep 5pt
        \item $\psi(x)\leq1$.
        \item If $\psi(x)>0$, then $\frac{x}{\psi(x)}\in X$.
        \item $\psi(x)=0$ if and only if $x=0$ or $x\neq0$ and $\frac{x}{\|x\|}\in S(X)$.
    \end{enumerate}
    The proof of \Cref{lemma: cont} also shows that $\psi$ is lower semi-continuous, as this part of the proof only uses that $X$ is closed and star-shaped.
    
    Let $\upB\subseteq\R^n$ be the closed unit ball, and consider the map
    \begin{equation*}
        \phi\colon X\to \upB\smallsetminus S(X),\qquad x\mapsto\frac{x}{1-\psi(x)+\|x\|}.
    \end{equation*}
    The map is well-defined because for all $x\in X$ we have $0\leq\psi(x)\leq1$ and $x\neq0$ whenever $\psi(x)=1$. Its norm is lower semi-continuous, because $\psi$ is. For fixed $x\in X$, the map
    \begin{equation*}
        \phi_x\colon[0,1]\to \R x,\qquad \lambda\mapsto\phi(\lambda\cdot x)
    \end{equation*}
    is continuous, its norm is monotonically increasing, and it satisfies $\phi_x(0)=0$ and $\phi_x(1)=x$. This shows that the image $\phi(X)$ is star-shaped with respect to the origin. Furthermore, $\phi$ is definable, and a short calculation shows that it is injective. Thus $\chi_*(X)=\chi_*(\phi(X))$. 
    
    We claim that the closure $Y$ of $\phi(X)$ is the disjoint union of $\phi(X)$ and $S(X)$. A point $x\in X$ is mapped by $\phi$ to $ S$ if and only if $\psi(x)=1$. Thus, by (3), we cannot have $\phi(x)=\frac{x}{\|x\|}\in S(X)$. This shows that $\phi(X)$ and $S(X)$ are disjoint. Now let $x\in S(X)$, meaning that $\|x\|=1$ and $t\cdot x\in X$ for all $t\geq0$. We have
    \begin{equation*}
        \lim_{t\to\infty}\phi(tx)=\lim_{t\to\infty}\frac{tx}{1+{t}}=x.
    \end{equation*}
    This shows that $S(X)$ is contained in the closure of $\phi(X)$. 
    
    Conversely, let $(x_i)_{i\in\N}$ be a sequence in $X$ such that $(\phi(x_i))_{i\in\N}$ converges. If it is unbounded, then $(\phi(x_i))_{i\in\N}$ converges to a point in $S(X)$. Otherwise, it converges to a point in $\phi(X)$ by semi-continuity and star-shapedness. Thus we have shown that the closure $Y$ of $\phi(X)$ is the disjoint union of $\phi(X)$ and $S(X)$. As $Y$ is a compact star-shaped definable set, we have
    \begin{equation*}
        1=\chi_*(Y)=\chi_*(\phi(X))+\chi_*(S(X))=\chi_*(X)+\chi_*(S(X)).\qedhere
    \end{equation*}
\end{proof}

\begin{cor}
    For every M-convex set $\emptyset\neq J\subseteq\Delta_n^d$, we have $\chi_*(\P\upL_J)=\chi_*(\Dr_J)$.
\end{cor}

\begin{proof}
    It suffices to prove that $\chi_*(\log(\upL_J))=\chi_*(\log(\upR_J(\T_0)))$. 
    By \Cref{lemma:mconvexist0} and \Cref{lem: diraremconvex}, the star-shaped definable sets $\log(\upL_J)$ and $\log(\upR_J(\T_0))$ contain the same rays. The claim therefore follows from the preceding lemma.
\end{proof}

\begin{rem}\label{rem:eulerdressian}
 The preceding corollary, together with \Cref{thm:petterballness},
imply that for all $d,n$ we have
  $\sum_{J} \chi_*(\Dr_J)=1$,
  where the sum is over all nonempty M-convex sets $J\subseteq\Delta_n^d$ or over all matroids $J$ of rank $d$ on $[n]$.
  We do not know how to prove this combinatorial statement without the detour to Lorentzian polynomials.
\end{rem}

\begin{prop}\label{prop:singletonthenball}
    If $\emptyset\neq J\subseteq\Delta_n^d$ is an M-convex set such that $\ulineGr_J(\T_\infty)$ is a singleton, then $\overline{\P\upL}_J$ is homeomorphic to the base polytope $\BP_J$. In particular, it is homeomorphic to a closed ball.
\end{prop}

\begin{proof}
    The assumptions imply that $\overline{\P\upL}_J$ is the orbit closure of the generating polynomial under the action of $\R_{>0}^n$ by rescaling the variables. Thus the claim follows directly from the proposition in \cite{fulton}*{§4.2}.
\end{proof}

If $\ulineGr_J(\T_0)$ is a singleton, but $\ulineGr_J(\T_\infty)$ is not, we are not able to conclude that $\overline{\P\upL}_J$ is homeomorphic to a closed ball, but we can at least show that it has the same Euler characteristic as a closed ball:

\begin{prop}\label{prop:rigideuler}
    Let $\emptyset\neq J\subseteq\Delta_n^d$ be a rigid M-convex set, that is, the space $\ulineGr_J(\T_0)$ is a singleton. Then $\chi(\overline{\P\upL}_J)=1$.
\end{prop}

\begin{proof}
    We use the notation from \Cref{thm:eulerforrays}. Because $J$ is rigid, it follows that $\log(\upR_J(\T_0))=W_J$. Thus, for every initial subset $J'$ of $J$, the union of rays contained in $\log(\overline{\upL}_J\cap\upL_{J'})$ is exactly the set $W_{J'}$. \Cref{lemma:eulerstar} now implies that
    \begin{equation*}
        \chi_*(\log(\overline{\P\upL}_J\cap\P\upL_{J'}))=-\chi_*(\log(\overline{\upL}_J\cap\upL_{J'}))=-\chi_*(W_{J'})=(-1)^{\dim(W_{J'})-1}.
    \end{equation*}
    This shows that $\chi(\overline{\P\upL}_J)=\sum_{J'} (-1)^{\dim(W_{J'})-1}$, where the sum is over all initial subsets $J'$ of $J$. Because $J$ is rigid, every M-convex function $J\to\R$ is the restriction of a linear function $\R^n\to\R$, see \Cref{rem:polymatroidsubdivisions}. This implies that the initial subsets $J'$ of $J$ are exactly those nonempty subsets of $J$ for which $\BP_{J'}$ is a face of $\BP_J$.
    Moreover, the dimension of $\BP_{J'}$ is $\dim(W_{J'})-1$. It follows that $\chi(\overline{\P\upL}_J)$ is the Euler characteristic of $\BP_J$, which implies the claim.
\end{proof}

The following theorem will allow us to compute the Euler characteristic $\chi(\overline{\P\upL}_J)$ in the case that the polyhedral fan $\Dr_J=-\log \ulineGr_J(\T_0)$ consists of finitely many rays. Here and in the following, we refer to the \emph{restriction map} $\res_{J,J'}:\log \ulineGr_J(\T_0)\to\log \ulineGr_{J'}(\T_0)$ that is given by $\rho\mapsto\rho|_{J'}$.

\begin{thm}\label{thm:eulercomp}
    Let $\emptyset\neq J\subseteq\Delta_n^d$ be an M-convex set such that $\log\ulineGr_J(\T_0)$ consists of $m$ rays. For $i=0,\ldots,n$, let $g_i$ be the number of initial subsets $J'$ of $J$ such that $\BP_{J'}$ has dimension $i$ and is not a face of $\BP_{J}$. In addition, for $i=0,\ldots,n$ and $j=0,\ldots,m$, let $f_{ij}$ be the number of initial subsets $J'$ of $J$ such that $\BP_{J'}$ has dimension $i$, is a face of $\BP_{J}$, and such that the image of $\res_{J,J'}:\log \ulineGr_J(\T_0)\to\log \ulineGr_{J'}(\T_0)$ consists of $j$ rays. Then
    \begin{equation*}
        \chi(\overline{\P\upL}_J)=\sum_{i=0}^n(-1)^i\cdot\left( g_i +\sum_{j=0}^mf_{ij}\cdot\left( 1-j \right)\right).
    \end{equation*}
\end{thm}

\begin{proof}
    We use the formula
    \begin{equation*}
    \chi(\overline{\P\upL}_J)=-\sum_{J'} \chi_*(\log(\overline{\upL}_J \cap \upL_{J'})),
\end{equation*}
where the sum is over all nonempty initial subsets of $J$. 
By \Cref{lemma:eulerstar} and \Cref{thm:eulerforrays}, we find, as in the proof of \Cref{prop:rigideuler}, that 
\begin{equation*}
    -\chi_*(\log(\overline{\upL}_J \cap \upL_{J'}))=(-1)^{\dim(W_{J'})-1}=(-1)^{\dim \BP_{J'}}
\end{equation*}
in the case that $\BP_{J'}$ is a face of $\BP_{J}$ and
\begin{equation*}
    -\chi_*(\log(\overline{\upL}_J \cap \upL_{J'}))=-((-1)^{\dim(W_{J'})}+j\cdot(-1)^{\dim(W_{J'})+1})=(-1)^{\dim \BP_{J'}}(1-j)
\end{equation*}
otherwise, where $j$ is the number of rays of the image of $\log \ulineGr_J(\T_0)\to\log \ulineGr_{J'}(\T_0)$.
This implies the claim.
\end{proof}
\begin{rem}\label{rem:howtocomputegiandfij}
    Given an M-convex set $J$, one can compute the polyhedral fan $\log \ulineGr_J(\T_0)$, for example, using the software \texttt{gfan} \cite{gfan}. In the case that $J$ is a matroid, this is described in detail in \cite{Brandt-Speyer22}*{Algorithm 1}\footnote{With minor adaptions, this algorithm can also treat the case of a general M-convex set $J$ but in this paper we will perform such computations only in the case that $J$ is a matroid.}.
    If $\log \ulineGr_J(\T_0)$ consists only of finitely many rays, then one can compute the numbers $g_i$ and $f_{ij}$ from \Cref{thm:eulercomp} as follows. We first compute all initial subsets of $J$. To this end, we have to compute all faces of the polytope $\BP_J$ and all faces of the regular subdivisions of $\BP_J$ defined by a generator of each ray of $\log \ulineGr_J(\T_0)$. This can be done, for example, using the software \texttt{polymake} \cite{polymake}, where the command \texttt{SubdivisionOfPoints} allows to compute the regular subdivisions. Finally, for every initial subset $J'$, we record the dimension of $\BP_{J'}$ and, if $\BP_{J'}$ is a face of $\BP_{J}$, then we also record the number of rays of the image of $\res_{J,J'}:\log \ulineGr_J(\T_0)\to\log \ulineGr_{J'}(\T_0)$. 
\end{rem}

\begin{ex}
    Let $M=U_{2,4}$ be the uniform matroid of rank 2 on 4 elements, whose base polytope is an octahedron. It has three non-trivial regular matroid subdivisions corresponding to the three rays of $\log \ulineGr_J(\T_0)$. 
    None of these subdivisions subdivides any of the proper faces. It follows that $f_{00}=6$, $f_{10}=12$, $f_{20}=8$, and $f_{33}=1$. All other $f_{ij}$ are zero. Each of the three non-trivial subdivisions consists of two pyramids that intersect at a quadrilateral, hence each contains two additional faces of dimension three and one of dimension two. Thus $g_0=g_1=0$, $g_2=3$, and $g_3=6$. By \Cref{thm:eulercomp}, we obtain
    \begin{equation*}
        \chi(\overline{\P\upL}_M)=6-12+(3+8)-(6+1\cdot(1-3))=1.
    \end{equation*}
\end{ex}

\begin{df}
    For every $n\geq4$, the \emph{elliptic matroid} $\cT_n$ is defined to be the matroid of rank $3$ on the ground set $[n]$ for which a three element subset $\{i,j,k\}\subseteq[n]$ is a non-basis if and only if $i+j+k\equiv0 \pmod n$.
\end{df}

\begin{rem}
 The matroid $\cT_n$ can be realized in the projective plane as the subgroup of points on an elliptic curve generated by a torsion point of order $n$, hence the name. The complex representation spaces of $\cT_n$ have been studied in \cites{Kuhne-Roulleau24,Borisov-Roulleau24}.
\end{rem}

Computer experiments suggest that $\log\ulineGr_{\cT_p}(\T_0)$ is one-dimensional, consisting of precisely $\frac{p-1}{2}$ rays, whenever $p \geq 7$ is a prime number. We compute the Euler characteristic for some small prime numbers.

\begin{ex}
    The matroid $M=\cT_5$ has the two non-bases $\{0,1,4\}$ and $\{0,2,3\}$. This matroid is binary and therefore rigid. Thus $\chi(\overline{\P\upL}_M)=1$ by \Cref{prop:rigideuler}.
\end{ex}

\begin{ex}
    The matroid $M=\cT_7$ has the five non-bases $\{0,1,6\}$, $\{0,2,5\}$, $\{0,3,4\}$, $\{1,2,4\}$, and $\{3,5,6\}$. It coincides with the matroid $P_7$ from \cite{OxleyMatroidTheory}*{page 644}. The fan $\log\ulineGr_{M}(\T_0)$ consists of three rays, and we calculate that
    \begin{equation*}
        (f_{ij})_{i,j}=
        \begin{pmatrix}
            30&0&0&0\\
            150&0&0&0\\
            281&0&0&0\\
            222&0&0&24\\
            68&0&0&36\\
            6&0&0&13\\
            0&0&0&1
        \end{pmatrix}.
    \end{equation*}
    The vector $(g_i)_{i}$ is equal to $(0,0,72,267,294,111,12)$. From \Cref{thm:eulercomp}, we can now calculate that $\chi(\overline{\P\upL}_M)=1$.
\end{ex}

\begin{figure}
 \includegraphics{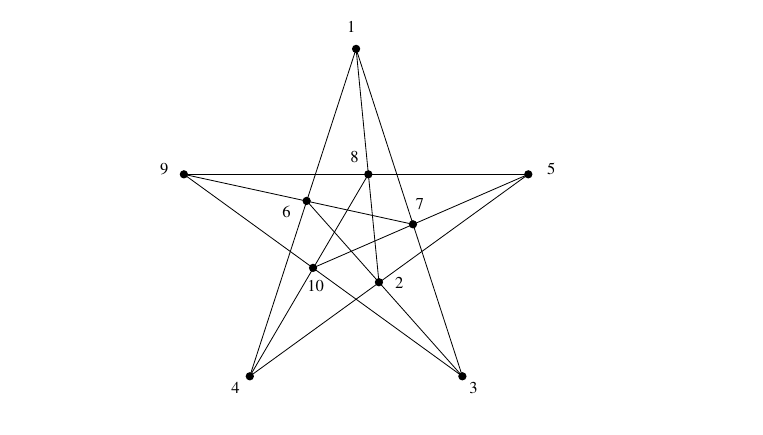}
        \caption{The deletion of $\cT_{11}$ by $11$ is a configuration of ten points and ten lines in the plane such that each point is contained in
three lines and each line contains three points. Its Dressian was studied in \cite{Brandt-Speyer22}.}
        \label{fig:10star}
 \end{figure}

\begin{ex} \label{ex:T11counterexample}
    Finally, we consider the matroid $M=\cT_{11}$. The fan $\log\ulineGr_{M}(\T_0)$ consists of five rays, and we calculate that
    \begin{equation*}
        (f_{ij})_{i,j}=
        \begin{pmatrix}
            150&0&0&0&0&0\\
            1620&0&0&0&0&0\\
            5510&0&0&0&0&0\\
            8680&0&0&1120&0&0\\
            8260&0&0&1960&0&784\\
            5720&0&0&975&0&1747\\
            2820&0&0&205&0&1445\\
            915&0&0&15&0&645\\
            175&0&0&0&0&165\\
            15&0&0&0&0&22\\
            0&0&0&0&0&1\\
        \end{pmatrix}.
    \end{equation*}
    The vector $(g_i)_{i}$ is equal to
    \begin{equation*}
\left(\begin{array}{ccccccccccc}
0& 0& 3360& 21590& 51250& 63330& 47005& 22050& 6405& 1040& 70
\end{array}\right).
    \end{equation*}
 From \Cref{thm:eulercomp}, we can now calculate that $\chi(\overline{\P\upL}_M)=11$. In particular, the space $\overline{\P\upL}_M$ is not homeomorphic to a ball. This answers \cite{Branden21}*{Questions 5.1 and 5.3} in the negative.
 \end{ex}

We believe that the spaces $\overline{\P\upL}_M$ and $\ulineGr_M(\T_0)$ for $M=\cT_n$ deserve further attention. We formulate a conjecture and a question in this regard.

\begin{conj}\label{conj: reduced Dressian of elliptic matroids}
Let $n\geq7$ be an integer. 
    \begin{enumerate}[(1)]\itemsep 5pt
    \item The reduced Dressian $\ulineDr_{\cT_n}$ is one-dimensional if and only if $n$ is prime.  
    \item When $n=p$ is prime, the one-dimensional fan $\ulineDr_{\cT_p}$ consists of precisely $\frac{p-1}{2}$ rays.
    \end{enumerate}
\end{conj}

Griffin Edwards has verified the conjecture computationally for all $n$ with $7 \leq n \leq 26$. 

\begin{question}\label{question: Euler characteristic of PL_M for elliptic matroids}
    Let $p\geq11$ be a prime number and $M=\cT_p$. What is the Euler characteristic $\chi(\overline{\P\upL}_M)$? In particular, is it equal to $p$?
\end{question}

Note that $p=7$ and $p=11$ are the first prime numbers satisfying the assumptions of \cite{Borisov-Roulleau24}*{Theorem 1} and \cite{Borisov-Roulleau24}*{Theorem 2}, respectively. This might explain the failure of \Cref{conj: reduced Dressian of elliptic matroids} and \Cref{question: Euler characteristic of PL_M for elliptic matroids} for small primes $p$.
 
In \cite{Branden21}*{page 7}, it was also asked whether the closure of $\P\upS_J$ is homeomorphic to a closed Euclidean ball for every M-convex set $J$. We give a negative answer to this question as well:

\begin{ex}\label{ex:betsyrosstable}
    Let $M=B_{11}$ be the Betsy Ross matroid from \Cref{subsection: Betsy-Ross}. Because $M$ is rigid, every initial matroid corresponds to a face of the base polytope. 
    The $f$-vector of the base polytope of $M$ is
    \begin{equation*}
        \left(
\begin{array}{ccccccccccc}
 140 & 1410 & 5010 & 9355 & 10774 & 8257 & 4295 & 1470 & 305 & 32 & 1 \\
\end{array}
\right).
    \end{equation*}
    We have seen in \Cref{subsection: Betsy-Ross} that $\P\upS_M$ consists of two orbits under rescaling of the variables. By \Cref{lemma:closoftorusorbit}, the closure of each orbit is the disjoint union of cells corresponding to faces $F$ of the base polytope that are, under the $\log$ map, homeomorphic to a real vector space of dimension $\dim(F)$. For each face $F$, these cells are themselves orbits. Thus the two cells are either disjoint or equal to one other. If $M_F$ is the initial matroid corresponding to the face $F$, then the two cells are equal if and only if $\res_{M,M_F}:\ulineGr_M(\T_\infty)\to\ulineGr_{M_F}(\T_\infty)$ is not injective. The following table records the number of faces in each dimension for which this happens:
    \begin{equation*}
\left(
\begin{array}{ccccccccccc}
 140 & 1410 & 5010 & 8705 & 8770 & 5775 & 2570 & 715 & 105 & 5 & 0 \\
\end{array}
\right).
    \end{equation*}
    Therefore, we can compute the Euler characteristic $\chi(\overline{\P\upS}_M)$ of $\overline{\P\upS}_M$ as
    \begin{equation*}
        2-(140 - 1410 + 5010 - 8705 + 8770 - 5775 + 2570 - 715 + 105 - 5 + 0)=17.
    \end{equation*}
    In particular, this shows that the space $\overline{\P\upS}_M$ is not homeomorphic to a ball.
\end{ex}

\subsection{Compactifications of orbit spaces}
The next goal is to provide a natural compactification of the orbit spaces $\ulineL_J$ and $\ulineGr_J(\T_q)$. This is done via the concept of Hausdorff quotients, as introduced in \cite{hausdorffmoduli}. We first recall their construction.
\subsubsection{The Hausdorff quotient}
Let $X$ be a compact metrizable space.
Recall that for any choice of metric on $X$, the set $H(X)$ of compact subsets of $X$ equipped with the Hausdorff metric is again a compact metric space. The topology of $H(X)$ does not depend on the metric on $X$. A continuous map $f\colon X\to Y$ of compact metrizable spaces induces a continuous map $H(f)\colon H(X)\to H(Y)$. We will make frequent use the following easy lemma:
\begin{lemma}\label{lem:limitdescription}
    Let $(K_i)_{i\in\N}$ be a sequence of compact subsets $K_i$ of $X$ that converges to the compact set $K\subseteq X$ in the Hausdorff metric. For each $i\in\N$, let $U_i$ be a dense subset of $K_i$. Then $K$ is the set of all $x\in X$ for which there is a sequence $(x_i)_{i\in\N}$ with $x_i\in U_i$ that converges to $x$.
\end{lemma}

Given a topological group $G$ acting continuously on $X$, we recall from \cite{hausdorffmoduli} the construction of the {Hausdorff quotient} $X/_H G$. We consider the map
\begin{equation*}
    e\colon X\longrightarrow H(X),\qquad x\longmapsto\overline{Gx}.
\end{equation*}
Let $\mathbb{U}$ be the set of open, dense, and $G$-invariant subsets of $X$. Then the \emph{Hausdorff quotient} of $X$ by $G$, denoted $X/_H G$, is defined as
\begin{equation*}
    X/_HG=\bigcap_{U\in\mathbb{U}}\overline{e(U)}.
\end{equation*}
Because $X/_HG$ is a closed subspace of $H(X)$, it is compact and metrizable.
\begin{df}
    A set $U\in\mathbb{U}$ is \emph{stable} if $e|_U\colon U\to H(X)$ is continuous. A set $U\in\mathbb{U}$ is \emph{semi-stable} if $X/_H G=\overline{e(U)}$.
\end{df}
\begin{rem}
    Being stable implies being semi-stable \cite{hausdorffmoduli}*{Remark 1.2}.
\end{rem}
In certain situations, the Hausdorff limit interacts nicely with the Hilbert quotient of a projective variety.
\begin{thm}[\cite{hausdorffmoduli}*{Theorem 7}]
    Let $X$ be a complex projective variety equipped with the action of a group scheme $G$. Assume that every point of the Hilbert quotient $X/\!\!/\!\!/ G$ is a reduced closed subscheme of $X$. Then, with respect to the analytic topology on $X/\!\!/\!\!/ G$, there is a natural homeomorphism
    \begin{equation*}
        X/\!\!/\!\!/ G \longrightarrow X/_H G
    \end{equation*}
    which maps a reduced closed subscheme to its underlying set. 
\end{thm}

As a corollary, we find that the Hausdorff quotient of the complex Grassmannian with respect to the action of the algebraic torus is the underlying topological space of the Chow quotient:
\begin{cor}\label{cor:chowhausdorff}
    Consider the action of $G=(\C^\times)^n$ on the complex Grassmannian $\Gr(d,n)(\C)$. There is a natural homeomorphism
    \begin{equation*}
        \Gr(d,n)(\C)/\!\!/ G\longrightarrow \Gr(d,n)(\C)/_H G
    \end{equation*}
    from the Chow quotient to the Hausdorff quotient which sends a cycle to its underlying set.
\end{cor}
\begin{proof}
    Let us denote $X=\Gr(d,n)(\C)$.
    By \cite{chowquotientsI}*{Lemma 1.5.3}, every point in $X/\!\!/\!\!/ G$ is a reduced closed subscheme of $X$. 
    Thus the preceding theorem gives a natural homeomorphism $X/\!\!/\!\!/ G \to X/_H G$. Furthermore, by \cite{chowquotientsI}*{Theorem 1.5.2}, the natural morphism $X/\!\!/\!\!/ G \to X/\!\!/ G$ is an isomorphism.
\end{proof}

\subsubsection{Compactifications of \texorpdfstring{$\ulineL_J$}{ulineLJ}  and \texorpdfstring{$\ulineGr_J(\T_q)$}{ulineGrJ(T2)}}
Let $J$ be a polymatroid on $[n]$ of rank $r$, and let $U$ denote either $\P\upL_J$ or $\Gr_J(\T_q)$ for some arbitrary but fixed $q>0$. 
We consider $U$ as a subset of $\P\R_{\geq0}[x_1,\ldots,x_n]$, which we identify with the set of polynomials in $\R[x_1,\ldots,x_n]_r$ whose coefficients sum to one. On the latter space, we consider the Euclidean metric $d$ and we let $X$ denote the closure of $U$.
The group $G=\R_{+}^n$ acts on $X$ by rescaling. 
For $J'\subseteq J$ another M-convex set, we write $U_{J'}=\P\upL_{J'}$ or $U_{J'}=\Gr_{J'}(\T_q)$, respectively.

\begin{rem}\label{rem:orbitclosures}
    For $p\in X$, we can explicitly describe the orbit closure of $p$. Namely, for every face $F$ of the base polytope of $J$, we denote by $p_F$ the polynomial obtained from $p$ by removing all terms that do not correspond to elements of $F$. Then $\overline{Gp}$ is the union of $Gp_F$ over all faces $F$.   
\end{rem}
\begin{lemma}
    Let $K\subseteq U$ be a compact subset.
    For every $\epsilon>0$, there exists $\delta>0$ such that for all $p_1,p_2\in K$ with $d(p_1,p_2)<\delta$ and all $q_1$ in the orbit closure of $p_1$, there is a point $q_2$ in the orbit closure of $p_2$ with $d(q_1,q_2)<\epsilon$.
\end{lemma}
\begin{proof}
    We denote by $\overline{G}$ the orbit closure in $\P\R_{\geq0}[x_1,\ldots,x_n]$ of the polynomial with support $J$ and all coefficients equal to each other. We consider the map
    \begin{equation*}
        \psi\colon \overline{G}\times K\longrightarrow X
    \end{equation*}
    defined by coefficient-wise multiplication and then dividing by the sum of the coefficients. By \Cref{rem:orbitclosures}, for every $p\in K$ the image $\psi(\overline{G}\times\{p\})$ is exactly the orbit closure of $p$. By the Heine--Cantor theorem, the map $\psi$ is uniformly continuous. This means that for every $\epsilon>0$, there exists $\delta>0$ such that for all $(g_1,p_1),(g_2,p_2)\in\overline{G}\times K$ with $d(g_1,g_2)+d(p_1,p_2)<\delta$, we have $d(\psi(g_1,p_1),\psi(g_2,p_2))<\epsilon$. Given  $p_1,p_2\in K$ with $d(p_1,p_2)<\delta$ and a point $q_1$ in the orbit closure of $p_1$, it follows that $q_1=\psi(g,p_1)$ for some $g\in\overline{G}$ and $d(q_1,\psi(g,p_2))<\epsilon$, which proves the claim.
\end{proof}
\begin{cor}\label{cor:stable}
    The open subset $U$ of $X$ is stable.
\end{cor}
\begin{proof}
    This follows from the previous lemma, together with the fact that $U$ is locally compact.
\end{proof}
\begin{lemma}\label{lem:open}
    The map $e\colon X\to H(X),\, x\mapsto\overline{Gx}$ restricted to $U$ is an open mapping onto its image.
\end{lemma}
\begin{proof}
    Let $V\subseteq U$ be an open subset, and let $x\in V$. We have to show that $e(x)=\overline{Gx}$ has an open neighborhood in $e(U)$ which is contained in the image of $V$ under $e$. Let $\epsilon>0$ be such that the open ball $B$ of radius $\epsilon$ around $x$ is contained in $V$. We claim that the open ball of radius $\epsilon$ around $\overline{Gx}$ is contained in the image of $V$. Indeed, if the orbit closure $\overline{Gy}$ of $y\in U$ has Hausdorff distance smaller than $\epsilon$ to $\overline{Gx}$, then in particular it intersects $B$ nontrivially. Because orbits of $G$ are closed in $U$, there exists $y'\in Gy\cap B$. This shows that $e(y)=e(y')$ is in $e(B)$.
\end{proof}
\begin{cor}\label{cor:compactification}
    The Hausdorff quotient $X/_HG$ is a compactification of $U/G$.
\end{cor}
\begin{proof}
    By \Cref{cor:stable} and \Cref{lem:open}, the map that sends an orbit $x\in U/G$ to the point in $H(X)$ that corresponds to its closure is a homeomorphism onto its image $V$. By stability of $U$, the closure of $V$ is $X/_HG$.
\end{proof}
We will denote the compactification $X/_HG$ by $\HC(\ulineL_J)$ resp. $\HC(\ulineGr_J(\T_q))$ when $U=\P\upL_J$ and $U=\Gr_J(\T_q)$, respectively.
 Our compactification is compatible with the Chow quotient of complex Grassmannians:
\begin{thm}\label{thm:chowhausdorffmap}
    Let $J=U_{d,n}$ and $q>0$.
    \begin{enumerate}[(1)]\itemsep 5pt
        \item   The map
    \begin{equation*}
        \Gr(d,n)(\C) /\!\!/ (\C^\times)^n \longrightarrow \HC(\ulineGr_J(\T_q)),
    \end{equation*}
    which sends a cycle in the Chow quotient to the image of its underlying set under the coordinate-wise map $z\mapsto|z|^q$, is continuous.
    \item Similarly, the map
    \begin{equation*}
        \Gr(d,n)(\C) /\!\!/ (\C^\times)^n \longrightarrow \HC({\ulineL}(d,n)_\sqfree),
    \end{equation*}
    which sends a cycle in the Chow quotient to the image of its underlying set under the coordinate-wise absolute square map, is continuous.
    \end{enumerate}
\end{thm}
\begin{proof}
    This follows from \Cref{cor:chowhausdorff} and the fact that the $q$th power of the coordinate-wise absolute value map sends every $(\C^\times)^n$-orbit closure to an $\R_{>0}^n$-orbit closure.
\end{proof}
Next, we study the boundary points of $X/_HG$. Recall that $U$ denotes either $\P\upL_J$ or $\Gr_J(\T_q)$ and that $X$ is the closure of $U$ considered as a subset of $\P\R_{\geq0}[x_1,\ldots,x_n]$. Then, by definition, the elements of $X/_HG$ are compact subsets of $X$ and the set of orbit closures of elements of $U$ is an open dense subset of $X/_HG$. We will see in the following \Cref{thm:bondarypoints} that the remaining points of $X/_HG$ are not orbit closures themselves but finite unions of such.

We let $P=\BP_J\subseteq\R^n$ be the base polytope of $J$. Let $\mathscr{P}$ be a regular polymatroid subdivision of $P$, whose maximal cells correspond to initial polymatroids $J_1,\ldots,J_s$. We denote by $Y(\mathscr{P})$ the set of all $K\in X/_H G$ for which there exist $x_i\in U_{J_i}$, for $i=1,\ldots,s$, such that $K=\cup_{i=1}^s\overline{Gx_i}$. 

\begin{thm}\label{thm:bondarypoints}The following holds true.
    \begin{enumerate}[(1)]\itemsep 5pt
        \item For $x\in X$ there exists $K\in X/_HG\subseteq H(X)$ such that $x\in K$. 
        \item We have $X/_HG=\cup_{\mathscr{P}} Y(\mathscr{P})$, where the union is over all regular polymatroid subdivisions  $\mathscr{P}$ of $P$.
        \item We have $Y(\mathscr{P})\neq\emptyset$ for every regular polymatroid subdivision  $\mathscr{P}$ of $P$. 
    \end{enumerate}
\end{thm}
\begin{proof}[Proof of \Cref{thm:bondarypoints}(1)]
    Let $y\in X$, and let $(y_i)_{i\in\N}$ be a sequence in $U$ converging to $y$. Then the sequence $(\overline{Gy_i})_{i\in\N}$ in $X/_HG$ has a subsequence converging to some point $x$. Then $x$ contains $y$ by \Cref{lem:limitdescription}.
\end{proof}
\begin{lemma}
    Let $K\in X/_HG$ and let $J'\subseteq J$ be an M-convex set such that $K\cap U_{J'}\neq\emptyset$. Then $K\cap U_{J'}$ is a $G$-orbit.
\end{lemma}
\begin{proof}
    Let $(y_i)_{i\in\N}$ be a sequence in $U$ such that the corresponding sequence of orbit closures converges to $K$ in the Hausdorff metric. By \Cref{lem:limitdescription}, we can assume without loss of generality that $(y_i)_{i\in\N}$ converges to a point $y\in K\cap U_{J'}$. We will prove that $K\cap U_{J'}=Gy$. If $g\in G$, then $(gy_i)_{i\in\N}$ converges to $gy$. Thus, by \Cref{lem:limitdescription}, we see that $Gy\subseteq K\cap U_{J'}$. Conversely, let $z\in K\cap U_{J'}$. By \Cref{lem:limitdescription}, there is a sequence $(g_i)_{i\in\N}$ in $G$ such that $(g_iy_i)_{i\in\N}$ converges to $z$. Then the sequence $(g_iy)_{i\in\N}$ in $Gy$ also converges to $z$. Since $Gy$ is a closed subset of $U_{J'}$, it follows that $z\in Gy$.
\end{proof}
In light of \Cref{rem:orbitclosures} and the preceding lemma, in order to prove part (2) of \Cref{thm:bondarypoints}, it remains to show that for all $K\in X/_HG$, the set of $J'\subseteq J$ such that $K\cap U_{J'}\neq\emptyset$ corresponds exactly to the cells in a regular polymatroid subdivision of $P$. To this end, we will associate to a sequence $(y_i)_{i\in\N}$ in $U$ an M-convex function on $J$ with values in a certain real closed field $R$. M-convex functions with values in a real closed field $R$ other than $\R$ are defined in the same way as over $\R$, and it follows either from Tarski's principle, which says that a first-order sentence in the language of ordered fields holds in a given real closed field if and only if it holds in $\R$ \cite{Prestel84}*{Section 5}, or in the same way as over $\R$ that parts (1) and (2) of \Cref{rem:polymatroidsubdivisions} remain valid in this more general setup. Similarly, Tarski's principle implies that every  regular polymatroid subdivision of $P$ that is induced by an M-convex function with values in $R$ can also be realized by an M-convex function with values in $\R$. In the following, we will introduce the real closed field that we will be using.

\subsubsection{Interlude on the \texorpdfstring{$\aleph_1$}{aleph1}-saturation}\label{sssec:aleph1}
We first recall some basic properties of the so-called $\aleph_1$-saturation $\R^*$ of $\R$, see \cite{presteldelzell}*{§2.2}. Let $\cF$ be a non-principal ultrafilter on $\N$, i.e., an ultrafilter containing the filter of cofinite subsets of $\N$, and let $\R^*$ be the ultrapower $\R^{\N}/\cF$ of $\R$ with respect to $\cF$. This is the set of equivalence classes of sequences in $\R$, where two sequences $(x_i)_{i\in\N}$ and $(y_i)_{i\in\N}$ are defined to be equivalent if the set of indices $i\in\N$ for which $x_i=y_i$ is in $\cF$.
By \cite{presteldelzell}*{Theorem 2.2.7 and 2.2.8}, the ultrapower $\R^*$ is a real closed field. Here sum and product are defined as the equivalence class of the component-wise sum and product of two representing sequences, respectively. 
Moreover, one can consider $\R$ as a subfield of $\R^*$ via the map $\R\to\R^*$ that sends $a$ to the class of the constant sequence $(a)_{i\in\N}$. The total order on $\R^*$ is given by 
\begin{equation}\label{eq:saturationorder}
[(x_i)_{\in\N}]\geq 0 \Longleftrightarrow \{i\in\N\mid x_i\geq0\}\in\cF.
\end{equation}

Next recall, for example from \cite{ERA}, that the \emph{convex hull} $\fo$ of $\R$ in $\R^*$ is the set of all elements $x\in \R^*$ such that there are $a,b\in\R$ with $a\leq x\leq b$. It is clear that $\fo$ is a valuation ring of $\R^*$ which contains $\R$. The maximal ideal $\fm$ of $\fo$ consists of all $x\in\R^*$ such that for all $\epsilon\in\R$ with $\epsilon>0$, we have $|x|<\epsilon$. Finally, we denote by $\kappa=\fo/\fm$ the corresponding residue field, and by $\pi\colon\fo\to\kappa$ the natural homomorphism.

\begin{lemma}\label{lem:res}
 The restriction of $\pi$ to $\R$ is an isomorphism  onto $\kappa$.
\end{lemma}

\begin{proof}
 This follows from \cite{ERA}*{Proposition 2.5.3} and the fact that $\R$ has no free Dedekind cut, see \cite{ERA}*{§2.9}.
\end{proof}

In light of \Cref{lem:res}, we may identify $\kappa$ with $\R$. The \emph{value group} of $\R^*$ is the abelian group $\Gamma=(\R^*)^{\times}/\fo^{\times}$. The relation $a\leq b$ defined by $a^{-1}b\in \fo$ for $a,b\in\Gamma$ makes $\Gamma$ a totally ordered group. Since $\R^*$ is real closed, $\Gamma$ is divisible. We will use additive notation for the group $\Gamma$. 

By Hahn's embedding theorem \cite{hahn}, we can embed $\Gamma$ into the additive group  of a real closed field $R$, namely the field of Hahn series over $\R$, whose value group is the divisible closure of the group of Archimedean equivalence classes of $\Gamma$. Because $\Gamma$ is divisible, it is a $\Q$-linear subspace of $R$. For $f\in (\R^*)^\times$, we let $v(f)$ denote the residue class of $f$ in $\Gamma\subseteq R$. We further define $v(0)=\infty$. The map $v\colon \R^*\to R\cup\{\infty\}$ is a valuation on $\R^*$ that is compatible with the ordering of $\R^*$, meaning that $v(a)>v(b)$ for positive $a,b\in\R^*$ implies $a<b$.
\begin{rem}\label{rem:limits}
Let $(x_i)_{i\in\N}$ be a sequence in $\R$ and let $x^*=[(x_i)_{i\in\N}]\in\R^*$ be its corresponding class. If $(x_i)_{i\in\N}$ is convergent, then $\pi(x^*)=\lim_{i\to\infty}(x_i)_{i\in\N}$. Conversely, if $x^*\in\fo$, then $(x_i)_{i\in\N}$ has a subsequence converging to $\pi(x^*)$. Furthermore, we have $v(x^*)>0$ if and only if $\pi(x^*)=0$.
If $x\not\in\fo$, then $(x_i)_{i\in\N}$ is unbounded.
\end{rem}

In the proof of \Cref{thm:bondarypoints} we will need the following two lemmas.

\begin{lemma}\label{lemma:tqimpliesmconvex}
    Let $J\subseteq\Delta^d_n$ be an  M-convex set.
    Let $(y_m)_{m\in\N}$ be a sequence in $\upR_J(\T_q)$, for some $q>0$, and let $y^*=[(y_m)_{m\in\N}]$ be its equivalence class in $(\R^*)^J$. The map \begin{equation*}\rho\colon J\longrightarrow R,\qquad \alpha\longmapsto v(y^*(\alpha))\end{equation*} is M-convex. 
\end{lemma}
\begin{proof}
    In order to prove that $\rho$ is M-convex, we will use a characterization of M-convex functions due to Murota \cite{Murota03}*{Theorem 6.4}: A function $f\colon J\to\R$ is M-convex if and only if for all $\alpha\in\Delta^{d-2}_n$ and all $i,j,k,l\in[n]$ such that $\{i,k\}\cap\{j,l\}=\emptyset$,
    \begin{multline}\label{eq: Murota's 3-term relations}
        f(\alpha+e_i+e_k) + f(\alpha+e_j+e_l) \\ \geq \ \min\big\{ f(\alpha+e_i+e_j)+ f(\alpha+e_k+e_l), \ f(\alpha+e_i+e_l)+ f(\alpha+e_j+e_k) \big\}.
    \end{multline}
    By Tarski's principle the same characterization applies to functions with values in the real closed field $R$.

    Now consider a sequence $(y_m)_{m\in\N}$ in $\upR_J(\T_q)$ for some $q>0$. Let $\alpha\in\Delta^{d-2}_n$ and $i,j,k,l\in[n]$ such that $\{i,k\}\cap\{j,l\}=\emptyset$. Then for all $m\in\N$ we have
    \begin{multline*}
        y_m(\alpha+e_i+e_k)^{1/q} \cdot y_m(\alpha+e_j+e_l)^{1/q} \\ \leq \  y_m(\alpha+e_i+e_j)^{1/q}\cdot y_m(\alpha+e_k+e_l)^{1/q}+ \ y_m(\alpha+e_i+e_l)^{1/q}\cdot y_m(\alpha+e_j+e_k)^{1/q} .
    \end{multline*}
    By the definition of the order on $\R^*$, see \Cref{eq:saturationorder}, this implies that 
    \begin{multline*}
        y^*(\alpha+e_i+e_k)^{1/q} \cdot y^*(\alpha+e_j+e_l)^{1/q} \\ \leq \  y^*(\alpha+e_i+e_j)^{1/q}\cdot y^*(\alpha+e_k+e_l)^{1/q}+ \ y^*(\alpha+e_i+e_l)^{1/q}\cdot y^*(\alpha+e_j+e_k)^{1/q} .
    \end{multline*}
    Now taking the valuation of both sides implies that $\rho$ satisfies \Cref{eq: Murota's 3-term relations} because the valuation $v$ on $\R^*$ is compatible with the order on $\R^*$.
\end{proof}
\begin{lemma}\label{lemma:lorentzianimpliesmconvex}
    Let $J\subseteq\Delta^d_n$ be an  M-convex set.
    Let $(g_m)_{m\in\N}$ be a sequence in $\upL_J$ and consider the equivalence class $y^*=[(y_m)_{m\in\N}]\in(\R^*)^J$ where $y_m=\rho_{g_m}\in\R^J$. The map \begin{equation*}\rho\colon J\longrightarrow R,\qquad \alpha\longmapsto v(y^*(\alpha))\end{equation*} is M-convex. 
\end{lemma}
\begin{proof}
    The corresponding statement was shown in \cite{Branden-Huh20}*{Theorem 3.20} for the valued real closed field of real Puiseux series and the same proof applies to $\R^*$.

    Alternatively, the statement of the lemma follows from the previous lemma and the fact that $g_m\in\upL_J$ implies $y_m\in\upR_J(\T_2)$, for all $m\in\N$, which we will show in \cite{BHKL2}.
\end{proof}
The role that the field $\R^*$ plays in the proof of \Cref{thm:bondarypoints} can be roughly explained as follows. For $K\in X/_HG$ every point $x\in K$ is the limit of a sequence in $U$ by \Cref{lem:limitdescription}. Such a sequence can be interpreted as a point with coordinates in $\R^*$. Its valuation is M-convex by the two preceding lemmas and thus defines a regular polymatroid subdivision. 

\subsubsection{Proof of Theorem \ref{thm:bondarypoints}}
For part (2) of \Cref{thm:bondarypoints}, it remains to show that for $K\in X/_HG$, those $J'\subseteq J$ with $K\cap U_{J'}\neq\emptyset$ correspond exactly to the cells in a certain regular polymatroid subdivision of $P$.
We consider a sequence in $U$ such that the corresponding sequence of orbit closures converges to $K$ in the Hausdorff metric. Such a sequence is represented by a sequence $(y_i)_{i\in\N}$ in $\R^J$. 
Letting $y^*=[(y_i)_{i\in\N}]\in(\R^*)^J$ be its equivalence class, Lemmas \ref{lemma:tqimpliesmconvex} and \ref{lemma:lorentzianimpliesmconvex} imply that the map $\rho\colon J\to R,\, \alpha\mapsto v(y^*_\alpha)$ is M-convex. 
Let $\mathscr{P}$ be the polymatroid subdivision of $J$ induced by $\rho$.

First, we let $J'\subseteq J$ be such that there exists $z\in K\cap U_{J'}$. We will show that $J'$ defines a cell of the subdivision $\mathscr{P}$. By \Cref{lem:limitdescription}, there is a sequence $(g_i)_{i\in\N}$ in $G$ such that $(z_i)_{i\in\N}\coloneq(g_iy_i)_{i\in\N}$ converges to $z$. Let $z^*\in(\R^*)^J$ be the equivalence class of $(z_i)_{i\in\N}$. As above, the map $\rho'\colon J\to R,\, \alpha\mapsto v(z^*_\alpha)$ is M-convex and $J'$ is the set of $\alpha\in J$ where $\rho'$ attains its minimum.
If $g^*$ is the equivalence class of $(g_i)_{i\in\N}$ in $(\R^*_{>0})^n$, then for all $\alpha\in J$, we have
\begin{equation*}
 \rho'(\alpha)=\rho(\alpha)+\sum_{j=1}^nv(g^*_j)\cdot\alpha_j.
\end{equation*}
In particular, the M-convex functions $\rho$ and $\rho'$ only differ by a linear function, which shows that they induce the same subdivision of $J$. Thus, the cell defined by $J'$ is in $\mathscr{P}$.

Conversely, let $J'\subseteq J$ be the lattice points of a cell in $\mathscr{P}$. This means that there is a linear function $l\colon R^n\to R$ such that $\rho'\coloneq\rho+l|_J$ attains its minimum exactly at $J'$. We can further assume that this minimum is zero. Because $\rho$ takes its values in the $\Q$-linear subspace $\Gamma$ of $R$, we can choose $l$ to take its values on $J$ in $\Gamma$ as well. For $j\in[n]$, choose $t^*_j\in\R^*_{>0}$ with $v(t^*_j)=l(e_j)$, and let $(g_i)_{i\in\N}$ be a sequence in $G$ representing the tuple $(t^*_1,\ldots,t^*_n)\in(\R_{>0}^*)^n$. Let $(z_i)_{i\in\N}\coloneq(g_iy_i)_{i\in\N}$, and let $z^*$ be the element of $(\R^*)^J$ it represents. By construction, we have $v(z^*_\alpha)\geq0$ for all $\alpha\in J$ and $v(z^*_\alpha)=0$ if and only if $\alpha\in J'$. This means that $z^*_\alpha\in\fo$ for all $\alpha\in J$ and $\pi(z^*_\alpha)\neq0$ if and only if $\alpha\in J'$. By \Cref{rem:limits}, there is a subsequence of $(g_iy_i)_{i\in\N}$ which converges to an element with support $J'$. Such element lies in $x\cap U_{J'}$ by \Cref{lem:limitdescription}. Thus we have proven part (2) of \Cref{thm:bondarypoints}.

For part (3), let $\mathscr{P}$ be a regular polymatroid subdivision, and let $\rho\colon J\to\R$ be an M-convex function that induces $\mathscr{P}$. Let $(t_i)_{i\in\N}$ be a nullsequence with $t_i>0$ for all $i\in\N$, and define the sequence $(y_i)_{i\in\N}$ in $\Gr_J(\T_0)$ via
\begin{equation*}
 (y_i)_\alpha=t_i^{\rho(\alpha)}.
\end{equation*}
Since $X/_HG$ is compact, after passing to a subsequence if necessary, we can assume that $(\overline{Gy_i})_{i\in\N}$ converges. We claim that the limit point is in $Y(\mathscr{P})$. Indeed, letting $y^*=[(y_i)_{i\in\N}]\in(\R^*)^J$ and $t^*=[(t_i)_{i\in\N}]\in\R^*$, it follows that
\begin{equation*}
 v(y^*_\alpha)=\rho(\alpha)\cdot v(t^*)
\end{equation*}
for all $\alpha\in J$, concluding the proof of \Cref{thm:bondarypoints}.\qed

\subsubsection{Examples}\label{ssec:examplescompactification}
In some examples, we can show that $X/_HG$ is homeomorphic to a closed Euclidean ball.
\begin{ex}[Rigid polymatroids]
    If $J$ is rigid, i.e., the space $\ulineGr_J(\T_0)$ is a singleton, then $U/G$ is homeomorphic to a closed Euclidean ball whose dimension is the rank of the multiplicative group of the foundation of $J$, see \Cref{thm:ATqreduced} and \Cref{thm:Alorreduced}. Moreover, since $U/G$ is compact, we have $X/_HG=U/G$ by \Cref{cor:compactification}. 
    
    Examples of rigid matroids are binary matroids and projective geometries over finite fields, for which $X/_HG$ is a point, and the Betsy Ross matroid, for which $X/_HG$ is homeomorphic to a closed interval.
\end{ex}

The next simplest case is when $\log\ulineGr_J(\T_0)$ is one-dimensional and consists of a finite number $r$ of rays. In this case, by \Cref{thm:ATqreduced} and \Cref{thm:Alorreduced}, $U/G$ is homeomorphic to a closed Euclidean ball with $r$ points removed from its boundary --- one for each ray of $\log\ulineGr_J(\T_0)$. These rays correspond to pairwise different regular polymatroid subdivisions. If we further assume that every proper initial matroid of $M$ has a finite foundation, then \Cref{thm:bondarypoints} implies that $X/_HG$ is obtained from $U/G$ by adding $r$ points. Hence $X/_HG$ is homeomorphic to a closed Euclidean ball.

\begin{ex}\label{ex:hausdorffu24}
    Every proper initial matroid of $U_{2,4}$ is binary. Thus $X/_HG$ is homeomorphic to a two dimensional closed disc.
\end{ex}

\begin{ex}
  For the non-Fano matroid $J=F_7^-$, the space $\log\ulineGr_J(\T_0)$ consists of exactly one ray. The two maximal cells of the corresponding regular matroid subdivision are the base polytopes of the Fano matroid and another matroid which is graphic by \cite{Ferroni}*{Theorem 5.4}. In particular, both are binary. The space $X/_HG$ is homeomorphic to a closed interval, with one endpoint corresponding to the non-trivial matroid subdivision.
\end{ex}

\begin{ex}
  Let $\cT_{11}$ be the matroid from \Cref{ex:T11counterexample}. We have already seen that $\log\ulineGr_J(\T_0)$ consists of five rays. A computer calculation shows that every initial matroid of $\cT_{11}$ is binary. Therefore, the space $X/_HG$ is homeomorphic to a four-dimensional closed ball.
\end{ex}

\subsubsection{The Grothendieck--Knudson moduli space of stable rational curves}\label{ssec:stablecurves}
Recall that the Chow quotient of ${\Gr}(2,n)(\C)$ is isomorphic to the Grothendieck--Knudson moduli space $\overline{\cM}_{0,n}$ of stable rational curves with $n$ marked points \cite{chowquotientsI}*{Chapter IV}. By \Cref{thm:chowhausdorffmap}, there is a natural continuous map
\begin{equation*}
\overline{\cM}_{0,n}\longrightarrow\HC({\ulineL}(2,n)_\sqfree),
\end{equation*}
where $\HC({\ulineL}(2,n)_\sqfree)=\HC({\ulineL}_{U_{2,n}})$.
We discuss this map for $n=4$ and $n=5$. (These are the first interesting cases, as for $n<4$ source and target are a point.) We further observe that, by construction, the map is constant on orbits of the action of complex conjugation on $\overline{\cM}_{0,n}$.

The space $\overline{\cM}_{0,4}$ is the complex projective line, and by \Cref{ex:hausdorffu24} the space $\HC({\ulineL}(2,4))$ is homeomorphic to a closed disc. 
\begin{lemma}\label{lemma:stablecurvesmapisquotient1}
The map $\overline{\cM}_{0,4}\to\HC({\ulineL}(2,4)_\sqfree)$ is the quotient by the action of complex conjugation.
\end{lemma}

\begin{proof}
    Since the map is continuous and both source and target are compact Hausdorff spaces, it suffices to show that the map is surjective and its fibers are exactly the orbits under complex conjugation. By \Cref{lemma:imageofsquaremap}, the map ${\cM}_{0,4}\to{\ulineL}_{U_{2,4}}$ is surjective. This establishes the surjectivity of our map, because $\overline{\cM}_{0,4}$ and $\HC({\ulineL}(2,4)_\sqfree)$ are compactifications of ${\cM}_{0,4}$ and ${\ulineL}_{U_{2,4}}$. Fibers of points in ${\ulineL}_{U_{2,4}}$ are orbits by \Cref{lemma:orbitfibers}. The compactification $\overline{\cM}_{0,4}$ has three additional points, all of them are real, and these are mapped injectively to the three additional points in the compactification of ${\ulineL}_{U_{2,4}}$.
\end{proof}
The space $\overline{\cM}_{0,5}$ is the complex projective plane blown-up at four general points. By \Cref{lemma:imageofsquaremap}, the image of the map $\overline{\cM}_{0,5}\to\HC({\ulineL}(2,5)_\sqfree)$ is the closure of $\partial\ulineL_{U_{2,5}}$ inside $\HC({\ulineL}(2,5)_\sqfree)$, which we denote by $\partial\HC({\ulineL}(2,5)_\sqfree)$.
\begin{lemma}
The map $\overline{\cM}_{0,5}\to\partial\HC({\ulineL}(2,5)_\sqfree)$ is the quotient by the action of complex conjugation.
\end{lemma}
\begin{proof}
    As in \Cref{lemma:stablecurvesmapisquotient1}, we need to show that each fiber of the map is a single orbit under complex conjugation. 
    By the description in \cite{chowquotientsI}*{Section 1.2}, for every cycle $Z\in\Gr(2,5)(\C)/\!\!/(\C^\times)^5$ there is a regular matroid subdivision of the base polytope of $U_{2,5}$, corresponding to matroids $M_1,\ldots,M_r$, such that
    \begin{equation*}
        Z=Z_1\cup\cdots\cup Z_r,
    \end{equation*}
    where each $Z_i$ is the orbit closure of a point from $\Gr_{M_i}(\C)$. The fiber over the image of $Z$ is equal to
    \begin{equation*}
        \{\sigma^{e_1}(Z_1)\cup\cdots\cup\sigma^{e_r}(Z_r)\mid e_1,\ldots,e_r\in\{0,1\}\},
    \end{equation*}
    where $\sigma$ denotes complex conjugation. However, for every regular matroid subdivision of the base polytope of $U_{2,5}$, at most one of the $Z_i$ is not mapped to itself by $\sigma$. Thus the fiber over the image of $Z$ is just the orbit of $Z$.
\end{proof}
\begin{rem}
    The quotient map of $\overline{\cM}_{0,5}$ by complex conjugation was studied in detail in \cite{takayama-yoshida}, and it exhibits some beautiful combinatorics. For example, the authors define cell structures on $\overline{\cM}_{0,5}$ and its quotient which are, in a certain precise sense, dual to the Desargues graph and the Petersen graph, respectively.
\end{rem}

For $n \geq 6$, fibers of $\overline{\cM}_{0,n}\to\HC({\ulineL}(2,n)_\sqfree)$ can consist of more than one complex conjugate pair:

\begin{ex}
    Let $n\geq6$, let $a,b\in\C$ be distinct non-real complex numbers, and let $c_1,\ldots,c_n\in\R$ be pairwise distinct real numbers. We construct a stable curve $X_{a,b}$ with $n$ marked points that has two irreducible components. On the first component, we mark the three real points $c_1,c_2,c_3$, while the remaining $n-3$ points $c_4,\ldots,c_{n}$ are marked on the second component. Finally, the point $a$ on the first component is identified with the point $b$ on the second one. Then the four stable curves $X_{a,b}$ ,$X_{\bar{a},b}$, $X_{a,\bar{b}}$, and $X_{\bar{a},\bar{b}}$ represent four distinct points in $\overline{\cM}_{0,n}$ which are all mapped to the same point under the map $\overline{\cM}_{0,n}\to\HC({\ulineL}(2,n)_\sqfree)$.
\end{ex}

\begin{small}
 \bibliographystyle{plain}
 \bibliography{lorentzian}
\end{small}

\end{document}